\documentclass[reqno,10pt]{amsart}
\usepackage{graphicx,amsfonts,amssymb,amsmath,amsthm,url,amscd,comment}
\usepackage{color}
\usepackage{enumerate}

\let\counterwithin\relax
\usepackage{chngcntr}
\usepackage[usenames,dvipsnames]{xcolor}
\usepackage[normalem]{ulem}
\usepackage{pdfsync}
\usepackage{graphicx, amsmath, amssymb, amsfonts, amsthm, stmaryrd, tikz, amscd}
\usepackage[all]{xy}
\usetikzlibrary{matrix,arrows,decorations.pathmorphing}

\usepackage{graphicx, amsmath, amssymb, amsfonts, amsthm, stmaryrd, tikz, amscd}
\usepackage[all]{xy}
\usetikzlibrary{matrix,arrows,decorations.pathmorphing}

 \usepackage[hyperfootnotes=false, colorlinks, citecolor=	 RoyalBlue, urlcolor=blue, linkcolor=blue         ]{hyperref}

\theoremstyle{plain} 
\newtheorem{theorem}             {Theorem} 

\theoremstyle{definition}

\theoremstyle{plain} 
\newtheorem{proposition} {Proposition}

\theoremstyle{remark}

\counterwithin*{definition}{subsubsection}
\newtheorem*{definition*}  {Definition}

\newtheorem*{example*}    {Example}

\newtheorem{remark}             {Remark}
\newtheorem*{remark*}            {Remark}

\newtheoremstyle{itplain} 
    {6pt}                    
    {5pt\topsep}                    
    {\itshape}                   
    {}                           
    {\itshape}                   
    {.}                          
    {5pt plus 1pt minus 1pt}                       
    {}  

\theoremstyle{itplain} 
\newtheorem{lemma}{Lemma}

\newtheorem*{lemma*}{Lemma}

\newtheorem*{corollary*} {Corollary} 

\theoremstyle{remark} 

\newtheorem*{lemmatest*}{Lemma}

\counterwithin*{lemma}{subsubsection}
\counterwithin*{remark}{subsubsection}
\counterwithin*{corollary}{subsubsection}

\usepackage{etoolbox}
\patchcmd{\section}{\scshape}{\bfseries}{}{}
\makeatletter
\renewcommand{\@secnumfont}{\bfseries}
\makeatother

\renewcommand{\Re}{\mathrm{Re}}

\renewcommand{\geq}{\geqslant}
\renewcommand{\leq}{\leqslant}

\numberwithin{equation}{section}
\DeclareMathOperator{\sgn}{sgn}
\DeclareMathOperator{\SL}{SL}
\DeclareMathOperator{\Mp}{Mp}
\DeclareMathOperator{\GL}{GL}

\DeclareMathOperator{\GO}{GO}
\DeclareMathOperator{\PGO}{PGO}

\DeclareMathOperator{\htt}{ht}

\DeclareMathOperator{\SO}{SO}

\DeclareMathOperator{\ad}{ad}
\DeclareMathOperator{\Ad}{Ad}

\def\slLie{\mathfrak{s}\mathfrak{l}}

\DeclareMathOperator{\End}{End}
\DeclareMathOperator{\JL}{JL}

\def\eps{\varepsilon}

\def\PB{\operatorname{PB}}
\def\PGL{\operatorname{PGL}}
\def\Sob{\mathcal{C}}

\DeclareMathOperator{\Weil}{Weil}

\DeclareMathOperator{\U}{U}

\DeclareMathOperator{\trace}{trace}

\DeclareMathOperator{\bPB}{{\mathbf P}{\mathbf B}}

\DeclareMathOperator{\PSL}{PSL}

\DeclareMathOperator{\diag}{diag}
\DeclareMathOperator{\Sym}{Sym}

\DeclareMathOperator{\Lie}{Lie}

\def\O{\operatorname{O}}

\DeclareMathOperator{\Opp}{Op}

\DeclareMathOperator{\sym}{sym}

\DeclareMathOperator{\cusp}{cusp}

\DeclareMathOperator{\pr}{pr}

\DeclareMathOperator{\reg}{reg}

\DeclareMathOperator{\res}{res}

\DeclareMathOperator{\vol}{vol}
\DeclareMathOperator{\nr}{nr}

\DeclareMathOperator{\tr}{tr}

\DeclareMathOperator{\supp}{supp}

\DeclareMathOperator{\h}{h}
\def\p{{p}}
\DeclareMathOperator{\loc}{loc}
\DeclareMathOperator{\jac}{jac}

\author{Paul D. Nelson}
\address{Aarhus University, Institut for Matematik, Denmark}
\email{paul.nelson@math.au.dk}

\date{\today}
\title{Quantum variance on quaternion algebras, III}
\hypersetup{
  pdfkeywords={},
  pdfsubject={},
  pdfcreator={Emacs 24.5.1 (Org mode 8.2.10)}}
\begin{document}

\begin{abstract}
  We determine the asymptotic quantum variance of microlocal lifts of Hecke--Maass cusp forms on the arithmetic compact hyperbolic surfaces attached to maximal orders in quaternion algebras.  Our result extends those for the non-compact modular surface obtained by Luo--Sarnak--Zhao, whose method required a cusp.  The global arguments in our proof involve an analytic study of the theta correspondence, the interplay between additive and multiplicative harmonic analysis on quaternion algebras, the equidistribution of translates of elementary theta functions, and the Rallis inner product formula.  These reduce the proof to local problems involving the construction and analysis of microlocal lifts via integral operators on the group, addressed using an analytic incarnation of the method of coadjoint orbits.
\end{abstract}
\maketitle

\setcounter{tocdepth}{1} \tableofcontents

\section{Introduction}\label{sec-1}
\subsection{Overview}\label{sec-1-1}
The quantum variance problem (see, e.g.,  \cite[\S1]{nelson-variance-73-2}, \cite{MR848319}, \cite[\S15.6]{2009arXiv0911.4312Z}, \cite[\S4.1.3]{MR3204186}, \cite{ 2013arXiv1303.6972S, MR1361757,MR1465794,luo-sarnak-mass,MR2103474,MR2651907}) concerns sums of the shape
\begin{equation}\label{eq:qv-sums-classical}
  \sum_{\varphi \in \mathcal{F}}
  \langle \varphi, \Psi_1 \varphi \rangle
  \langle \Psi_2 \varphi, \varphi \rangle.
\end{equation}
Here $\Psi_1,\Psi_2$ are fixed mean-zero functions on the unit cotangent bundle of a Riemannian manifold $M$ with ergodic geodesic flow, $\mathcal{F}$ traverses a sequence of families of microlocal lifts of Laplace eigenfunctions with eigenvalues in $[0, T^2]$, and $T \rightarrow \infty$.  The problem is to determine the leading order asymptotic behavior of \eqref{eq:qv-sums-classical}.  The difficulty of the problem may be appreciated by comparing the expected magnitude $\asymp T$ for \eqref{eq:qv-sums-classical} for typical $\Psi_1 = \Psi_2$ with the best known general upper bound $O(T^{\dim(M)} / \log T)$ (see, e.g., \cite[\S1]{nelson-variance-73-2} and references for details).

Although a mathematically rigorous solution to the problem seems hopeless on general $M$, Sarnak--Zhao \cite{2013arXiv1303.6972S} (following Luo--Sarnak \cite{luo-sarnak-mass} and Zhao \cite{MR2651907}) managed to solve it completely on 
$M = \SL_2(\mathbb{Z}) \backslash \mathbb{H}$ for $\mathcal{F}$ consisting of Hecke eigenfunctions.  It is natural to seek analogous results on other arithmetic quotients, such as the compact quotients attached to orders in quaternion division algebras.  The method of Luo, Sarnak and Zhao demonstrates the power of \emph{parabolic Fourier expansions}, such as the $q$-expansions $\sum a_n q^n$ enjoyed by classical holomorphic modular forms on $\SL_2(\mathbb{Z})$ at the cusp $\infty$, to establish results that are inaccessible by means of semiclassical analysis or trace formulas alone.  Conversely, their technique is fundamentally limited to \emph{split} quotients, such $\SL_2(\mathbb{Z}) \backslash \mathbb{H}$ and its congruence covers, on which such expansions are available.

In this article, we develop systematically a method for studying quantum variance that applies also to \emph{non-split} arithmetic quotients arising from non-split quaternion algebras, in contrast to the split matrix algebra $M_2(\mathbb{Q})$ underlying the quotient $\SL_2(\mathbb{Z}) \backslash \mathbb{H}$ considered by Luo, Sarnak and Zhao.  Part \ref{part:quant-vari-theta} of this paper establishes global estimates concerning the quantum variance of general families of automorphic forms (see Theorem \ref{thm:var-2}, stated in \S\ref{sec:orgb2d0a08}, and
Theorem \ref{thm:main-estimate-general-variance}, stated in \S\ref{sec-4}).  Part \ref{part:appl-micr-lifts} provides additional local estimates, at the archimedean place, relevant for determining the quantum variance of families of microlocal lifts of Hecke--Maass forms (Theorem \ref{thm:var-3}) and its analogue for holomorphic forms (Theorem \ref{thm:var-quat-annals-submission:limit-begin-lim_h}).

\subsection{Trace formulas and linear statistics}\label{sec:trace-formula-linear}
Let $\mathbf{X} := \Gamma \backslash G$ be the quotient by an arithmetic lattice of the points of a semisimple $\mathbb{Q}$-group over a local field, such as the real numbers, and let $\mathcal{F}$ be a ``large'' collection of eigenfunctions $\varphi : \mathbf{X} \rightarrow \mathbb{C}$.
One can ask for the asymptotic statistics, as $\mathcal{F}$ traverses a sequence of families, of the random measure on $\mathbf{X}$ sending a test function $\Psi$ to $\langle \varphi, \Psi \varphi \rangle$, where $\varphi \in \mathcal{F}$ is sampled randomly with respect to (say) the normalized counting measure.

The \emph{linear} statistics of this random measure are captured by the \emph{mean}
$\Psi \mapsto \mathbb{E}_{\varphi \in \mathcal{F}} \langle \varphi, \Psi \varphi \rangle$.
When $\mathcal{F}$ admits a nice harmonic-analytic description, it can be (at least approximately) picked off by a convolution kernel $f \in C_c^\infty(G)$.  The mean can then be studied using trace formula techniques: by integrating the pretrace formula
\begin{equation}\label{eq:intro-pretrace-fromula}
  \sum_{\gamma \in \Gamma} f(x^{-1} \gamma y)
  \approx
  \sum_{\varphi \in \mathcal{F} }
  \overline{\varphi(x)} \varphi(y) \quad (x,y \in \Gamma \backslash G)
\end{equation}
over the diagonal against $\Psi$, one obtains an identity
\begin{equation}\label{eq:intro-pretrace-fromula-1}
  \sum_{\varphi \in \mathcal{F} }
  \langle \varphi, \Psi \varphi  \rangle
  \approx
  \int_{x \in \Gamma \backslash G}
  \Psi(x) \sum_{\gamma \in \Gamma} f(x^{-1} \gamma x) \,d x 
\end{equation}
whose right hand side may be studied by methods for bounding orbital integrals much as in the ``Weyl's law'' case $\Psi \equiv 1$.

Higher-order statistics, such as the $n$-point correlations
\[
  (\Psi_1,\dotsb,\Psi_n) \mapsto \mathbb{E}_{\varphi \in \mathcal{F}} \langle \varphi, \Psi_1 \varphi \rangle \dotsb \langle \varphi, \Psi_n \varphi \rangle,
\]
are more mysterious.  The quantum variance problem concerns the quadratic statistics \eqref{eq:qv-sums-classical}, about which trace formulas alone say little.

\subsection{Hecke multiplicativity and variance statistics}\label{sec:35ac3e567f}
Until this series of works, the only known asymptotic formulas for higher-order statistics in this setting of \S\ref{sec:trace-formula-linear} were those of Luo--Sarnak--Zhao concerning $\SL_2(\mathbb{Z}) \backslash \SL_2(\mathbb{R})$.  The point of departure for their method is that when the eigenfunctions $\varphi$ admit Fourier expansions whose coefficients $\lambda(n)$ enjoy a ``doubling identity'' of the shape
\begin{equation}\label{eqn:hecke-multiplicativity}
  \lambda(m) \lambda(n) = \sum \lambda(\dotsb),
\end{equation}
one can try to reduce variance statistics to linear ones and apply trace formulas such as \eqref{eq:intro-pretrace-fromula}.  This method does not apply when such expansions are not available.

\subsection{Theta functions and variance statistics}\label{sec:theta-funct-vari}
When the space $\mathbf{X}$ arises from a quaternion algebra $B$ (over $\mathbb{Q}$, say), the Eichler/Shimizu theta correspondence provides an analogue of the doubling identity \eqref{eqn:hecke-multiplicativity} that suggests a natural strategy for studying quantum variance.  We pursue this strategy here.  Let $\mathcal{F}$ be a family of eigenfunctions on $\mathbf{X}$.  Oversimplifying for now, Shimizu's theorem (see \cite[II.1]{MR783511}) says that one can find
\begin{itemize}
\item a space $\mathbf{X} '$ (a congruence cover of $\PGL_2(\mathbb{Z}) \backslash \PGL_2(\mathbb{R})$),
\item a function of three variables $\Theta : \mathbf{X} \times \mathbf{X} \times \mathbf{X} ' \rightarrow \mathbb{C}$ (a theta kernel), and
\item for each $\varphi \in \mathcal{F}$, a function $\Phi_\varphi : \mathbf{X} ' \rightarrow \mathbb{C}$ (an Eichler/Jacquet--Langlands lift)
\end{itemize}
with the property that
\begin{equation}\label{eq:shimizu-identity}
  \overline{\varphi(x)} \varphi(y)
  =
  \int_{z}
  \Phi_\varphi(z)
  \Theta(x,y;z) \, d z
  \quad \text{ for all $\varphi \in \mathcal{F}$
    and $x,y \in \mathbf{X}$.
  }
\end{equation}
By integrating the diagonal case $x=y$ of \eqref{eq:shimizu-identity} against $\Psi$, one obtains
\begin{equation}\label{eq:identity-pre-parseval}
  \langle \varphi, \Psi \varphi  \rangle
  =
  \int_{z}
  \Phi_\varphi(z)
  \int_x \Psi(x)
  \Theta(x,x;z)  \, d x \, d z.
\end{equation}
If the functions $\Phi_\varphi$ are orthogonal to one another and the family $\mathcal{F}$ is sufficiently ``complete,'' then a cavalier application of Parseval's formula to \eqref{eq:identity-pre-parseval} suggests that
\begin{equation}\label{eq:formula-seesaw}
  \begin{split}
    &\sum_{\varphi \in \mathcal{F}}
  \left(\int |\Phi_\varphi|^2 \right)^{-1}
  \langle \varphi, \Psi_1 \varphi  \rangle
  \langle \Psi_2 \varphi, \varphi  \rangle \\
 &\quad  =
  \int_{z}
  \left(\int_x \Psi_1(x)
  \Theta(x,x;z) \, d x\right)
  \overline{\left(\int_y \Psi_2(y)
    \Theta(y,y;z) \,d y \right)} \,d z.
\end{split}
\end{equation}
The left hand side of \eqref{eq:formula-seesaw} may be understood as a reasonable proxy for the quantum variance \eqref{eq:qv-sums-classical} of $\mathcal{F}$ provided that the weights $\int |\Phi_\varphi|^2$ are sufficiently uniform in $\varphi$.
A first aim of this article is to develop robust techniques for determining the asymptotics of the right hand side of \eqref{eq:formula-seesaw}, which is not \emph{a priori} any simpler to analyze than the left hand side.  A second aim is to apply the resulting machinery to an interesting family of automorphic forms.

We have oversimplified by neglecting that the theta kernel $\Theta$ produced by Shimizu's theorem may (and generally does) depend upon the automorphic form $\varphi$.  For the above argument to make sense, we need to choose one $\Theta$ that works for every element of the family $\mathcal{F}$.  It is natural instead to interpret \eqref{eq:formula-seesaw} as \emph{defining} a (weighted) family $\mathcal{F}$ in terms of $\Theta$.  A third aim of this article is then to clarify in general how to invert the association $\Theta \mapsto \mathcal{F}$.

\subsection{Setup and notation}\label{sec:intro-setup}
We now prepare to state our main results.  To that end, we begin by recalling (from, e.g., \cite[\S9]{MR0314766} or \cite[\S4]{MR580949} or \cite[\S38]{voightQA}) the parametrization of arithmetic hyperbolic surfaces $\mathbf{Y}$ and their unit cotangent bundles $\mathbf{X}$, restricting for simplicity to those attached to maximal orders.  Let $M := M_2(\mathbb{R})$ denote the $2 \times 2$ matrix algebra and $G := \PGL_2(\mathbb{R}) = M^\times / \mathbb{R}^\times$ its projectivized unit group.  Let $F$ be a totally real number field.  (Our results are new when $F = \mathbb{Q}$, but the generality does not introduce complication.)  Let $B$ be a quaternion algebra over $F$ for which there is exactly one real place $\mathfrak{q}$ of $F$ such that $B_{\mathfrak{q}}$ is isomorphic to $M$; we fix such an isomorphism, together with a maximal order $R \subseteq B$, and denote by $\Gamma \leq G$ the image of $R^\times$ under the induced isomorphism $B_{\mathfrak{q}}^\times / F_{\mathfrak{q}}^\times \cong G$.  Then $\Gamma \leq G$ is a discrete cofinite subgroup; it is cocompact except when $B$ is split, in which case $F = \mathbb{Q}$ and $B \cong M_2(\mathbb{Q})$.  We denote by $K \leq G$ the image of $\O(2)$ and by $G^1, K^1, \Gamma^1$, the subgroups consisting of positive determinant elements.  We take
\begin{equation*}
  \mathbf{X} := \Gamma \backslash G,
  \quad 
\mathbf{Y} := \mathbf{X} / K^1.
\end{equation*}
We equip $G$ with any Haar measure $d g$ and $\mathbf{X}$ with any $G$-invariant measure.  We write
\begin{equation*}
  \langle \varphi_1, \varphi_2 \rangle := \int_{\mathbf{X}} \varphi_1 \overline{\varphi_2}
\end{equation*}
for $\varphi_1, \varphi_2 \in L^2(\mathbf{X})$.  The group $G$ acts unitarily on $L^2(\mathbf{X})$ by right translation: $g \varphi (x) := \varphi(x g)$.  We denote by
\begin{equation*}
  H =
  \begin{pmatrix}
    \ast    & 0 \\
    0 & \ast \\
  \end{pmatrix}
\end{equation*}
the diagonal subgroup of $G$.

We assume for simplicity of presentation that $F$ has odd narrow class number.  The strong approximation theorem (see \S\ref{sec-4-4-2} or \cite[\S28]{voightQA}) then identifies $\mathbf{X}$, rather than a finite disjoint union of similarly defined quotients, with the adelic quotient $\mathbf{G}(F) \backslash \mathbf{G}(\mathbb{A}) /J$, where
\begin{itemize}
\item $\mathbb{A}$ denotes the adele ring of $F$,
\item $\mathbf{G}$ denotes the $F$-algebraic group $A \mapsto (B \otimes_{F} A)^\times / A^\times$, and
\item $J = J_{\infty} \prod_{\p < \infty }J_{\p}$, with $J_{\p }$ the image of $R_p^\times$ and $J_{\infty} \cong \SO(3)^{[F:\mathbb{Q}]-1}$ the points of $\mathbf{G}$ over the product of the archimedean completions of $F$ other than $F_{\mathfrak{q}}$.
\end{itemize}
We obtain for each finite prime $\p$ of $F$ a Hecke operator $T_{\p}$ acting on $L^2(\mathbf{X})$, defined by a sum over cosets: temporarily writing $r_p \in R_p$ for any element whose norm has valuation one, we have
\begin{equation}\label{eq:cnsez2t3gv}
  T_p \varphi(x) =
  \sum_{
    \alpha J_p
    \subseteq J_p r_p J_p
  } \varphi(x \alpha).
\end{equation}
The number of cosets is $|\p| + 1$ or $1$ according as $\p$ does or does not split $B_{\p}$, where $|p|$ denotes the absolute norm.  These operators commute with one another and also with $G$.  Strong approximation also implies that the group $\Gamma$ contains elements of negative determinant.  We may thus identify $\mathbf{X}$ with $\Gamma^1 \backslash G^1$ and $\mathbf{Y}$ with $\Gamma^1 \backslash \mathbb{H}$, where $\mathbb{H} \cong G^1/K^1$ is the hyperbolic plane.

By an \emph{eigenfunction} $\Psi : \mathbf{X} \rightarrow \mathbb{C}$, we mean a smooth $K$-finite function that generates an irreducible representation of $G$ and is a $T_{\p}$-eigenfunction for each $p$.

We assume that $B$ is non-split, so that $\mathbf{X}$ is compact.  We denote by $L^2_0(\mathbf{X}) \subseteq L^2(\mathbf{X})$ the subspace of mean-zero functions and by $A_0$ the set of subspaces $\pi \subseteq L_0^2(\mathbf{X})$ that are irreducible under $G$ and eigenspaces for each $T_{\p}$.  The smooth $K$-finite vectors in $\pi$ are eigenfunctions in the above sense, and each nonzero eigenfunction generates one such $\pi$.  The multiplicity one theorem implies that the space $L^2_0(\mathbf{X})$ is the Hilbert direct sum of the $\pi \in A_0$.  Under strong approximation in the sense noted above, $A_0$ identifies with the set of generic automorphic representations of $\mathbf{G}$ containing a nonzero $J$-invariant vector.

Each $\pi \in A_0$ has an infinitesimal character $\lambda_\pi \in \mathbb{R}$, describing the action of the center of the universal enveloping algebra of $G$ (see \S\ref{sec:prelims-representations}).  If $\lambda_\pi < 0$, then $\pi$ contains a one-dimensional space of $K^1$-invariant vectors $\varphi_\pi$; these descend to Laplace eigenfunctions on $\mathbf{Y}$ of eigenvalue $1/4-\lambda_\pi$, giving a bijection
\[
  \{\pi \in A_0 : \lambda_\pi < 0 \} \leftrightarrow \frac{\{\text{Hecke--Maass eigenforms $\varphi_\pi$ on $\mathbf{Y}$ of eigenvalue $>1/4$}\}}{\text{scaling}}.
\]
We normalize $\varphi_\pi$ so that $\langle \varphi_\pi, \varphi_\pi \rangle = 1$.  We denote by $\mu_\pi$ the representation-theoretic microlocal lift of $\varphi_\pi$ constructed by Zelditch and studied in the related work of Sarnak--Zhao \cite{2013arXiv1303.6972S}.  We recall the precise construction of $\mu_\pi$ in \S\ref{sec:constr-mu-pi-ZW}.  We mention for now only that it is a functional
\begin{equation}\label{eq:mu-pi-first-appearance}
  \mu_\pi :
  \{
  \text{$K$-finite smooth
    $\Psi : \mathbf{X} \rightarrow \mathbb{C}$}
  \}
  \rightarrow \mathbb{C}
\end{equation}
with the following properties:
\begin{itemize}
\item {\bf Lifting.} If $\Psi$ is $K^1$-invariant, so that it comes from a function on $\mathbf{Y}$, then $\mu_\pi(\Psi) = \langle \varphi_\pi \Psi, \varphi_\pi \rangle$.
\item {\bf Flow invariance.}  $\mu_\pi$ is asymptotically $H$-invariant: for each fixed $h \in H$ and $\Psi$ as above, the difference $\mu_{\pi}(h \Psi) - \mu_\pi(\Psi)$ tends to zero as $|\lambda_\pi| \rightarrow \infty$.
\end{itemize}
For further discussion concerning the $\mu_\pi$, we refer to \cite{MR916129, MR850155, MR976312, MR1838659, MR1814849, MR1859345, MR2346281, MR2314452}.

A theorem of Lindenstrauss \cite{MR2195133}, resolving a case of the arithmetic quantum unique ergodicity conjecture of Rudnick--Sarnak \cite{MR1266075}, implies that $\mu_\pi(\Psi) \rightarrow \langle \Psi, 1 \rangle / \vol(\mathbf{X})$ for each fixed continuous function $\Psi: \mathbf{X} \rightarrow \mathbb{C}$ as $|\lambda_\pi| \rightarrow \infty$.  Equivalently, $\mu_\pi(\Psi) \rightarrow 0$ when $\Psi$ has mean zero.


Let $S$ denote the (finite) set of finite primes of $F$ at which $B$ ramifies.  For $\p \in S$, the operator $T_{\p}$ is an involution (as follows from the definition \eqref{eq:cnsez2t3gv} and the fact that $J_p$ is the kernel of the norm-valuation map $\mathbf{G}(F_p) \rightarrow \mathbb{Z} / 2 \mathbb{Z}$).  We may thus speak of the parity of an eigenfunction with respect to $T_{\p}$.  The local root number of $\sigma \in A_0$ at the distinguished real place $\mathfrak{q}$ is of the form $\pm 1$; it is $+1$ precisely when $\pi$ admits a nonzero functional invariant by the normalizer in $G$ of $H$.  We say that $\sigma \in A_0$ is \emph{even} if
\begin{itemize}
\item for each $p \in S$, the $T_p$-eigenvalue of $\sigma$ is $+1$, and
\item the local root number of $\sigma$ at $\mathfrak{q}$ is $+1$.
\end{itemize}
If $\sigma$ is not even, then $\mu_\pi(\Psi) = 0$ for all $\Psi \in \sigma$ (see \S\ref{sec:branch-coeff}); for this reason, we focus primarily on even $\sigma$.  We say also that an eigenfunction $\Psi$ is even if it belongs to an even $\sigma \in A_0$.

For a function $\Psi : \mathbf{X} \rightarrow \mathbb{C}$, we write $\Psi^w := \frac{1}{2} (\Psi + w \Psi)$ for its symmetrization with respect to the Weyl element $w := \begin{pmatrix}
  & 1 \\
  -1 &
\end{pmatrix} \in G$.

We equip the diagonal subgroup $H \leq G$ with the Haar measure given by
\[
  \int_H f := \int_{y \in \mathbb{R}^\times} f (\begin{pmatrix}
    y &   \\
    & 1
  \end{pmatrix}) \, \frac{d y}{|y|},
\]
where $d y$ denotes Lebesgue measure.

We denote in what follows by $L^{(S)}(\dotsb,s)$ the finite part of an $L$-function, omitting Euler factors in $S$.  For $\pi \in A_0$, we abbreviate
\begin{equation*}
  \iota_{\pi} := L^{(S)}(\ad \pi,1).
\end{equation*}

\subsection{Main result}\label{sec:main-result}
We henceforth fix a pair of nonzero mean-zero even eigenfunctions $\Psi_1, \Psi_2$.  We denote by $\sigma_1, \sigma_2 \in A_0$ the representations that they generate.

\begin{theorem}[Quantum variance of microlocal lifts on compact arithmetic surfaces]\label{thm:var-3}
  The limit
  \begin{equation}\label{eq:thm-1}
    \lim_{\h \rightarrow 0}
    \h
    \sum_{
      \substack{
        \pi \in A_0 : \\
        0 < -\h^2 \lambda_\pi < 1
      }
    }
    \iota_{\pi}
    \mu_\pi(\Psi_1)
    \overline{\mu_\pi(\Psi_2)}
  \end{equation}
  exists.  If $\sigma_1 \neq \sigma_2$, then that limit is zero.  If $\sigma_1 = \sigma_2 =: \sigma$, then it is given by
  \begin{equation}\label{eqn:limiting-variance-formula}
    \frac{c_B}{2 \pi}
    L^{(S)}(\sigma,\tfrac{1}{2})
    \int_{s \in H}
    \langle s \Psi_1^w, \Psi_2^w \rangle \,d s,
  \end{equation}
  where $c_B := 2^{\# S} \zeta_F^{(S)}(2) / \vol(\mathbf{X})$.
\end{theorem}

\begin{remark}
  As discussed in \cite{2013arXiv1303.6972S} or \S\ref{sec:local-convergence-lemmas}, the integral on the right hand side of \eqref{eqn:limiting-variance-formula} converges absolutely.  We note also that each of the expressions \eqref{eq:thm-1} and \eqref{eqn:limiting-variance-formula} is independent of the choice of Haar measure on $\mathbf{X}$.  If we equip $\mathbf{X}$ with the pullback of the standard hyperbolic measure $\frac{d x \, d y}{y^2}$ on $\mathbf{Y}$, then we may verify as in \cite[\S1]{watson-2008} that
  \[
    \frac{c_B}{2 \pi} = 2^{\# S} \frac { (4 \pi^2)^{[F:\mathbb{Q}]-1} \prod_{\p \in S} (1 + 1/|\p|) } { 4 \Delta_F^{3/2} \Delta_B }
  \]
  where $\Delta_F$ and $\Delta_B$ denote the absolute discriminant and absolute reduced discriminant, respectively, and $|\p|$ denotes the absolute norm of the finite prime $\p$.  For instance, if $F = \mathbb{Q}$, then
  \[
    \frac{c_B}{2 \pi} = 2^{\# S} \frac { \prod_{p \in S} (1 + 1/p) } { 4 \Delta_B }
  \]
  The factor $2^{\# S}$ may be understood (see \S\ref{sec:heuristics}) as coming from the nontrivial normalizer of $\Gamma$, corresponding to the involutory Hecke operators $T_{\p}$ ($\p \in S$).  If one instead sums over only those $\pi$ having eigenvalue $+1$ under such operators, then this factor disappears.
\end{remark}

\begin{remark}\label{rmk:var-quat-annals-submission:proof-gives-rate}
  The proof gives a rate of convergence in \eqref{eq:thm-1} of the form $\h^{\delta}$ for some fixed $\delta > 0$.  We do not explicate or optimize the exponent here.  For the variant problem obtained by replacing the sharp truncation $0 < - \h^2 \lambda_\pi < 1$ by a smooth dyadic weight, one could likely optimize our methods to obtain the rate $\O(\h^{1/2})$ and show that this rate is best possible (cf. \cite[\S6.5]{nelson-variance-73-2}).
\end{remark}

\begin{remark}\label{rmk:involving-sym}
  For the sake of comparison\footnote{
    Our main term is half the analogue of that obtained in the cited reference, due to minor computational errors in the latter.
  }
  with \cite{2013arXiv1303.6972S}, we note that
  \begin{equation}\label{eq:weyl-inv}
    \int_{s \in H}
    \langle s \Psi_1^w, \Psi_2^w \rangle \,d s 
    =
    2
    \int_{u \in \mathbb{R}}
    \langle \begin{pmatrix}
      e^{u/2} &  \\
      & e^{-u/2}
    \end{pmatrix} \Psi_1^{\sym}, \Psi_{2}^{\sym} \rangle \, d u,
  \end{equation}
  where $\Psi^{\sym}$ denotes the average of $\Psi$ over its translates by the four-element subgroup of $G$ generated by $\diag(-1,1)$ and $w$.
\end{remark}

\begin{remark}\label{rmk:removing-weights}
  The ``arithmetic weights'' $\iota_{\pi}$ arise in our method for reasons illustrated best by \cite[\S2.8, \S7]{nelson-variance-73-2}.  They have mild size ($\O(\h^{-\eps})$ for any fixed $\eps > 0$) and mean $1$.  Sarnak--Zhao \cite{2013arXiv1303.6972S} showed in the non-compact case that if one modifies the sums \eqref{eq:thm-1} by omitting the weights $\iota_{\pi}$, then the conclusion remains valid after multiplying the main term \eqref{eqn:limiting-variance-formula} by a certain explicit factor $c_\sigma > 0$.  To do this, they used zero density estimates for families of $L$-functions to approximate $\iota_{\pi}^{-1}$ for most $\pi$ by a short Dirichlet polynomial, and then appealed to estimates for Hecke-twisted variants of \eqref{eq:thm-1}.  Their method applies in our setting with the (analogous) constant
  \begin{equation}\label{eq:c-sigma-defn}
    c_\sigma :=
    \frac{1}{\zeta_F^{(S)}(2)}
    \prod_{p \notin S}
    \left( 1 - \frac{\lambda_\sigma(p)}{|p|^{3/2} + |p|^{1/2}}
    \right),
  \end{equation}
  where $\lambda_\sigma(p)$ denotes the Hecke eigenvalue normalized so that the Ramanujan conjecture says $|\lambda_\sigma(p)| \leq 2$.  We do not replicate here the details of their argument, but explain in \S\ref{sec:remov-arithm-weights} how the factor $c_\sigma$ arises naturally from the perspective of our method.
\end{remark}
\begin{remark}
  In \S\ref{sec:heuristics}, we extend the semiclassical heuristics for variance asymptotics from the generic non-arithmetic setting (see, e.g.,  \cite[\S15.6]{2009arXiv0911.4312Z}, \cite[\S4.1.3]{MR3204186}) to the setting of Theorem \ref{thm:var-3}.  In brief, the generic heuristics come from postulating that the average of $\mu_\pi(\Psi_1) \overline{\mu_\pi (\Psi_2)}$ over $\pi \in \mathcal{F}$ should be approximated by the double average of $\mu_{\pi_1}(\Psi_1) \overline{\mu_{\pi_2} (\Psi_2)}$ restricted to those $\pi_1, \pi_2 \in \mathcal{F}$ whose archimedean parameters are close in the sense that $\left\lvert \sqrt{- \lambda_{\pi_1}} - \sqrt{- \lambda_{\pi_2}} \right\rvert < \eps$ for some small $\eps > 0$.  In arithmetic settings, it is natural to impose the further condition that $\left\lvert \lambda_{\pi_1}(p) - \lambda_{\pi_2}(p) \right\rvert < \eps$ for all small primes $p$.  We show that, modulo identifying limits of finite Euler products with their formal $L$-function limits, the resulting predictions are consistent with our results.
\end{remark}

\subsection{The holomorphic analogue}\label{sec:holomorphic-analogue}
Our method applies just as well to holomorphic forms, giving an extension of Luo--Sarnak \cite[Thm 1]{MR2103474} to the setting of compact arithmetic surfaces.\footnote{  Our main term is half the analogue of that obtained in the cited reference, due to a minor error in the latter: on p788, the penultimate display should be multiplied by $1/2$, since the sum before (36) is taken only over even integers. }

Let $\pi \in A_0$ with $\lambda_\pi > 0$.  Then $\pi$, as a representation of $G$, is a discrete series representation of lowest weight $k$ for some natural number $k$, and $\lambda_\pi = (k-1/2)^2$.  (We normalize so that $k$ is not necessarily even -- see \S\ref{sec:prelims-representations}.)  For such $\pi$, the analogue of the microlocal lifts are the $L^2$-masses
\begin{equation*}
  \mu_\pi(\Psi) := \langle \varphi_\pi \Psi, \varphi_\pi \rangle
\end{equation*}
attached to a unit vector $\varphi_\pi \in \pi$ of weight $k$, i.e., the lift of a holomorphic modular form on $\mathbf{Y}$.  Such measures are $K$-invariant, so it suffices to test them against observables $\Psi$ that are likewise $K$-invariant.  Thus, fix a pair of nonzero $K$-invariant eigenfunctions $\Psi_1, \Psi_2$ whose $T_p$-eigenvalue, for each $p \in S$, is $+1$, and let $\sigma_1, \sigma_2 \in A_0$ denote the (even) representations that they generate.
\begin{theorem}[Quantum variance of holomorphic forms on compact arithmetic surfaces]\label{thm:var-quat-annals-submission:limit-begin-lim_h}
  The limit
  \begin{equation*}
    \lim_{\h \rightarrow 0}
    \h
    \sum_{
      \substack{
        \pi \in A_0 : \\
        0 < \h^2 \lambda_\pi < 1
      }
    }
    \iota_{\pi}
    \mu_\pi(\Psi_1)
    \overline{\mu_\pi(\Psi_2)}
  \end{equation*}
  exists.  If $\sigma_1 \neq \sigma_2$, then that limit is zero.  If $\sigma_1 = \sigma_2 =: \sigma$, then it is given by
  \begin{equation*}
    c_B
    L^{(S)}(\sigma,\tfrac{1}{2})
    \langle \Psi_1, \Psi_2 \rangle
  \end{equation*}
  with $c_B$ as in Theorem \ref{thm:var-3}.
\end{theorem}

The proof differs very mildly from that of Theorem \ref{thm:var-3}. A unified treatment could be given, but to keep the exposition concrete, we first prove Theorem \ref{thm:var-3}, then explain in \S\ref{sec:holomorphic-analogue-1} the modifications needed to obtain a proof of Theorem \ref{thm:var-quat-annals-submission:limit-begin-lim_h}.

\subsection{A quadratic trace formula}\label{sec:orgb2d0a08}
To present the proof of Theorem \ref{thm:var-3} as clearly as we can, we separate the difficulties concerning general families of automorphic forms (Part \ref{part:quant-vari-theta}) from those specific to the families of microlocal lifts considered above (Part \ref{part:appl-micr-lifts}).  Our treatment of general families is encapsulated by a result which we now formulate.

Let $\pi \in A_0$.  We identify finite-rank operators $T$ on $\pi$ with finite-rank tensors $T = \sum_i v_i \otimes \overline{v_i'} \in \pi \otimes \overline{\pi}$.  Given any such $T$ and any bounded measurable $\Psi : \mathbf{X} \rightarrow \mathbb{C}$, we set
\[
  \mu(T,\Psi) := \sum_i \langle v_i \Psi, v_i' \rangle.
\]
We verify readily (see, e.g., \cite[\S26.3]{nelson-venkatesh-1}) that $|\mu(T,\Psi)| \leq \|T\|_1 \|\Psi \|_{L^\infty}$, where $\|.\|_1$ denotes the trace norm.  We may thus extend the assignment $T \mapsto \mu(T,\Psi)$ continuously to trace class operators $T$ on $\pi$, and in particular, to the integral operators $\pi(f) := \int_{g \in G} f(g) \pi(g) \, d g$ attached to $f \in C_c^\infty(G)$ and our choice of Haar measure $d g$ on $G$.  Equivalently, we may express $\mu(T,\Psi)$ as the absolutely convergent sum
\[
  \mu(T,\Psi) = \sum_{v \in \mathcal{B}(\pi)} \langle T v \cdot \Psi, v \rangle,
\]
where $\mathcal{B}(\pi)$ is an orthonormal basis for $\pi$ consisting of $K$-isotypic vectors.

Recall that we have fixed some nonzero mean-zero even eigenfunctions $\Psi_1 \in \sigma_1, \Psi_2 \in \sigma_2$.  We define hermitian forms $\mathcal{V}$ and $\mathcal{M}$ on $C_c^\infty(G)$ as follows:
\begin{itemize}
\item $\mathcal{V}(f) := \sum_{\pi \in A_0} \iota_{\pi} \mu(\pi(f),\Psi_1) \overline{\mu(\pi(f),\Psi_2)}$.
\item For a function $f : G \rightarrow \mathbb{C}$, we define its ``symmetrization under inversion''
  \begin{equation*}
    \mathfrak{S} f(g) := \frac{f(g) + f(g^{-1})}{2}
  \end{equation*}
  and, for $g \in G$, the conjugated function $\Ad(g) f(x) := f(g^{-1} x g)$
\item If $\sigma_1 \neq \sigma_2$, then $\mathcal{M}(f) := 0$.  If $\sigma_1 = \sigma_2 =: \sigma$, then we set
  \begin{equation*}
    \mathcal{M}(f) := c_B L^{(S)}(\sigma,\tfrac{1}{2}) \mathcal{I}(f),
  \end{equation*}
  with $c_B$ as in \eqref{eqn:limiting-variance-formula} and
  \begin{equation*}
    \mathcal{I}(f) := \int_{g \in G} \langle \Ad(g) \mathfrak{S} f, \mathfrak{S} f \rangle_{G} \langle g \Psi_1, \Psi_2 \rangle \,d g
  \end{equation*}
  with $\langle , \rangle_G$ the inner product in $L^2(G)$.
\end{itemize}
We note that the sum defining $\mathcal{V}(f)$ converges rapidly (see \S\ref{sec:pretrace-formula}, \S\ref{sec:main-general-estimate-key-defns}) and the integral defining $\mathcal{M}(f)$ converges absolutely (see \S\ref{sec:local-convergence-lemmas}, \S\ref{sec:local-integrals-for-rallis-ipf}).

For any real vector space $V$, we denote by $\mathcal{S}(V)$ the space of Schwartz functions.  Recall that $M = M_2(\mathbb{R})$ denotes the $2 \times 2$ matrix algebra.  For each $\tau \in \mathbb{R}^\times$ and $f \in C_c^\infty(G)$, we define $\heartsuit^{\tau} f \in \mathcal{S}(M)$ by the formula
\[
  \heartsuit^{\tau} f(x) := 1_{M^\times}(x) \frac{W(\tau \det(x))}{|\tau \det(x)|} f(\pr(x)),
\]
where $\pr : M^\times \rightarrow G$ denotes the natural projection and $W \in C_c^\infty(\mathbb{R}^\times)$ is a nonzero test function that we fix once and for all.

\begin{remark*}
  The motivation for introducing the sums $\mathcal{V}(f)$ is that for suitable $f$, they will be seen to approximate the basic variance sums of interest.  The ``expected main terms'' $\mathcal{M}(f)$ will arise after some calculations involving theta functions and the Rallis inner product formula.  We refer to \S\ref{sec:heuristics} for some heuristic discussion, independent of our rigorous arguments, about why one should expect $\mathcal{V}(f) \approx \mathcal{M}(f)$ for nice enough $f$.  The operators $\heartsuit^{\tau}$ should be understood as associating to a function on the multiplicative group $G$ its ``thickening'' on the additive group $M$.  The key observation concerning these operators, detailed in \S\ref{sec:estimates-general-var}, is that $\heartsuit^{\tau} f$ is the kernel of a theta function with $L^2$-norm proportional to $\mathcal{V}(f)$.
\end{remark*}

Let $M^0 \leq M$ denote the trace zero subspace.  We identify $\mathbb{R}$ with the subspace of scalar matrices in $M$.  We then have an orthogonal decomposition $M = \mathbb{R} \oplus M^0$.

For $y \in \mathbb{R}^\times$, we denote by $D_y$ the operator on $\mathcal{S}(M)$ given by normalized scaling of the $M_0$ component: for $\Phi \in \mathcal{S}(M), t \in \mathbb{R}, u \in M^0$,
\[
  D_y \Phi(t+u) := |y|^{3/2} \Phi(t + y u).
\]
It extends to a unitary operator on $L^2(M)$.

\begin{theorem}\label{thm:var-2}
  There is a finite subset $X$ of $\mathbb{R}^\times$ and a collection $(\mathcal{E}_{\tau_1, \tau_2})_{\tau_1, \tau_2 \in X}$ of sesquilinear forms on $\mathcal{S}(M)$ so that for each $f \in C_c^\infty(G)$,
  \begin{equation}\label{eq:var-2-main-identity}
    \mathcal{V}(f) = \mathcal{M}(f) + \sum_{\tau_1,\tau_2 \in X}
    \mathcal{E}_{\tau_1, \tau_2}(\heartsuit^{\tau_1} f,
    \heartsuit^{\tau_2} f).
  \end{equation}
  Moreover, there is a continuous seminorm $\mathcal{C}$ on the Schwartz space $\mathcal{S}(M)$ so that for all $y \in \mathbb{R}^\times$ and $\phi_1, \phi_2 \in \mathcal{S}(M)$,
  \begin{equation}\label{eqn:main-estimate}
    |\mathcal{E}_{\tau_1,\tau_2}(D_y \phi_1,
    D_y\phi_2)|
    \leq
    \frac{\log(|y| + |y|^{-1})}{|y| + |y|^{-1}}
    \mathcal{C}(\phi_1)
    \mathcal{C}(\phi_2).
  \end{equation}
\end{theorem}
The proof is completed in \S\ref{sec:deduct-theor-refthm}.  The key feature, elaborated in \S\ref{sec:sketch-deduct-theor}, is that Theorem \ref{thm:var-2} reduces the proofs of Theorem \ref{thm:var-3} and Theorem \ref{thm:var-quat-annals-submission:limit-begin-lim_h} to local problems.

\subsection{Outline of the proof of Theorem \ref{thm:var-2}}
\label{sec-1-6}
We follow the general strategy of \S\ref{sec:theta-funct-vari}.  The thickenings $f \mapsto \heartsuit^{\tau} f$ have been carefully constructed so as to define some $\Theta$ for which something like \eqref{eq:formula-seesaw} holds.  The integral $\int_z \Psi_i(z) \Theta(x,x;z) \,d z$ does not define a theta lift of $\Psi_i$ in the traditional sense, but instead decomposes as a sum of products $\theta_i(z) h_i(z)$, where $\theta_i$ is a variant of the Jacobi theta function and $h_i$ is a theta lift of $\Psi_i$.  The right hand side of \eqref{eq:formula-seesaw} then decomposes as a sum of inner products
\begin{equation}\label{eq:ip-4-theta-before-rearr}
  \langle \theta_1 h_1, \theta_2 h_2 \rangle.
\end{equation}
Suppose we can approximate each such inner product by
\begin{equation}\label{eq:ip-4-theta-after-rearr}
  \langle \theta_1, \theta_2 \rangle \langle h_1, h_2 \rangle.
\end{equation}
The Rallis inner product formula \cite{2012arXiv1207.4709T,
  MR2837015}
for theta lifts applies to
\eqref{eq:ip-4-theta-after-rearr};
summing it up,
we obtain
\begin{equation}\label{eq:asymptotic-after-rallis-in-intro-sketch}
  \sum_{\varphi \in \mathcal{F}}
\left(\int |\Phi_\varphi|^2 \right)^{-1}
\langle \varphi, \Psi_1 \varphi  \rangle
\langle \Psi_2 \varphi, \varphi  \rangle
\approx
\sharp
\mathcal{I}(f),
\end{equation}
where:
\begin{itemize}
\item $\approx$ means up to the error incurred by replacing each term \eqref{eq:ip-4-theta-before-rearr} with \eqref{eq:ip-4-theta-after-rearr}, and
\item $\sharp$ means ``modify by a central $L$-value as in Theorem \ref{thm:var-3}.'' \end{itemize}
To complete the proof of Theorem \ref{thm:var-2}, we must show that the ``error'' hidden by $\approx$ satisfies the required estimate.  Unfolding the definitions of our theta series, this reduces to showing that, for $D_y$ the diagonal flow attached to $y \in \mathbb{R}^\times$,
\begin{align*}
  \langle \theta_1 \cdot D_y h_1, \theta_2 \cdot D_y h_2 \rangle
  &=
  \langle \theta_1 \overline{\theta_2}, D_y (\overline{h_1} h_2) \rangle \\
  &=
  \langle \theta_1, \theta_2 \rangle \langle h_1, h_2 \rangle
  + \O\left( 
    \frac{\log(|y| + |y|^{-1})}{|y| + |y|^{-1}} \right),
\end{align*}
where the error depends continuously upon the data underlying $\theta_i, h_j$.  This estimate is a variant, established in \cite{nelson-theta-squared} for the non-square-integrable function $\theta_1 \overline{\theta_2}$, of the well-known mixing property for the diagonal flow.

\subsection{Deduction of Theorem \ref{thm:var-3} from Theorem \ref{thm:var-2}}\label{sec:sketch-deduct-theor}
Here we reduce  that deduction to a series of estimates, established in Part \ref{part:appl-micr-lifts}.  

We first set some asymptotic notation and terminology.  We consider a sequence $\{\h\}$ of positive reals $\h$ tending to zero, as in the statement of Theorem \ref{thm:var-3}.  By an ``$\h$-dependent element'' of a set $U$, we mean a map $\{\h\} \rightarrow U$, which we understand colloquially as an element $u \in U$ that depends (perhaps implicitly) upon the parameter $\h$.  The word ``fixed'' will be taken to mean ``independent of $\h$.''  Our default convention is that quantities not labeled ``fixed'' may depend upon $\h$, but we will usually mention this dependence for the sake of clarity.  Standard asymptotic notation is defined accordingly: $A = \O(B)$, $A \ll B$ and $B \gg A$ mean that $|A| \leq c |B|$ for some fixed $c \geq 0$, while $A \asymp B$ means that $A \ll B \ll A$; the meaning of an infinite exponent as in $A = \O(\h^\infty)$ is that the indicated estimate holds upon substituting for $\infty$ any fixed positive quantity.  The fixed quantities $c$ may of course depend upon any previously mentioned fixed quantities.  We always assume that $\h$ is small enough with respect to any mentioned fixed quantities.  (For instance, we may speak of an $\h$-dependent element $\pi \in A_0$ satisfying $1/2 < -\h^2 \lambda_\pi < 1$; its microlocal lift $\mu_\pi$ is an $\h$-dependent distribution on $\mathbf{X}$ that satisfies $|\mu_\pi(\Psi)| = \O(1)$ for fixed $\Psi$ as in \eqref{eq:mu-pi-first-appearance}; for fixed $\Psi \in C^\infty(\mathbf{X})$, we have $\langle \varphi_\pi, \Psi \rangle = \O(\h^\infty)$.)

For the proof of Theorem \ref{thm:var-3}, a simple approximation argument reduces our task to showing that there is a fixed $\delta > 0$ so that for each fixed nonnegative $k \in C_c^\infty(\mathbb{R}_{<0})$,
\begin{equation}\label{eq:thm-1-b}
  \h
  \sum_{
    \substack{
      \pi \in A_0
    }
  }
  \iota_{\pi}
  k(\h^2 \lambda_\pi)^2
  \mu_\pi(\Psi_1)
  \overline{\mu_\pi(\Psi_2)}
  =
  c_k
  L^{(S)}(\sigma,\tfrac{1}{2})
  \int_{s \in H}
  \langle s \Psi_1^w, \Psi_2^w \rangle \,d s + \O(\h^\delta),
\end{equation}
where
\begin{equation*}
  c_k := c_B \int_{t \in \mathbb{R}_{\geq 0}} k(-t^2)^2 \, \frac{d t}{2 \pi}.
\end{equation*}
Indeed, it is enough to show this for a class $\mathcal{K}$ of $\h$-dependent nonnegative functions $k \in C_c^\infty(\mathbb{R}_{<0})$ with the following properties:
\begin{itemize}
\item ($\mathcal{K}$ is ``controlled'') Each $k \in \mathcal{K}$ is supported on a fixed compact subset of $\mathbb{R}_{<0}$ and bounded from above by a fixed quantity.
\item ($\mathcal{K}$ is ``sufficiently rich'') For each fixed nonnegative $k_0 \in C_c(\mathbb{R}_{<0})$, we may find $k, k_+ \in \mathcal{K}$ so that $|k - k_0| \leq k_+$ and $\int k_+ \rightarrow 0$ as $\h \rightarrow 0$.
\end{itemize}
We construct such a class $\mathcal{K}$ explicitly in \S\ref{sec:org8bcf1bd}.  (In fact, for the class that we construct, quantitatively stronger properties hold, adequate for obtaining the rate mentioned in Remark \ref{rmk:var-quat-annals-submission:proof-gives-rate}.)

In \S\ref{sec:defn-f} and beyond, we construct for each $k \in \mathcal{K}$ an $\h$-dependent element $f \in C_c^\infty(G)$ and show that
\begin{equation}\label{eqn:req-V}
  \mathcal{V}(f)
  = \h \sum_{\pi \in A_0}
  \iota_{\pi} k(\h^2 \lambda_\pi)^2
  \mu_\pi(\Psi_1)
  \overline{\mu_\pi(\Psi_2)}
  + \O(\h^\delta)
\end{equation}
and
\begin{equation}\label{eqn:req-I}
  \mathcal{I}(f)
  = \int_{s \in H} \langle s \Psi_1^w, \Psi_2^w
  \rangle
  \,d s
  \int_{t > 0}
  k(-t^2)^2 \, \frac{d t}{2 \pi} 
  + \O(\h^\delta)
\end{equation}
and
\begin{equation}\label{eqn:req-E}
  \mathcal{E}_{\tau_1,\tau_2}(\heartsuit^{\tau_1} f, \heartsuit^{\tau_2} f) \ll \h^{1-\delta'}
\end{equation}
for fixed $\tau_1, \tau_2 \in \mathbb{R}^\times$, where $\delta' = \delta'(\delta) > 0$ is a fixed quantity with $\delta' \rightarrow 0$ as $\delta \rightarrow 0$.  The required estimate \eqref{eq:thm-1-b} then follows from the identity \eqref{eq:var-2-main-identity}.

The idea behind the construction of $f$, completed in \S\ref{sec:defn-f}, is to arrange that $\pi(f)$ is an approximate weighted projector onto a ``space'' spanned by ``$k(\h \lambda_\pi)$-many'' unit vectors $v \in \pi$ for which $\langle v \Psi, v \rangle \approx \mu_\pi(\Psi)$; this leads to \eqref{eqn:req-V}.  The orbit method and philosophy developed in \cite{nelson-venkatesh-1} and summarized in \S\ref{sec:operator-calculus} is a suitable tool for constructing and studying such approximate projectors.  This step of the proof may be understood as implementing the identity \eqref{eq:intro-pretrace-fromula} for the family of microlocal lifts.

For the proof (\S\ref{sec:orge0a88b2}) of the main term estimate \eqref{eqn:req-I}, we pull the inner product $\langle , \rangle_G$ back to the Lie algebra, apply Parseval, and disintegrate the resulting integral along the coadjoint orbits; the subgroup $H$ then arises naturally as the stabilizer of the ``limiting microlocal support'' of the vectors underlying $\mu_\pi$.  In particular, the $H$-integral in Theorem \ref{thm:var-3} arises naturally and geometrically, unlike in \cite{2013arXiv1303.6972S}.

The error estimate \eqref{eqn:req-E}, proved in \S\ref{sec:org543f416}, is ultimately a consequence of \eqref{eqn:main-estimate} and the fact that the function $f$ that we construct concentrates just above the scale $1 + \O(\h) \subseteq G$ and barely oscillates below that scale.  The thickenings $\heartsuit^{\tau} f(x)$, for fixed $\tau$, are thus given in ``polar coordinates'' $x = r^{1/2} g$ ($r > 0, g \in \SL_2(\mathbb{R})$) by a mildly modulated bump function on the region where $r \asymp 1$ and $g = 1 + \O(\h)$.  This region may also be described in ``Cartesian coordinates'': it consists of $x \in M$ for which $\trace(x) \asymp 1$ and whose traceless part $x^0 := x - \trace(x)/2$ satisfies $|x^0| \ll \h$.  It follows that $\heartsuit^{\tau} f$ is the image under $D_{1/\h}$ of an essentially fixed function, so we are in good position to apply \eqref{eqn:main-estimate}.

\subsection{Related work}\label{sec:35ac3e56ab}
The prequel \cite{nelson-variance-II} applies the results of Part \ref{part:quant-vari-theta} of this paper to the ``$p$-adic microlocal lifts'' introduced in \cite{nelson-padic-que}; the analysis there is simpler, owing mainly to the availability of exact projectors in the $p$-adic Hecke algebra.  The prequel to the prequel \cite{nelson-variance-73-2} introduces the method by application to the simplest nontrivial case, but taking many \emph{ad hoc} shortcuts.  The paper \cite{Nelson-TwistedSym2} employs a related method to give nontrivial quantum variance upper bounds in a ``horizontal'' level aspect, where new difficulties emerge.  The spiritual ancestor to each of these was the paper \cite{2012arXiv1210.1243N} concerning the numerical evaluation of modular forms on compact arithmetic surfaces.

Raphael Steiner \cite{2018arXiv181103949S} (see also \cite{2020arXiv200907194K, MR4752128}) recently introduced a method for bounding sup norms of automorphic forms via fourth moments $\sum_{\varphi \in \mathcal{F}} |\varphi(g)|^4$ over families.  Such moments may be understood as ``degenerate quantum variance sums'' obtained by taking the observables $\Psi_1 = \Psi_2$ to be point masses.  The first step in Steiner's method and ours are related: both consist of expressing the sums of interest in terms of inner products of theta functions.  The methods diverge thereafter: Steiner et al.\ use geometry of numbers techniques to obtain upper bounds, while we employ (among other things) spectral theory to obtain asymptotic formulas.

\part{Quantum variance and theta functions}\label{part:quant-vari-theta}
The purpose of this section is to prove Theorem \ref{thm:var-2}.  The inputs are the general machinery of theta functions, the multiplicity one theorem and Eichler/Shimizu/Jacquet--Langlands correspondence, the Rallis inner product formula, and the results of the companion article \cite{nelson-theta-squared} concerning the spectral decomposition of $|\theta|^2$ for one-variable theta functions $\theta$.  A key construction is that of the $\heartsuit^{\tau}$ operators (\S\ref{sec:heartsuit}); these solve the ``inversion problem,'' mentioned at the end of \S\ref{sec:theta-funct-vari}, of constructing theta functions whose integrals recover quantum variance sums over any given family.

We deduce Theorem \ref{thm:var-2} by specializing a more general result, Theorem \ref{thm:main-estimate-general-variance}, formulated adelically over number fields.  The main benefit of working generally is that each input to our proof is formulated adelically, and it is more natural to carry out the classical translation once at the end of the proof rather than separately for each input.  An auxiliary benefit is that we may specialize in other ways, e.g., to the depth aspect \cite{nelson-variance-II}.

The main intermediary results are:
\begin{itemize}
\item The relation between quantum variance sums and integrals of theta functions, and the identification of a proposed ``main term'' in the integrals of theta functions that arise (Proposition \ref{prop:after-extracting-main-term}, Lemma \ref{lem:main-term-eval-for-heartsuit}).
\item General estimates for the ``error terms'' (Proposition \ref{prop:main-error-estimate-global-adelic-general}).
\end{itemize}
These are related respectively to the Eichler/Shimizu correspondence, the Rallis inner product formula, and a variant (established in \cite{nelson-theta-squared}) of the equidistribution of the diagonal flow.

The reader looking for a quick overview might first study carefully \S\ref{sec:general-notation}, \S\ref{sec-4-1}, \S\ref{sec:main-general-estimate-key-defns}, \S\ref{sec-4-3}, which are essentially self-contained.

\section{General notation\label{sec:general-notation}}
\label{sec-1-8}
Let $A$ be a field or an adele ring.  Let $B$ be a quaternion algebra over $A$.  We denote by $\iota : B \rightarrow B$ the main involution, by $\nr : B \rightarrow A$ and $\tr : B \rightarrow A$ the reduced norm and trace
\begin{equation*}
  \nr(x) := x x^{\iota},
  \quad
  \tr(x) := x + x^{\iota},
\end{equation*}
and by
\begin{equation*}
  B^0 := \{x \in B : \tr(x) = 0\}
\end{equation*}
the subspace of traceless quaternions.  We employ the notations
\[
  n(x) := \begin{pmatrix}
    1 & x \\
    0 & 1
  \end{pmatrix}, \quad a(y) :=
  \begin{pmatrix}
    y & 0 \\
    0 & 1
  \end{pmatrix}, \quad t(y) :=
  \begin{pmatrix}
    y & 0 \\
    0 & y^{-1}
  \end{pmatrix}, \quad n'(x) :=
  \begin{pmatrix}
    1 & 0 \\
    x & 1
  \end{pmatrix}
\]
\[
  \Ad(g) x := g x g^{-1}, \quad \Ad(g) f(x) := f(g^{-1} x g)
\]
and
\[
  \mathfrak{S} f(x) := \frac{f(x) + f(x - \tr(x))}{2}
\]
whenever they make sense.  For example, this is the case if $g$ belongs to the unit group $B^\times$ of a quaternion algebra $B$ over $A$ as above and $x$ belongs to (resp. $f$ is a function on) one of the sets $B^\times/A^\times, B, B^0$. 

We define the right regular representation $\rho_{\reg}(g) f(x) := f(x g)$ whenever it makes sense.

Given a local (resp. global) field $F$, a nontrivial unitary character $\psi$ of $F$ (resp. of $\mathbb{A}/F$) and an element $a \in F^\times$, we denote by $\psi^a$ the nontrivial unitary character with the same domain as $\psi$ given by
\begin{equation*}
  \psi^a(x) := \psi(a x).
\end{equation*}

For a finite-dimensional vector space $V$ over a local field or adele ring, we denote by $\mathcal{S}(V)$ the space of Schwartz--Bruhat functions $\phi : V \rightarrow \mathbb{C}$, topologized as usual (see, e.g., \cite[\S11]{MR0165033}).

Let $G$ be a group over an adele ring or a finite product of local fields.  We let $C_c^\infty(G)$ denote the space of smooth compactly supported functions; as usual, smooth means infinitely differentiable (resp. locally constant) with respect to the archimedean (resp. non-archimedean) variables.  Assume that we have equipped $G$ with a Haar measure $d g$.  Let $\pi$ be a smooth representation of a group that contains $G$.  Let $f \in C_c^\infty(G)$.  We then define the operator $\pi(f) \in \End(\pi)$ by
\begin{equation*}
  \pi(f) v := \int_{g \in G} f(g) \pi(g) v \,d g.
\end{equation*}

The use of Vinogradov notation is standard: $A = \O(B)$, $A \ll B$ and $B \gg A$ each signify that $|A| \leq c |B|$ for some ``constant'' $c$, with dependencies indicated by subscripts; $A \asymp B$ signifies that $A \ll B \ll A$.
 
We write $1_E$ for the characteristic function of a subset $E$ of some set $X$.  For an assertion $A$, we set $1_A := 1$ if $A$ is true and $1_A := 0$ if $A$ is false.

We set $\mathbb{C}^{(1)} := \{ z \in \mathbb{C}^{\times} : |z| = 1\}$.

\section{Local preliminaries}
\label{sec-2}
The purpose of this section is to collect local definitions, notation and identities for later use.  The notation introduced here should be self-descriptive with the exception of that for the similitude Weil representation $\Omega$ defined in \S\ref{sec:defn-local-omega}.

Let $k$ be a local field of characteristic $\neq 2$, thus $k$ is either $\mathbb{R}$, $\mathbb{C}$ or a finite extension of $\mathbb{Q}_p$ or (for $p \neq 2$) of $\mathbb{F}_p(t)$.  The assumption on the characteristic is relevant only for sections discussing the Weil representation.

Let $\psi : k \rightarrow \mathbb{C}^{(1)}$ be a nontrivial unitary character of $k$, and let $B$ be a quaternion algebra over $k$.  Set $G := B^\times/ k^\times$.  When $k$ is non-archimedean, let $R \subset B$ be a maximal order.

\subsection{Generalities}
\label{sec-2-1}
\subsubsection{The number field}
\label{sec-2-1-1}
We denote by $|.| := |.|_k : k \rightarrow \mathbb{R}_{\geq 0}$ the normalized absolute value, so that $d(c x) = |c| \, d x$ for $c \in k$ and any Haar measure $d x$ on $k$.

Let $\zeta_k(s)$ denote the local zeta function, thus $\zeta_k(s) = \pi^{-s/2} \Gamma(s/2), 2 (2 \pi)^{-s} \Gamma(s)$ or $(1-q^{-s})^{-1}$ as $k = \mathbb{R}, \mathbb{C}$, or a non-archimedean local field with residue field of cardinality $q$.

Recall that $B$ is \emph{split} if it is isomorphic to the algebra $M_2(k)$ of $2 \times 2$ matrices.  Otherwise, $B$ is called \emph{non-split} or \emph{ramified}; in that case, it is the unique (up to isomorphism) quaternion division algebra over $k$, and the group $G$ is compact.

When $k$ is non-archimedean, we denote by $\mathfrak{o}$ the maximal order, by $\mathfrak{q}$ the maximal ideal, by $q := \# \mathfrak{o} / \mathfrak{q}$ the cardinality of the residue field, by $\varpi \in \mathfrak{q} = \varpi \mathfrak{o}$ a uniformizer (thus $|\varpi| = q^{-1}$), and by $\Delta_{\psi}$ the absolute conductor of $\psi : k \rightarrow \mathbb{C}^{(1)}$, thus $\Delta_{\psi} = q^d$ if $\psi$ is trivial on $\mathfrak{q}^{-d}$ but not on $\mathfrak{q}^{-d-1}$.  Recall that $\psi$ is \emph{unramified} if $\Delta_{\psi} = 1$.

\subsubsection{Measures\label{sec:local-measures}}
\label{sec-2-1-2}
For $X \in \{k,B^0,B\}$, define the perfect pairing $\langle , \rangle : X \otimes X \rightarrow k$ by $\langle x,y \rangle := x y$ if $X = k$ and by $\langle x, y \rangle := \tr(x^{\iota} y)$ if $X = B^0,B$.  Equip $X$ with the Haar measure $d x$ for which the Fourier transform $\mathcal{F} : \mathcal{S}(X) \rightarrow \mathcal{S}(X)$ defined by
\begin{equation*}
  \mathcal{F} f(\xi) := \int_{x \in X} f(x) \psi(\langle x, \xi \rangle) \, d x
\end{equation*}
satisfies the inversion formula
\begin{equation}\label{eq:fourier-inversion-for-normalization}
  \mathcal{F} \mathcal{F} f(x) = f(-x).
\end{equation}
Equip $k^\times$ with the Haar measure $\int_{k^\times } f := \int_{x \in k^\times} f(x) \, \frac{d x}{|x|}$.

The quotient $k^\times / k^{\times 2}$ is finite.  We equip it with the Haar measure $d^\times_2 x$ compatible with the squaring map, so that for $f \in C_c(k^\times)$,
\begin{equation}\label{eqn:compatibility-squaring-map-local-measure}
  \int_{x \in k^\times}
  f(x)
  \,
  \frac{d x}{|x|}
  = 
  \int_{y \in k^\times/k^{\times 2}}
  \left(\int_{z \in k^\times}
  f(y z^2) \, \frac{d z}{|z|}\right) \, d _2^\times y.
\end{equation}
For $f : k^\times / k^{\times 2} \rightarrow \mathbb{C}$, one has explicitly
\begin{equation*}
  \int_{x \in k^\times / k^{\times 2}} f(x) \, d _2^\times y = \frac{|2|_k}{2} \sum_{x \in k^\times / k^{\times 2}} f(x).
\end{equation*}

Equip $G$ with the Haar measure $d g$ for which the integral formula
\begin{equation}\label{eq:integarl-formula-for-integrating-over-B}
  \int_{x \in B}
  f(x) \, d x
  = \int_{g \in G}
  \left(\int_{z \in k^\times}
  |\nr(z g)|^2 f(z g) \, \frac{d z}{|z|}\right) \, d g
\end{equation}
holds for $f \in C_c(B)$.  When $B = M_2(k)$ is split, so that $G = \PGL_2(k)$, a direct calculation with differential forms gives for $f \in C_c(G)$ that
\begin{equation}\label{eqn:explicit-M2-integral-formula}
  \int_G f
  =
  \int_{x_1,x_2 \in k}
  \int_{y \in k^\times}
  f(
  n'(x_1)
  n(x_2)
  a(y)
  )
  \, d x_1 \, d x_2 \, \frac{d y}{|y|}.
\end{equation}

\subsubsection{Volume formulas\label{sec:local-vol-formulas}}
\label{sec-2-1-3}
Assume (for all but the final assertion of \S\ref{sec-2-1-3}) that $k$ is non-archimedean.  Write $\vol(E \subseteq X)$ to denote the volume of $E$ with respect to the measure that we have defined on $X$.  Let $J \leq G$ denote the image of $R^\times$; if $B$ is split, then $J$ is a maximal compact subgroup of $G$, otherwise it has index $2$ in the compact group $G$.  Abbreviate $\vol(\mathfrak{o}) := \vol(\mathfrak{o} \subseteq k)$, $\vol(\mathfrak{o}^\times ) := \vol(\mathfrak{o}^\times \subseteq k^\times )$, $\vol(J) := \vol(J \subseteq G)$, $\vol(R) := \vol(R \subseteq B)$ and $\Delta := \Delta_{\psi}$.  Let $\Delta_{B}$ denote the reduced discriminant, thus $\Delta_{B_\mathfrak{p}} = 1$ or $q$ according as $B$ splits or ramifies.  Set
\begin{equation*}
  \zeta_B(s) :=
  \begin{cases}
    \zeta_k(2 s) \zeta_k(2 s - 1)    & \text{ if $B$ splits}, \\
    \zeta_k(2 s)                             & \text{ otherwise.}
  \end{cases}
\end{equation*}
\begin{lemma*}\label{lem:local-vol-formulas}
  ~\begin{enumerate}
  \item[(i)] $\vol(\mathfrak{o}) = \Delta^{-1/2}$, $\vol(\mathfrak{o}^\times) = \zeta_k(1)^{-1} \Delta^{-1/2}$.
  \item[(ii)] $\vol(R) = \Delta_B^{-1} \Delta^{-4/2}$, $\vol(J) = \zeta_k(1) \zeta_B(1)^{-1} \Delta_B^{-1} \Delta^{-3/2}$.
  \item[(iii)]
    If $B$ is split, then
    \[\frac{\vol(R)}{\vol(J) \Delta^{-1/2}}
      = \zeta_k(2).\]
  \item[(iv)] If $k$ is real, $B$ is non-split and $\psi(x) = e^{2 \pi i x}$, then $\vol(G) = 4 \pi^2$.
  \end{enumerate}
\end{lemma*}
\begin{proof}
  For (i)---(iii), we may reduce by dimensional analysis to the case $\Delta = 1$.  The required formulas then follow from \eqref{eqn:compatibility-squaring-map-local-measure} applied to $f = 1_\mathfrak{o}$ or $f = 1_R$ and by \eqref{eq:integarl-formula-for-integrating-over-B} applied to $f = 1_{1 + \varpi R}$ (see \cite[Lem 2.4.3]{MR580949} for details).  For (iv), set $f(x) := e^{- 2 \pi \nr(x)}$. Apply \eqref{eq:fourier-inversion-for-normalization} to see that $\int_B f = 1$.  Apply \eqref{eq:integarl-formula-for-integrating-over-B} and the substitution $z \mapsto z / (2 \pi \nr(g))^{1/2}$ to deduce that $(2 \pi)^2 = \vol(G) \int_{z \in \mathbb{R}^\times} |z|^4 e^{-|z|^2} \, \frac{d z}{|z|} = \vol(G)$.
\end{proof}

\subsubsection{Cartan
  decomposition}\label{sec:cartan-decomposition}
Suppose $B = M_2(k)$, so that $G = \PGL_2(k)$.  Let $K \leq G$ be the standard maximal compact subgroup.  Then $G = \cup_{y \in k^\times : |y| \leq 1} K a(y) K$.  When $k$ is non-archimedean, one has for $f \in C_c(K \backslash G / K)$ the integral formula
\begin{equation}\label{eq:cartan-decomp-integral-formula}
  \int_{G}
  f
  =
  \vol(K)
  \sum_{m \geq 0}
  q^m
  (1 + 1_{m > 0} q^{-1})
  f(a(\varpi^m)).
\end{equation}

\subsubsection{The $\Xi$-function\label{sec:local-Xi}}
\label{sec-2-1-4}
Given a maximal compact subgroup $K \leq G$, let $\Xi : G \rightarrow \mathbb{R}_{>0}$ denote the Harish--Chandra function relative to $K$:
\begin{itemize}
\item If $B$ is non-split, then $\Xi \equiv 1$.
\item If $B$ is split, then $\Xi(g) = \langle g v, v \rangle$, where $v$ is a $K$-invariant unit vector in the unitary induction of the trivial character of a Borel subgroup of $G$ (see \cite{MR946351}).
\end{itemize}
The following properties of $\Xi$ are relevant for us:
\begin{enumerate}
\item It satisfies $\Xi(1) = 1$, and is bi-$K$-invariant.
\item If $B$ is split, then under any fixed identification $G = \PGL_2(k)$, one has $\Xi(a(y)) \asymp \log(t)/t^{1/2}$ with $t := |y| + |y|^{-1}$.
\item Let $\pi$ be an irreducible unitary representation of $G$.  If $B$ is split, assume that $\dim(\pi) > 1$.  Then there exists $\delta > 0$ so that for $v_1, v_2 \in \pi$, one has $\langle g v_1, v_2 \rangle \ll_{v_1,v_2} \Xi(g)^{\delta}$ for all $g \in G$.  (See for instance \cite[\S2.5.1]{michel-2009} for a more precise assertion).
\end{enumerate}

\subsubsection{Convergence
  lemmas\label{sec:local-convergence-lemmas}}
We record some estimates that follow readily from the Cartan decomposition for $G$.

\begin{lemma}\label{lemma:cheap-matrix-coeff-schwartz-space-B-estimate-via-Xi}
  Either let $X$ be one of the spaces $B^0, B$ and take $\phi_1, \phi_2 \in \mathcal{S}(X)$, or let $X = G$ and take $\phi_1,\phi_2 \in C_c^\infty(G)$.  For $g \in G$, one then has
  \[\langle \Ad(g) \phi_1, \phi_2 \rangle_{L^2(X)}
    \ll_{\phi_1,\phi_2} \Xi(g)^2.\]
\end{lemma}

\begin{lemma}\label{lemma:convergence-Xi-along-G-and-H}
  Let $\delta > 0$.
  \begin{enumerate}
  \item The integral $\int_{g \in G} \Xi^{2+\delta}(g) \,d g$ converges.
  \item Let $E$ be a separable quadratic subalgebra of $B$.  Let $H \leq G$ denote the image of $E^\times$.  Equip $H$ with some Haar measure.  Then the integral $\int_{h \in H} \Xi^\delta(h) \, d h$ converges.
  \end{enumerate}
\end{lemma}

\subsubsection{Conventions}\label{sec:35ac3e56e6}
By a \emph{representation} of a $k$-group, we always mean
\begin{itemize}
\item a smooth representation, if $k$ is non-archimedean, and otherwise
\item the space of smooth vectors in a unitary representation.
\end{itemize}

\subsection{Weil representations\label{sec:local-weil-reps}}
\label{sec-2-2}

\subsubsection{Quadratic spaces\label{sec:local-quadratic-spaces}}
\label{sec-2-2-1}

Let $V$ be a quadratic space over $k$, thus $V$ is a finite-dimensional $k$-vector space equipped with a non-degenerate quadratic form $q_V : V \times V \rightarrow k$.  We denote by $b_V : V \otimes V \rightarrow k$ the associated non-degenerate bilinear form given by
\begin{equation*}
  b_V(x,y) := q_V(x+y) - q_V(x) - q_V(y).
\end{equation*}

Recall that $\GO(V) \leq \GL(V)$ consists of all $g \in \GL(V)$ for which there exists $\lambda \in k^\times$ with $q_V(g x) = \lambda q_V(x)$ for all $x \in V$, $\O(V) \leq \GO(V)$ is the subgroup on which $\lambda = 1$, and $\SO(V) = \SL(V) \cap \O(V)$.  The group $\GO(V)$ contains the subgroup $k^\times$ of scalar operators, and we set $\PGO(V) := \GO(V) / k^\times$.

Let $\mu_V$ denote the measure on $V$ that is $(\psi,b_V)$-self dual, i.e., that for which $\mathcal{F} : \mathcal{S}(V) \rightarrow \mathcal{S}(V)$ defined by
\begin{equation*}
  \mathcal{F} \phi(\xi) := \int_{x \in V} \phi(x) \psi(b_V(x,\xi)) \, d \mu_V(x) \,d x
\end{equation*}
satisfies $\mathcal{F} \mathcal{F} \phi(x) = \phi(-x)$.

The following examples of quadratic spaces are relevant for us:
\begin{enumerate}
\item $V = B$, $q_V = \nr$, so that $b_V(x,\xi) = \tr(x^{\iota} \xi) = \langle x,\xi \rangle$.
\item $V = B^0$, $q_V$ the restriction of $\nr$.  The natural map $\Ad : G \rightarrow \SO(B^0)$
  is an isomorphism.
\item $V = k$, regarded as a subspace of $B$, and $q_V$ the restriction of $\nr$, thus $q_V(x) = x^2$ and $b_V(x,y) = 2 x y$ for $x \in V$.  In this case, we denote the orthogonal group by $\O_1(k) := \O(V) \cong \{\pm 1\}$.
\end{enumerate}
For $V = B, B_0$, the measure $d \mu_V(x)$ coincides with $d x$ as defined in \S\ref{sec:local-measures}.

\subsubsection{Metaplectic group} \label{sec-2-2-2}
Let $\Mp_2(k)$ denote the metaplectic double cover of $\SL_2(k)$.  It is convenient to identify $\Mp_2(k)$ with $\SL_2(k) \times \mu_2$, where $\mu_2 := \{\pm 1\}$, with the group law given by
\begin{equation*}
  (s_1,\zeta_1) (s_2,\zeta_2) = (s_1 s_2, \zeta_1 \zeta_2 c(s_1,s_2))
\end{equation*}
for a cocycle $c : \SL_2(k) \times \SL_2(k) \rightarrow \{\pm 1\}$ as in \cite[p.19]{MR0424695} or \cite[\S4.4]{nelson-theta-squared}.  Thus $\mu_2$ is a central subgroup of $\Mp_2(k)$, and one has a short exact sequence
\begin{equation*}
  1 \rightarrow \mu_2 \rightarrow \Mp_2(k) \xrightarrow{\pr} \SL_2(k) \rightarrow 1.
\end{equation*}
\subsubsection{Weil representation\label{sec:local-weil-repn}}
\label{sec-2-2-3}
For a quadratic space $V$, one has the Weil representation \cite{MR0165033} on the Schwartz--Bruhat space $\mathcal{S}(V)$:
\begin{equation*}
  \rho_{\Weil}^{\psi,V} : \Mp_2(k)
  \times \O(V) \rightarrow \GL(\mathcal{S}(V)).
\end{equation*}
This representation is continuous \cite[\S39]{MR0165033} for the standard topologies on all spaces involved and extends to a unitary representation on $L^2(V) := L^2(V,\mu_V)$.

For the remainder of \S\ref{sec:local-weil-repn}, abbreviate $\rho := \rho_{\Weil}^{\psi,V}$.  For $s \in \Mp_2(k)$ or $g \in \O(V)$, we abbreviate $\rho(s) := \rho(s,1)$ and $\rho(g) := \rho(1,g)$; one then has $\rho(s) \rho(g) = \rho(g) \rho(s)$.

Elements $\zeta$ of the central subgroup $\mu_2$ of $\Mp_2(k)$ act by the scalar operators $\rho(\zeta) = \zeta^{\dim(V)}$, so $\rho$ factors through $\SL_2(k)$ if and only if $\dim(V)$ is even.

There is a quartic character $\chi_{\psi,V} : k^\times \rightarrow \mathbb{C}^{(1)}$ and an eighth root of unity $\gamma_{\psi,V} \in \mathbb{C}^{(1)}$ so that, abbreviating $\rho(s) := \rho( (s,1))$ for $s \in \SL_2(k)$, one has for $\phi \in \mathcal{S}(V)$ and $x \in V$ that
\begin{align*}
  \rho(n(b))
  \phi(x)
  &= 
    \psi(b q_V(x)) \phi(x),
  \\
  \rho(t(a))
  \phi(x)
  &= 
    \chi_{\psi,V}(a) |a|^{\dim(V)/2} \phi(a x),
  \\
  \rho(w) \phi(x)
  &=
    \gamma_{\psi,V} \mathcal{F} \phi(x).
\end{align*}
If $V = M_2(k)$, then $\chi_{\psi,V}$ is trivial and $\gamma_{\psi,V} = 1$.

Elements $g$ of the orthogonal group $\O(V)$ act by $\rho(g) \phi(v) := \phi(g^{-1} v)$.  Suppose that $V = B^0$, so that $\Ad : G \xrightarrow{\cong} \SO(B^0)$.  For $g \in G$ and $\phi \in \mathcal{S}(B^0)$, the function $\Ad(g) \phi$ as defined in \S\ref{sec:general-notation} agrees with $\rho(\Ad(g)) \phi$: both send $x \in B^0$ to $\phi(g^{-1} x g)$.
\subsubsection{Factorization\label{sec:factorization-weil-repn}}
\label{sec-2-2-4}
Let $V$ be a quadratic space that admits an orthogonal decomposition $V = V' \oplus V''$ as a sum of two quadratic spaces.  (The relevant example is when $V = B, V' = k, V'' = B^0$.)

Recall the dense inclusion $\mathcal{S}(V') \otimes \mathcal{S}(V'') \rightarrow \mathcal{S}(V)$ obtained by identifying $\phi ' \otimes \phi '' \in \mathcal{S}(V') \otimes \mathcal{S}(V'')$ with the function $V' \oplus V'' \ni \alpha ' + \alpha '' \mapsto \phi '(\alpha ') \phi '' (\alpha '')$.

Given continuous linear operators $T, T', T''$ on $\mathcal{S}(V), \mathcal{S}(V'), \mathcal{S}(V'')$, respectively, write $T = T' \otimes T''$ to denote that $T (\phi ' \otimes \phi '') = T' \phi ' \otimes T'' \phi ''$ for all $\phi ' \in \mathcal{S}(V'), \phi '' \in \mathcal{S}(V'')$.  In this sense, one has $\rho_{\Weil}^{\psi, V}(s) = \rho_{\Weil}^{\psi, V'}(s) \otimes \rho_{\Weil}^{\psi, V''}(s)$ for all $s \in \Mp_2(k)$.

We denote by $1 \otimes \rho_{\Weil}^{\psi, V''}(s)$ the operator on $\mathcal{S}(V)$ sending $\phi ' \otimes \phi ''$ to $\phi ' \otimes \rho_{\Weil}^{\psi,V''}(s) \phi ''$.

\subsubsection{Extension to similitudes\label{sec:defn-local-omega}}
\label{sec-2-2-5}
The following definitions were inspired by \cite[I.3]{MR783511}.  Let $\Omega$ denote the space of functions $\phi : k^\times \times B \rightarrow \mathbb{C}$ satisfying the conditions:
\begin{itemize}
\item For each $t \in k^\times$, the function $\phi[t] : B \rightarrow \mathbb{C}$ given by $\phi[t](x) := \phi(t,x)$ belongs to the Schwartz--Bruhat space $\mathcal{S}(B)$.
\item One has $\phi(z^2 t, x) = \phi(t, z x)$ for all $t,z \in k^\times$, $x \in B$.
\end{itemize}
Let $\rho_{\Weil} : \PGL_2(k) \times \PGO(B) \rightarrow \GL(\Omega)$ denote the representation characterized by the identities: for $s \in \SL_2(k), y \in k^\times, g \in \GO(B)$,
\begin{align*}
  (\rho_{\Weil}(s) \phi)[t]
  &= \rho_{\Weil}^{\psi^t,B}(s) (\phi[t]),
  \\
  (\rho_{\Weil}(a(y)) \phi)[t]
  &=
    |y| \phi[t y],
  \\
  (\rho_{\Weil}(g) \phi)(t,x)
  &=
    \phi(\lambda(g) t, g^{-1} x)
\end{align*}
where $\lambda : \GO(B) \rightarrow k^\times$ denotes the similitude factor.

\begin{remark*}
  More generally, if $V$ is any even-dimensional quadratic space, then the representation $\rho_{\Weil}^{\psi,V}$ factors through $\SL_2(k) \times \O(V)$.  One can induce it to a representation of $\GL_2(k) \times \GO(V)$ on $\mathcal{S}(k^\times \times V)$, whose isomorphism class is independent of $\psi$.  By taking coinvariants for the action by the center, one arrives at a representation of $\PGL_2(k) \times \PGO(V)$.  In the relevant case that $V = B$, the representation obtained in that way is realized by $\Omega$.  Our global discussion concerns the restriction of $\Omega$ to $\SL_2(k) \times \O_1(k) \times \O(B^0)$, which embeds as the ``even subspace'' of $\oplus_{\tau \in k^{\times} / k^{\times 2}} \rho_{\Weil}^{\psi^{\tau},F} \otimes \rho_{\Weil}^{\psi^{\tau},B^0}$.
\end{remark*}

Equip $\Omega$ with the invariant hermitian norm $\|.\|_{\Omega}$ given by
\begin{equation}\label{eqn:inner-product-on-Omega-1}
  \|\phi\|^2_{\Omega}
  :=
  \int_{t \in k^\times / k^{\times 2}}
  |t|^2
  \int_{x \in B}
  |\phi|^2(t,x) \, d x
  \, d _2 ^\times t,
\end{equation}
or equivalently (by \eqref{eq:integarl-formula-for-integrating-over-B}, \eqref{eqn:compatibility-squaring-map-local-measure}),
\begin{equation}\label{eqn:inner-product-on-Omega-2}
  \|\phi\|^2_{\Omega}
  = 
  \int_{g \in G}
  |\nr(g)|^2
  \int_{t \in k^\times}
  |t|^2
  |\phi|^2(t,g)
  \, \frac{d t}{|t|} \, d g.
\end{equation}

Define $\mathfrak{S} : \Omega \rightarrow \Omega$ and $\Ad(g) : \Omega \rightarrow \Omega$ ($g \in G$) by applying the general definition (\S\ref{sec:general-notation}) to the second coordinate, so that for $\phi \in \Omega$ and $(t,x) \in k^\times \times B$, one has
\begin{equation*}
  (\mathfrak{S} \phi)[t] = \mathfrak{S} (\phi[t]), \qquad \mathfrak{S} \phi(t,x) = (\phi(t,x) + \phi(t,\tr(x)- x))/2,
\end{equation*}
\begin{equation*}
  \Ad(g) \phi = \rho_{\Weil}(\Ad(g)) \phi, \qquad (\Ad(g) \phi)[t] = \Ad(g) (\phi[t]),
\end{equation*}
\begin{equation*}
  \Ad(g) \phi(t,x) = \phi(t,g^{-1} x g).
\end{equation*}

\subsubsection{The distinguished element}\label{sec:dist-elem}
Suppose temporarily that $k$ is non-archimedean and that $B \cong M_2(k)$ is split; similar considerations apply to non-split $B$, but we do not need them.  The \emph{distinguished element} $\phi^0 \in \Omega$ (with respect to the chosen maximal order $R \subset B$) is then defined by
\begin{equation}
  \phi^0(t,x)
  :=
  \frac{
    \int_{z \in k^\times}
    1_{R}(z x)
    1_{\mathfrak{o}^\times}(z^{-2} t)
    \, \frac{d z}{|z|}
  }
  {
    \int_{z \in k^\times}
    1_{\mathfrak{o}^\times}(z)
    \, \frac{d z}{|z|}
  }.
\end{equation}
Note that $\phi^0$ takes values in $\{0,1\}$.  Let $K' \leq \PGO(B)$ denote the image of the $\O(B)$-stabilizer of $R$.  One verifies directly that
\begin{enumerate}[(i)]
\item $\phi^0$ is $K'$-invariant,
\item $\mathfrak{S} \phi^0 = \phi^0$, and
\item if $\psi$ is unramified, then $\phi^0$ is invariant under $\PGL_2(\mathfrak{o}) \leq \PGL_2(k)$.
\end{enumerate}

\begin{lemma}\label{lem:norm-of-distinguished-vector-in-Omega-local}
  $\|\phi^0\|^2_{\Omega} = \vol(R)$.
\end{lemma}
\begin{proof}
  Since $\phi$ takes values in $\{0 ,1\}$, one has
  \begin{equation*}
    \|\phi^0\|^2_{\Omega} = \int_{t \in k^\times / k^{\times 2}} |t|^2 \int_{x \in B} \phi^0(t,x)
    \, d x \, d_2^\times t.
\end{equation*}
By expanding the definition of $\phi^0$ and using that
\begin{equation*}
  \int_{x \in B} 1_R(z x) \,d x = |z|^{-4} \vol(R)
\end{equation*}
and
\begin{equation*}
  |t|^2 |z|^{-4} 1_{\mathfrak{o}^\times}(z^{-2} t) = 1_{\mathfrak{o}^\times}(z^{-2} t),
\end{equation*}
our task reduces to showing that
\begin{equation*}
\int_{t \in k^\times / k^{\times 2}} \int_{z \in k^\times} 1_{\mathfrak{o}^\times}(z^{-2} t) \, \frac{d z}{|z|} \, d_2^\times t = \int_{x \in k^\times} 1_{\mathfrak{o}^\times}(x) \, \frac{d x}{|x|},
\end{equation*}
as follows from \eqref{eqn:compatibility-squaring-map-local-measure}.
\end{proof}

Let $K \leq G$ denote the image of $R^\times$.  We may then fix an identification $B = M_2(k)$ under which $G = \PGL_2(k)$, $R = M_2(\mathfrak{o})$, $K = \PGL_2(\mathfrak{o})$.
\begin{lemma}\label{lem:formula-for-how-Ad-acts-on-distinguish-element-inner-products}
  Let $\phi_1, \phi_2 \in \mathbb{C} \phi^0$.  Let $g \in K a(\varpi^m) K$ for some $m \in \mathbb{Z}_{\geq 0}$ (see \S\ref{sec:cartan-decomposition}).  Then $\langle \Ad(g) \phi_1, \phi_2 \rangle_{\Omega} = q^{-m} \langle \phi_1, \phi_2 \rangle_{\Omega}$.
\end{lemma}
\begin{proof}
  We expand the definitions and use that $\vol(g R g^{-1} \cap R) = q^{-m} \vol(R)$.
\end{proof}

\subsection{Generic representations of
  $\operatorname{PGL}_2$}\label{sec-2-3}
We refer to \cite[\S4.4, \S4.6]{MR1431508} for details on and proofs of the facts collected here.  Let $\pi$ be an irreducible representation of $\PGL_2(k)$.  Recall that $\pi$ is \emph{generic} if it admits a Whittaker model $\mathcal{W}(\pi,\psi)$, consisting of $W : \PGL_2(k) \rightarrow \mathbb{C}$ satisfying $W(n(x) g) = \psi(x) W(g)$.  It then admits a Kirillov model $\mathcal{K}(\pi,\psi)$, consisting of $W : k^\times \rightarrow \mathbb{C}$ of the form $W(y) := W'(a(y))$ for some $W' \in \mathcal{W}(\pi,\psi)$.  The vector space $\mathcal{K}(\pi,\psi)$ is independent of $\psi$ and contains $C_c^\infty(k^\times)$.  Recall that $\pi$ is \emph{unramified} if the space $\pi^{\PGL_2(\mathfrak{o})}$ of $\PGL_2(\mathfrak{o})$-invariant vectors in $\pi$ is nonzero, and that in that case, $\dim(\pi^{\PGL_2(\mathfrak{o})}) = 1$.

Suppose for the remainder of \S\ref{sec-2-3} that $\pi$ is generic and unramified.  Let $\psi^0$ be an unramified unitary character of $k$.  There is then a unique $\PGL_2(\mathfrak{o})$-invariant vector $W_{\pi}^0$ in the Kirillov model $\mathcal{K}(\pi,\psi^0)$ of $\pi$ for which $W_{\pi}^0(1) = 1$.  There is a unique unordered pair $\{\alpha, \beta \}$ of complex numbers, the \emph{Satake parameters} of $\pi$, so that for $y \in k^\times$ with $|y| = q^{-n}$,
\begin{equation}\label{eq:explicit-formula-W-pi-0}
  W_{\pi}^0(y)
  =
  |y|^{1/2}
  \sum _{\substack{
      i, j \in \mathbb{Z}_{\geq 0}:
      i + j = n
    }
  }
  \alpha^{i} \beta^j
  =
  1_{\mathfrak{o}^\times}(y) |y|^{1/2}
  \frac{\alpha^{n+1} - \beta^{n+1}}{\alpha - \beta }.
\end{equation}
One has in general $\alpha \beta = 1$; if moreover $\pi$ is unitary, then either $|\alpha| = |\beta| = 1$ or $\alpha,\beta \in (-q^{1/2}, q^{1/2}) \subseteq \mathbb{R}$.  The \emph{adjoint $L$-factor} is defined for $s \in \mathbb{C}$ by
\[
  L(\ad \pi,s) := (1 - \alpha \beta^{-1} q^{-s})^{-1} (1 - q^{-s})^{-1} (1 - \alpha^{-1} \beta q^{-s})^{-1}.
\]
We have the following standard geometric series evaluation (see, e.g., \cite[Prop 3.8.1]{MR1431508}, taking into account that we have normalized measures differently).
\begin{lemma*}
  If $\pi$ is unitary and $\Re(s) \geq 0$, then $L(\ad \pi, 1+ s)$ is finite, and one has the identity
  \begin{equation}\label{eq:local-computation-norm-of-whittaker-newvector}
    \int_{y \in k^\times}
    |W_{\pi}^0(y)|^2
    |y|^s
    \, \frac{d y}{|y|}
    = \frac{L(\ad \pi,1 + s)}{\zeta_k(2 + 2 s)}
    \Delta_{\psi}^{-1/2}
    \frac{\zeta_k(1 + s)}{\zeta_k(1)}
  \end{equation}
  in which the left hand side converges absolutely.
\end{lemma*}

\subsection{Representations of $G$}
\label{sec-2-4}
Let $\pi$ be an irreducible representation of $G$.  Define a compact open subgroup $J \leq G$ in the following two cases:
\begin{itemize}
\item if $k$ is non-archimedean, take for $J \leq G$ the image of the unit group $R^\times$ of the chosen maximal order $R \subseteq B$;
\item if $k$ is real and $B$ is non-split, set $J := G$.
\end{itemize}
In either case, set
\begin{equation*}
  \vol(J) := \int_{g \in G} 1_J(g) \, d g,
\end{equation*}
\begin{equation*}
  e_J := \vol(J)^{-1} 1_J \in C_c^\infty(G).
\end{equation*}

\subsubsection{Hecke kernels and theta kernels\label{sec:hecke-kernels-local}}
\label{sec-2-4-1}
Assume that $k$ is non-archimedean.
For $y \in k^\times$, the \emph{normalized Hecke kernel} $T_y \in C_c^\infty(J \backslash G / J)$ is defined to be the element with the property that $|y|^{-1} \vol(J) T_y$ is the characteristic function of the image in $G$ of the subset $\{b \in R : |\nr(b)| = |y| \}$ of $B^\times$.
For example, if $y \in \mathfrak{o}^\times$, then $T_y = e_J$.
\begin{lemma*}\label{lem:relation-distinguished-elt-hecke-kernel}
  Let $y \in k^\times$, $g \in G$.  Choose $\tilde{g} \in B^\times$ with image $g$.  Then
  \[
    \rho_{\Weil}(a(y)) \phi^0(\nr(\tilde{g})^{-1}, \tilde{g}) = |y| \phi^0(y \nr(\tilde{g})^{-1}, \tilde{g}) = \vol(J) T_y(g)
  \]
  where $\phi^0 \in \Omega$ is the distinguished element (\S\ref{sec:dist-elem}).
\end{lemma*}
\begin{proof}
  We must verify that
  \begin{equation}\label{eq:formula-for-hecke-kernel}
    |y|^{-1} \vol(J) T_y(g) =
    \frac{
      \int_{z \in k^\times}
      1_R(z \tilde{g}) 1_{\mathfrak{o}^\times}(y \nr(z \tilde{g})^{-1})
      \, \frac{d z}{|z|}
    }
    {
      \int_{z \in k^\times}
      1_{\mathfrak{o}^\times}(z)
      \, \frac{d z}{|z|}
    }.
  \end{equation}
  Let $g \in G$.  The right hand side of \eqref{eq:formula-for-hecke-kernel} is independent of $\tilde{g}$, and both sides take values in $\{0,1\}$.
  The right hand side of \eqref{eq:formula-for-hecke-kernel} is nonzero iff its integrand is nonzero for some $z \in k^\times$, i.e., iff for some $z \in k^\times$ the element $b := z \tilde{g}$ lies in $R$ and $|\nr(b)| = |y|$, i.e., iff the left hand side of \eqref{eq:formula-for-hecke-kernel} is nonzero.

\end{proof}

\subsubsection{Hecke functionals and standard $L$-factors}
\label{sec-2-4-2}
Continue to assume that $k$ is non-archimedean.  Recall that $\pi$ is \emph{unramified} if the space $\pi^{J}$ of $J$-invariant vectors in $\pi$ is nonzero; it is known then that $\dim(\pi^J) = 1$.

Suppose for remainder of \S\ref{sec-2-4-2} that $\pi$ is unramified.
There is then a unique functional $\lambda_\pi : C_c^\infty(J \backslash G / J) \rightarrow \mathbb{C}$ so that $\pi(T) v = \lambda_\pi(T) v$ for all $T \in C_c^\infty(J \backslash G / J), v \in \pi^J$.  We may evaluate this functional on the elements $T_y$ attached above to $y \in k^\times$:
\begin{itemize}
\item If $B$ is split, then there is a unique unordered pair $\{\alpha,\beta\}$ of complex numbers (the \emph{Satake parameters}) satisfying $\alpha \beta = 1$ so that $\lambda_\pi(T_y) = 0$ unless $|y| = q^{-n}$ with $n \geq 0$, in which case   (see, e.g., \cite[\S4.6]{MR1431508})
  \begin{equation}\label{eq:explicitf-romula-lambda-pi-tY}
    \lambda_\pi(T_y)
    = 
    |y|^{1/2}
    \sum _{\substack{
        i, j \in \mathbb{Z}_{\geq 0}:
        i + j = n
      }
    }
    \alpha^{i} \beta^j
    =
    1_{\mathfrak{o}^\times}(y)
    |y|^{1/2}
    \frac{\alpha^{n+1} - \beta^{n+1}}{\alpha - \beta }.
  \end{equation}

\item If $B$ is non-split, then there is a unique unramified quadratic character $\eta$ of $k^\times$ so that $\lambda_\pi(T_y) = |y| \eta(y)$.
\end{itemize}
The \emph{standard $L$-factor} is then the meromorphic function defined for $s \in \mathbb{C}$ by
\[
  L(\pi,s) := \begin{cases}
    (1 - \alpha q^{-s})^{-1} (1 - \beta q^{-s})^{-1} & \text{ if $B$ is split}, \\
    (1 - \eta(\varpi) q^{-s-1/2})^{-1} & \text{ if $B$ is non-split}.
  \end{cases}
\]
\subsubsection{The local Jacquet--Langlands correspondence}
\label{sec-2-4-3}
The Jacquet--Langlands lift $\pi_{\JL}$ of $\pi$ is an irreducible representation of $\PGL_2(k)$ attached to $\pi$.  The following properties of the association $\pi \mapsto \pi_{\JL}$ are relevant for us:
\begin{itemize}
\item $\pi_{\JL}$ is generic if (and only if) either
  \begin{itemize}
  \item $B$ is non-split, or
  \item $B$ is split and $\dim(\pi) > 1$.
  \end{itemize}
\item If $B$ is split, then $\pi_{\JL}$ corresponds to $\pi$ under the isomorphism $G \cong \PGL_2(k)$.  In particular, if $\pi$ is unramified, then so is $\pi_{\JL}$, and Satake parameters (see \S\ref{sec-2-3}, \S\ref{sec-2-4-2}) are preserved.
\end{itemize}
Assume now that that $k$ is non-archimedean, that $B$ is split, that $\pi$ is unramified, and that $\dim(\pi) > 1$.  Then $\pi_{\JL}$ is generic and unramified.
Let $W^0_\pi : k^\times \rightarrow \mathbb{C}$ denote the function attached to $\pi_{\JL}$ in \S\ref{sec-2-3}.  By \eqref{eq:explicit-formula-W-pi-0} and \eqref{eq:explicitf-romula-lambda-pi-tY},
\begin{equation}\label{eqn:key-local-identity-relating-whittaker-and-hecke}
  W_\pi^0(y) = \lambda_\pi(T_y).
\end{equation}

\subsubsection{Local integrals}
\label{sec:local-integrals-for-rallis-ipf}

Assume first that $k$ is non-archimedean, that $\pi$ is unramified, and that $\pi$ is unitary.
Retain the notation of \S\ref{sec-2-4-2}.
\begin{lemma}\label{lem:local-rallis-ipf-unram-calc}
  Suppose that $B$ is split and that $\dim(\pi) > 1$.
  \begin{enumerate}
  \item[(i)] $|\alpha|, |\beta| < q^{1/2}$.  In particular, $L(\pi,\tfrac{1}{2})$ is finite.
  \item[(ii)] Let $\phi_1, \phi_2$ belong to the line $\mathbb{C} \phi^0$ spanned by the distinguished element $\phi^0 \in \Omega$.  Let $v_1, v_2 \in \pi^J$.  Then the identity
    \begin{equation}\label{eqn:local-rallis-ipf-unram-calc}
      \int_{g \in G}
      \langle \Ad(g) \phi_1, \phi_2 \rangle
      \langle g v_1, v_2 \rangle \,d g 
      =
      \frac{L(\pi,\frac{1}{2})}{\zeta_k(2)}
      \vol(J) 
      \langle \phi_1, \phi_2 \rangle
      \langle v_1, v_2 \rangle
    \end{equation}
    holds, with the left hand side converging absolutely.
  \end{enumerate}
\end{lemma}
\begin{proof}
  For (i), see \cite[Thm 4.6.7]{MR1431508}.  For (ii), the convergence follows from \S\ref{sec:local-convergence-lemmas}.
  Let $\{\alpha,\beta\}$ denote the Satake parameters of $\pi$ and set $t_1 := \alpha q^{-1/2}, t_2 := \beta q^{-1/2}$, so that $L(\pi,\tfrac{1}{2})^{-1} = (1 - t_1) (1-t_2)$.  The Macdonald formula \cite[Thm 4.6.6]{MR1431508} says that $\langle g v_1, v_2 \rangle = (u_1 t_1^m + u_2 t_2^m) \langle v_1, v_2 \rangle$, where
  \[
    u_1 := \frac{1}{1 + q^{-1}}\frac{1 - q^{-1} \beta/\alpha }{1 - \beta / \alpha }, \quad u_2 := \frac{1}{1 + q^{-1}}\frac{1 - q^{-1} \alpha / \beta }{1 - \alpha / \beta }.
  \]
  By the Cartan decomposition and Lemma \ref{lem:formula-for-how-Ad-acts-on-distinguish-element-inner-products} of \S\ref{sec:dist-elem}, we obtain
  \begin{equation*}
  \int_{g \in G} \langle \Ad(g) \phi_1, \phi_2 \rangle \,d g 
    \langle g v_1, v_2 \rangle = \vol(J) 
    \langle \phi_1, \phi_2 \rangle \langle v_1, v_2 \rangle \Sigma,
  \end{equation*}
  where $\Sigma := \sum_{i=1,2} \sum_{m \geq 0} (1 + 1_{m>0} q^{-1}) t_i^m$.  We compute that $\sum_{i=1,2} u_i (1 + q^{-1} t_i) ( 1 - t_i)^{-1} = L(\pi,\tfrac{1}{2}) \Sigma '$ with $\Sigma ' := \sum_{i=1,2} u_i (1 + q^{-1} t_i) ( 1 - t_{i'})^{-1}$, $\{i, i'\} = \{1,2\}$.  Direct calculation gives $\Sigma ' = \zeta_k(2)^{-1}$, as required.
\end{proof}

Suppose now that $B$ is non-split, so that $\pi$ is the one-dimensional representation corresponding to the character $G \ni g \mapsto \eta(\nr(g)) \in \{\pm 1\}$, as in \S\ref{sec-2-4-2}.  Let $v_1,v_2 \in \pi$.  Recalling that $[G:J] = 2$, we have
\begin{equation}\label{eqn:local-rallis-ipf-integral-nonsplit-finite}
  \int_{g \in G} \langle \Ad(g) e_J, e_J \rangle_{L^2(G)}
  \langle g v_1, v_2 \rangle \,d g 
  =
  \langle v_1, v_2 \rangle
  \cdot 
  \begin{cases}
    0 & \text{ if $\eta$ is nontrivial,} \\
    2 & \text{ if $\eta$ is trivial.}
  \end{cases}
\end{equation}

Suppose, finally, that $k \cong \mathbb{R}$, that $B$ is non-split, and that $\pi$ is trivial.  For $v_1,v_2 \in \pi$, one then has
\begin{equation}\label{eqn:local-rallis-ipf-integral-nonsplit-real}
  \int_{g \in G} \langle \Ad(g) e_J, e_J \rangle_{L^2(G)}
  \langle g v_1, v_2 \rangle \,d g 
  =
  \langle v_1, v_2 \rangle.
\end{equation}

\section{Global preliminaries}
\label{sec-3}
In this section, we collect those preliminaries for the proof of Theorem \ref{thm:main-estimate-general-variance} whose discussion makes sense independently of that proof.

Let $F$ be a number field with adele ring $\mathbb{A}$, let $B$ be a quaternion algebra over $F$, and let $\psi$ be a nontrivial unitary character of $\mathbb{A}/F$.

\subsection{Generalities}
\label{sec-3-1}
\subsubsection{Notation}
\label{sec-3-1-1}
We denote by $\mathcal{O}_F$, or simply $\mathcal{O}$, the ring of integers in $F$.  We denote by $\mathfrak{p}$ a place of $F$, finite or infinite.  A subscripted $\mathfrak{p}$ denotes completion; for example, $\mathcal{O}_\mathfrak{p}$ denotes the ring of integers of $F_{\mathfrak{p}}$ if $\mathfrak{p}$ is finite.  For a finite set of places $S$, a subscripted $S$ denotes a product taken over $S$, such as in $F_S := \prod_{\mathfrak{p} \in S} F_\mathfrak{p}, B_S := \prod_{\mathfrak{p} \in S} B_\mathfrak{p}$.

The character $\psi$ factors as $\psi(x) = \prod \psi_\mathfrak{p}(x_\mathfrak{p})$, where $\psi_\mathfrak{p}$ is a nontrivial unitary character of $F_\mathfrak{p}$.

For a place $\mathfrak{p}$, let $\zeta_\mathfrak{p} := \zeta_{F_\mathfrak{p}}$ denote the local Euler factor.  Let $\xi_F(s) := \prod \zeta_\mathfrak{p}(s)$ denote the Dedekind zeta function (absolutely convergent for $\Re(s) > 1$) and $\xi_F^*(1) := \res_{s \rightarrow 1} \xi_F(s)$ its residue.  For a finite set $S$ of places that contains the infinite places, let
\begin{equation*}
  \zeta_F^{(S)}(s) := \prod_{\mathfrak{p} \notin S} \zeta_\mathfrak{p}(s)
\end{equation*}
denote the partial Dedekind zeta function.

\subsubsection{Groups}\label{sec:35ac3e56f4}
For an algebraic $F$-group $\mathbf{G}$, we write $G := \mathbf{G}(F), G_\mathfrak{p} := \mathbf{G}(F_\mathfrak{p})$, $G_\mathbb{A} := \mathbf{G}(\mathbb{A})$, $G_S := \mathbf{G}(F_S) = \prod_{\mathfrak{p} \in S} G_\mathfrak{p}$, $[G] := G \backslash G_\mathbb{A}$.  This notation applies to the $F$-group $\bPB^\times$ given by $\bPB^\times(A) := (B \otimes_F A)^\times/A^\times$ and also to the $F$-groups ${{\mathbf P}{\mathbf G}{\mathbf L}}_2$, ${{\mathbf S}{\mathbf L}}_2$.  We similarly abbreviate $[\Mp_2] := \SL_2(F) \backslash \Mp_2(\mathbb{A})$ (see \S\ref{sec:global-metaplectic-gp}).

\subsubsection{Measures\label{sec:global-measures}}
\label{sec-3-1-2}
When $\mathbf{G}$ is semisimple, we equip $G_\mathbb{A}$ and $[G]$ with Tamagawa measures.  Then $\vol([\SL_2]) = 1$ and $\vol([\PGL_2]) = \vol([\PB^\times]) = 2$.  We denote by $\langle , \rangle_{G}$ the corresponding inner product on $L^2([G])$, omitting the subscripted $G$ if it is clear by context.

For each place $\mathfrak{p}$, the character $\psi_\mathfrak{p}$ induces (via the recipe of \S\ref{sec:local-measures}) a Haar measure on $F_\mathfrak{p}$, $B_\mathfrak{p}$, $F_\mathfrak{p}^{\times} / F_{\mathfrak{p}}^{\times 2}$, $\PB^\times_\mathfrak{p}$; we equip $\mathbb{A}, B_\mathbb{A}$, $\mathbb{A}^{\times} / \mathbb{A}^{\times 2}$ and $\PB^\times_\mathbb{A}$ with the corresponding restricted product measures, denoted similarly.  This defines the Tamagawa measure on $\PB^\times_\mathbb{A}$.  The quotient measures on $\mathbb{A}/F$ and $B_\mathbb{A}/B$ are then probability measures.  We likewise equip finite products such as $F_S$ or $\PB^\times_S$ with product measures.

We equip $\mathbb{A}^\times$ with the regularized product of the measures constructed in \S\ref{sec:local-measures}: for a factorizable function $f = \prod f_\mathfrak{p} \in C_c^\infty(\mathbb{A}^\times)$ for which $f_\mathfrak{p} = 1_{\mathcal{O}_\mathfrak{p}^\times}$ for almost all finite primes $\mathfrak{p}$, we set
\[
  \int_{y \in \mathbb{A}^\times} f(y) \, \frac{d y}{|y|} := \frac{1}{\xi_F^*(1)} \prod_{\mathfrak{p}} \zeta_\mathfrak{p}(1) \int_{y \in F_\mathfrak{p}^\times} f_\mathfrak{p}(y) \, \frac{d y}{|y|}.
\]
We thereby obtain a quotient Haar $\frac{d y}{|y|}$ on $\mathbb{A}^\times / F^\times$ whose pushforward under $|.| : \mathbb{A}^\times / F^\times \rightarrow \mathbb{R}^\times_+$ is the standard Haar measure $\frac{d t}{|t|}$ on $\mathbb{R}^\times_+$, where $d t$ denotes Lebesgue measure.

The quotient measure on the discrete group $F^{\times } / F^{\times 2}$ compatible with the squaring map is half the counting measure, i.e., for finitely-supported $f : F^\times \rightarrow \mathbb{C}$, one has
\begin{equation*}
  \sum_{x \in F^\times } f(x) = \frac{1}{2} \sum_{y \in F^\times / F^{\times 2}} (\sum_{z \in F^\times} f(y z^2) ).
\end{equation*}
On $\mathbb{A}^\times / F^\times \mathbb{A}^{\times 2}$, we take the quotient measure $d_2^\times y$ induced by the exact sequence $1 \rightarrow F^\times / F^{\times 2} \rightarrow \mathbb{A}^\times / \mathbb{A}^{\times 2} \rightarrow \mathbb{A}^\times / F^\times \mathbb{A}^{\times 2} \rightarrow 1$, where $F^{\times} / F^{\times 2}$ is equipped with half the counting measure.  Thus for $f \in C_c(\mathbb{A}^\times / \mathbb{A}^{\times 2})$,
\begin{equation}\label{eq:integral-formula-involving-squares}
  \int_{y \in \mathbb{A}^\times / F^\times \mathbb{A}^{\times  2}}
  \frac{1}{2} \sum_{a \in F^\times / F^{\times 2}}
  f(a y)
  \, d_2^\times y
  = \int_{\mathbb{A}^\times / \mathbb{A}^{\times 2}} f.
\end{equation}
By decomposing the Haar measure on $\mathbb{A}^\times$ in two ways, one finds for $f \in C_c(\mathbb{A}^\times / F^\times)$ that $\int_{\mathbb{A}^\times / F^\times} f = \int_{x \in \mathbb{A}^\times / F^{\times } \mathbb{A}^{\times 2}} \int_{y \in \mathbb{A}^\times / F^\times} f(x y^2) \, \frac{d y}{|y|} \, d_2^\times s$;
moreover, $\vol(\mathbb{A}^{\times} / F^\times \mathbb{A}^{\times 2}) = 2$.  Finally, for $f \in C_c^\infty([\PGL_2])$,
\begin{equation}\label{eq:sl2-vs-pgl2-integrals}
  \int_{[\PGL_2]} f
  =
  \int_{y \in \mathbb{A}^\times /  F^\times \mathbb{A}^{\times 2}}
  \int_{s \in [\SL_2]}
  f(s a(y)) \, d s \, d_2^\times y.
\end{equation}

\subsubsection{The $\Xi$-function\label{sec:Xi-global}}
\label{sec-3-1-3}
Fix a maximal compact subgroup $K = \prod K_\mathfrak{p} \leq \PB^\times_{\mathbb{A}}$.  Let $\Xi : \PB^\times_{\mathbb{A}} \rightarrow \mathbb{C}$ be the product $\Xi(g) := \prod \Xi_\mathfrak{p}(g_\mathfrak{p})$ of the functions $\Xi_\mathfrak{p}$ on $\PB^\times_\mathfrak{p}$ attached in \S\ref{sec:local-Xi} to the factors $K_\mathfrak{p}$.

\subsubsection{Conventions}\label{sec:conventions-cusp-forms}
A \emph{cusp form} is a smooth vector in the Hilbert space $L^2_{\cusp}([G])$ of square-integrable cuspidal functions.  A \emph{cuspidal automorphic representation} $\pi$ of $G_\mathbb{A}$ is the space of smooth vectors in an irreducible subrepresentation of $L^2_{\cusp}([G])$.

\subsection{Automorphic forms on $\PGL_2$}
\label{sec-3-3}
\subsubsection{Fourier expansions}\label{sec:aut-forms-fourier-exp}
\label{sec-3-3-1}
Let $\varphi : [\PGL_2] \rightarrow \mathbb{C}$ be a smooth function.  It admits the Fourier expansion
\begin{equation*}
  \varphi(n(x) a(y)) = c_{\varphi}(y) + \sum_{\tau \in F^\times} \psi(\tau x) W_{\varphi}(\tau y),
\end{equation*}
where $c_{\varphi}(y) := \int_{x \in \mathbb{A}/F} \varphi(n(x) a(y)) \,d x$ denotes the constant term and
\begin{equation}
W_{\varphi}(y) := \int_{x \in \mathbb{A}/F} \psi(- x) \varphi(n(x) a(y)) \,d x 
\end{equation}
denotes the (diagonal restriction of) the Whittaker function.  The upper-triangular Borel subgroup of $\PGL_2(\mathbb{A})$ has dense image in $[\PGL_2]$, so $\varphi$ is determined by the values $\varphi(n(x) a(y))$ for $x \in \mathbb{A}, y \in \mathbb{A}^\times$.  Recall that $\varphi$ is \emph{cuspidal} if $c_{\varphi} = 0$; in that case, $\varphi$ is determined by $W_\varphi$.

\subsubsection{Kirillov model}
\label{sec-3-3-2}
Let $\pi \subseteq L^2([\PGL_2])$ be a cuspidal automorphic representation; it is (the smooth completion of) a restricted tensor product $\otimes \pi_\mathfrak{p}$, where $\pi_\mathfrak{p}$ is a generic for every $\mathfrak{p}$ and unramified for almost all finite $\mathfrak{p}$.  Let $\mathcal{K}(\pi,\psi) := \{W_\varphi : \varphi \in \pi \}$.  The natural map $\pi \rightarrow \mathcal{K}(\pi,\psi)$ is a linear isomorphism under which the pure tensors in $\pi$ correspond to the factorizable functions $W(y) = \prod W_\mathfrak{p}(y_\mathfrak{p})$, where $W_\mathfrak{p}$ belongs to the local Kirillov model $\mathcal{K}(\pi_\mathfrak{p},\psi_\mathfrak{p})$ and satisfies $W_\mathfrak{p} = W_{\pi_\mathfrak{p}}^0$ (see \S\ref{sec-2-4-3}) for almost all finite $\mathfrak{p}$.  The following residue calculation is standard and may be derived from, for instance, \cite[Lem 2.2.3]{michel-2009}.
\begin{lemma*}\label{lem:inner-product-formula-petersson-vs-kirillov}
  Let $\varphi \in \pi$.  The integral $I(s) := \int_{y \in \mathbb{A}^\times} |W_\varphi(y)|^2 |y|^s \, \frac{d y}{|y|}$ converges absolutely for complex numbers $s$ with positive real part, extends to a meromorphic function on the complex plane, and satisfies
  \begin{equation}\label{eqn:rs-ipf}
    2 \res_{s \rightarrow 0} I(s) = \|\varphi\|^2.
  \end{equation}
\end{lemma*}

\subsubsection{Adjoint $L$-function}
\label{sec:adjoint-l-function}
For $\pi$, $S$ as in the lemma of \S\ref{sec-3-3-2}, the \emph{partial adjoint $L$-function} is defined for $\Re(s) > 1$ by the absolutely-convergent Euler product $L^{(S)}(\ad \pi,s) := \prod_{\mathfrak{p} \notin S} L(\ad \pi_\mathfrak{p},s)$; it continues meromorphically to the complex plane, and is holomorphic for (at least) $\Re(s) \geq 1$ (see \cite{MR533066}).

\subsection{Automorphic forms on $\PB^\times$}
\label{sec-3-4}

\subsubsection{Jacquet--Langlands lifts}
\label{sec-3-4-1}
Let $\pi = \otimes \pi_\mathfrak{p} \subseteq L^2([\PB^\times])$ be a cuspidal automorphic representation with $\dim(\pi) > 1$.  By \cite[Prop 4]{MR0333081},
\begin{equation}\label{eqn:generic-at-each-place-if-not-1-diml}
  \text{$\dim(\pi_\mathfrak{p}) > 1$ for any prime $\mathfrak{p}$ at
    which $B$ splits.}
\end{equation}
The Jacquet--Langlands lift $\pi_{\JL} = \otimes \pi_{\JL,\mathfrak{p}} \subseteq L^2([\PGL_2])$ is the unique cuspidal automorphic representation for which $\pi_{\JL,\mathfrak{p}} = (\pi_\mathfrak{p})_{\JL}$ for each place $\mathfrak{p}$.  If $\mathfrak{p}$ is a finite prime at which $B$ splits and for which $\pi_\mathfrak{p}$ is unramified, then $(\pi_{\JL})_\mathfrak{p}$ is unramified.  The association $\pi \mapsto \pi_{\JL}$ is injective.

\subsubsection{The pretrace formula\label{sec:pretrace-formula}}
Assume that $B$ is non-split, so that $[\PB^\times]$ is compact.  Fix a maximal compact subgroup $K$ of $\PB^\times_\mathbb{A}$.  The pretrace formula asserts that for $f \in C_c^\infty(\PB_\mathbb{A}^\times)$ and $x \in \PB^\times_\mathbb{A}$,
\begin{equation}\label{eq:pretrace-formula-general}
  \sum_{\pi}
  \sum_{\varphi}
  \overline{\varphi (x)}
  \pi(f) \varphi(x)
  = \sum_{\gamma \in \PB^\times} f(x^{-1} \gamma x),
\end{equation}
where $\pi$ traverses the irreducible subrepresentations of $L^2([\PB^\times])$ and $\varphi$ traverses an orthonormal basis $\mathcal{B}(\pi)$ of $\pi$ consisting of $K$-isotypic vectors.  Only finitely many summands on the right hand side of \eqref{eq:pretrace-formula-general} are nonzero, while the condition on $\mathcal{B}(\pi)$ implies that the left hand side of \eqref{eq:pretrace-formula-general} converges absolutely, or indeed, rapidly: Let
\begin{equation*}
  C(\pi) := \prod_\mathfrak{p} C((\pi_\mathfrak{p})_{\JL}) \in \mathbb{R}_{\geq 1}
\end{equation*}
denote the analytic conductor of $\pi$ (see, e.g.,  \cite[\S3.1.8, \S4.1.4]{michel-2009}).  By (e.g.) the proof of \cite[Thm 9.1]{MR1616155}, one has for each $A \geq 0$
\begin{equation}\label{eq:rapid-convergence-of-pretrace-formula}
  \sum_{\pi} C(\pi)^A
  \sum_{\varphi}
  |\overline{\varphi}(x) \pi(f) \varphi(x)|
  < \infty.
\end{equation}
Note also that there exists $A_0 > 3$ so that (see, e.g., \cite[(2.15)]{michel-2009}).
\begin{equation}\label{eq:polynomial-growth-of-reps}
  \sum_{\pi} C(\pi)^{-A_0}
  < \infty
\end{equation}

Let $S$ be a set of places containing the infinite ones.  Let $R \subseteq B$ be a maximal order.  For each $\mathfrak{p} \notin S$, let $J_\mathfrak{p} \leq \PB^\times_\mathfrak{p}$ denote the image of $R_\mathfrak{p}^\times$, as in \S\ref{sec-2-4}, and set $J := \prod_{\mathfrak{p} \notin S} J_\mathfrak{p}$.  Suppose that $f = f_S \otimes (\otimes_{\mathfrak{p} \notin S} T_{y_\mathfrak{p}})$ for some $f_S \in C_c^\infty(\PB_S^\times)$ and $y \in \mathbb{A}^\times$, where $T_{y_\mathfrak{p}}$ is the Hecke kernel as defined in \S\ref{sec:hecke-kernels-local} relative to $J_\mathfrak{p}$.  The formula \eqref{eq:pretrace-formula-general} then specializes to
\begin{equation}\label{eq:cnsez1neo0}
  \sum_\pi
  (\sum_{\varphi}
  \overline{\varphi(x)}
  \pi(f_S) \varphi(x))
  \prod_{\mathfrak{p} \notin S}
  \lambda_{\pi_\mathfrak{p}}(T_{y_\mathfrak{p}}) 
  =
  \sum_{\gamma \in \PB^\times}
  f_S(x_S^{-1} \gamma x_S)
  \prod_{\mathfrak{p} \notin S}
  T_{y_\mathfrak{p}}(x_\mathfrak{p}^{-1} \gamma x_\mathfrak{p}),
\end{equation}
where $\pi \subseteq L^2([\PB^\times])$ now traverses the subrepresentations that are unramified outside $S$ (i.e., that contain a nonzero $J$-fixed vector) and $\varphi$ traverses an orthonormal basis of $K$-isotypic vectors for the $J$-fixed subspace $\pi^J$ of $\pi$.

\subsubsection{$L$-functions}\label{sec:standard-l-function}
Let $\pi \subseteq L^2([\PB^\times])$ be a cuspidal automorphic representation with $\dim(\pi) > 1$.  Let $S$ be a finite set of places containing all infinite places as well as any places at which either $B$ or $\pi$ ramifies.

The \emph{partial standard $L$-function} is defined for $\Re(s) > 1$ by the absolutely-convergent Euler product $L^{(S)}(\pi,s) := \prod_{\mathfrak{p} \notin S} L(\pi_\mathfrak{p},s)$; it continues meromorphically to the complex plane, and is holomorphic for (at least) $\Re(s) \geq 1/2$ (see, e.g., \cite[\S3.5]{MR1431508}).

Set $L^{(S)}(\ad \pi,s) := L^{(S)}(\ad \pi_{\JL}, s)$ (see \S\ref{sec:adjoint-l-function}).  By \cite{HL94} (cf. \cite[\S 2.9]{2009arXiv0904.2429B}), one has
\begin{equation}\label{eq:HL}
  C(\pi)^{-\eps}  \ll_{\eps} L^{(S)}(\ad \pi,1) \ll_{\eps}
  C(\pi)^{\eps}
  \text{ for each } \eps > 0.
\end{equation}

\subsection{Theta functions}
\label{sec-3-2}
\subsubsection{Metaplectic group\label{sec:global-metaplectic-gp}}
\label{sec-3-2-1}
Let $\Mp_2(\mathbb{A})$ denote the metaplectic double cover of $\SL_2(\mathbb{A})$; it fits into a short exact sequence $1 \rightarrow \mu_2 \rightarrow \Mp_2(\mathbb{A}) \xrightarrow{\pr} \SL_2(\mathbb{A})$.  We may identify it with $\SL_2(\mathbb{A}) \times \mu_2$, with the group law given by $(s_1,\zeta_1) (s_2,\zeta_2) = (s_1 s_2, \zeta_1 \zeta_2 c(s_1,s_2))$, where $c$ is the product of the cocycles from \S\ref{sec-2-2-2}.  We identify $\SL_2(F)$ with its image under the unique splitting $\SL_2(F) \hookrightarrow \Mp_2(\mathbb{A})$.

We may similarly define $\Mp_2(F_S)$ as a double cover of $\SL_2(F_S)$ for any collection $S$ of places of $F$.
\subsubsection{Quadratic spaces}
\label{sec-3-2-2}
We define quadratic spaces $V$ over $F$ as in \S\ref{sec:local-quadratic-spaces}.  The relevant examples are still $V = B, B^0, F$.  We equip $V_\mathbb{A}$ with the $(\psi,b_V)$-self dual measure $\mu_V$.  That measure is the product of the measures $\mu_{V_\mathfrak{p}}$ on the local spaces $V_{\mathfrak{p}}$ attached to $\psi_\mathfrak{p}$, and is independent of $\psi$: it assigns volume one to a fundamental domain for $V_\mathbb{A}/V$.

\subsubsection{Weil representation\label{sec:weil-repn-global}}
\label{sec-3-2-3}
For a quadratic space $V$ over $F$, the Schwartz--Bruhat space $\mathcal{S}(V_\mathbb{A})$ factors as the (completed) restricted tensor product $\mathcal{S}(V_\mathbb{A}) = \otimes \mathcal{S}(V_\mathfrak{p})$.  The Weil representation $\rho_{\Weil}^{\psi,V} : \Mp_2(\mathbb{A}) \times \O(V_\mathbb{A}) \rightarrow \GL(\mathcal{S}(V_\mathbb{A}))$ is given by $\rho_{\Weil}^{\psi,V} = \otimes \rho_{\Weil}^{\psi_{\mathfrak{p}},V_{\mathfrak{p}}}$ in the evident sense.

We may similarly define a Weil representation $\rho_{\Weil}^{\psi,V} : \Mp_2(F_S) \times \O(V_S) \rightarrow \GL(\mathcal{S}(V_S))$ for a finite set $S$ of places of $F$.

\subsubsection{Theta kernels\label{sec:theta-kernels}}
\label{sec-3-2-5}
Let $V$ be a quadratic space over $F$.  For $\phi \in \mathcal{S}(V_\mathbb{A})$, $s \in \Mp_2(\mathbb{A})$ and $g \in \O(V_\mathbb{A})$, set
\begin{equation*}
  \theta_{\psi}(\phi)(s,g) := \sum_{x \in V} \rho_{\Weil}^{\psi,V}(s,g) \phi(x).
\end{equation*}
The sum converges absolutely and defines a smooth function
\begin{equation*}
  \theta_{\psi}(\phi) : [\Mp_2] \times [\O(V)] \rightarrow \mathbb{C}.
\end{equation*}
We employ notation such as $\theta_{\psi}(\phi;s,g) := \theta_{\psi}(\phi)(s,g)$.  Observe that
\begin{equation}\label{eq:equivariance-theta-kernel}
  \theta_{\psi}(\phi;s s', g g')
  = \theta_{\psi}(\rho_{\Weil}^{\psi,V}(s',g') \phi;s,g)
\end{equation}
\subsubsection{Elementary theta functions\label{sec:elem-theta-fns}}
\label{sec-3-2-6}
Let $V = F$, regarded as a quadratic subspace of $B$ as in \S\ref{sec:local-quadratic-spaces}.  In that case, we abbreviate $\O_1(F) := \O(V) \cong \{\pm 1\}$.  For $\phi \in \mathcal{S}(V_\mathbb{A}) = \mathcal{S}(\mathbb{A})$, we denote also by
\begin{equation*}
  \theta_{\psi}(\phi) : [\Mp_2] \rightarrow \mathbb{C}
\end{equation*}
the elementary theta function obtained by restricting to the first factor of the theta kernel defined in \S\ref{sec:theta-kernels}, thus
\begin{equation*}
  \theta_{\psi}(\phi)(s) := \theta_{\psi}(\phi)(s,1) = \sum_{x \in F} \rho_{\Weil}^{\psi,F}(s) \phi(x).
\end{equation*}
By \eqref{eq:equivariance-theta-kernel},
\begin{equation}\label{eq:equivariance-etf}
  \rho_{\reg}(s) \theta_{\psi}(\phi)
  = \theta_{\psi}(\rho_{\Weil}^{\psi,F}(s) \phi)
  \text{ for }
  s \in \Mp_2(\mathbb{A}).
\end{equation}
The $\O_1(F)$-invariance of the theta kernel says that for $\phi \in \mathcal{S}(\mathbb{A})$,
\begin{equation}\label{eqn:O1-invariance-elementary-theta-fn}
  \text{$\theta_{\psi}(\phi) = \theta_{\psi}(\phi_-)$
    with $\phi_-(x) := \phi(-x)$.}
\end{equation}

\subsubsection{Ternary theta lifts}
\label{sec-3-2-7}
Suppose $V = B^0$.  Given $\Psi : [\PB^\times] \rightarrow \mathbb{C}$ and $\phi \in \mathcal{S}(B^0_\mathbb{A})$ and $s \in \Mp_2(\mathbb{A})$, set
\begin{equation*}
  \theta_{\psi}(\phi,\Psi;s) := \int_{g \in [\PB^\times]} \Psi(g) \theta_{\psi}(\phi;s,\Ad(g)) \,d g,
\end{equation*}
where $\Ad : \PB^\times_\mathbb{A} \xrightarrow{\cong} \SO(B^0_\mathbb{A})$ is the isomorphism induced by the notation of \S\ref{sec:general-notation}.  If $\Psi$ is a cusp form, then the integral converges absolutely and defines a cusp form
\begin{equation*}
  \theta_{\psi}(\phi,\Psi) : [\Mp_2] \rightarrow \mathbb{C}.
\end{equation*}
By \eqref{eq:equivariance-theta-kernel},
\begin{equation}\label{eq:equivariance-ttf}
  \rho_{\reg}(s) \theta_{\psi}(\phi,\Psi)
  = \theta_{\psi}(\rho_{\Weil}^{\psi,B^0}(s) \phi, \Psi)
  \text{ for }
  s \in \Mp_2(\mathbb{A}),
\end{equation}
\begin{equation}\label{eq:equivariance-ttf-2}
  \theta_{\psi}(\Ad(g)\phi, \rho_{\reg}(g) \Psi)
  =
  \theta_{\psi}(\phi, \Psi)
  \text{ for }
  g \in \PB^\times_\mathbb{A}.
\end{equation}

\subsubsection{Factorization}
\label{sec-3-2-8}
If the quadratic space $V$ decomposes as the direct sum $V' \oplus V''$ of quadratic subspaces, then the factorization of the Weil representation (\S\ref{sec:factorization-weil-repn}) implies the factorization of theta functions: for $g = g' \times g'' \in \O(V'_\mathbb{A}) \times \O(V_\mathbb{A} '') \leq \O(V_\mathbb{A})$ and $\phi = \phi ' \otimes \phi '' \in \mathcal{S}(V_\mathbb{A})$ with $\phi ' \in \mathcal{S}(V'_\mathbb{A}), \phi '' \in \mathcal{S}(V_\mathbb{A} '')$ (see \S\ref{sec:factorization-weil-repn}),
\begin{equation}\label{eqn:Factorization-of-theta-fns}
  \theta_{\psi}(\phi;s,g)
  =
  \theta_{\psi}(\phi ';s,g')
  \theta_{\psi}(\phi '';s,g'').
\end{equation}
With $\phi '_-$ as in \eqref{eqn:O1-invariance-elementary-theta-fn} and notation as in \S\ref{sec:general-notation}, one has
\begin{equation}\label{eq:action-of-S-on-pure-tensors-for-B-0}
  \Ad(g) \phi = \phi ' \otimes \Ad(g) \phi ''
\end{equation}
\begin{equation}\label{eq:action-of-S-on-pure-tensors-for-B}
  \mathfrak{S} \phi =
  \frac{1}{2} 
  ( \phi ' + \phi '_-)
  \otimes \phi ''.
\end{equation}

\subsection{Equidistribution of products of pairs of elementary theta functions\label{sec:equid-prod-pairs-theta}}
\label{sec-3-5}
The purpose of this section is to recall and apply some results from \cite{nelson-theta-squared}.  Let $\tau_1,\tau_2 \in F^\times$.  Throughout this section, we regard $\psi,\tau_1,\tau_2,F,B$ as fixed: implied constants may depend upon them without explicit mention.  We assume also (for technical convenience) that $B$ is non-split.

\subsubsection{Some asymptotic notation}\label{sec:some-asympt-notat}
Given a topological vector space $\mathcal{S}$, we adopt the convention (similar to ``big O notation'') of denoting by $\mathcal{C}(\phi)$ any quantity depending continuously upon $\phi \in \mathcal{S}$; the continuity is assumed uniform in all auxiliary parameters except those explicitly labelled ``fixed.''  The space $\mathcal{S}$ itself is always regarded as fixed, of course.  This convention applies in particular to Schwartz--Bruhat spaces of finite-dimensional vector spaces over local fields, over finite products of local fields, or over adele rings.

Similarly to the ``$\eps$-convention'' of analytic number theory, we allow the precise meaning of $\mathcal{C}(\phi)$ to change from one occurrence to the next.  When we specifically wish to distinguish between several such quantities, we use the notation $\mathcal{C} '(\phi), \mathcal{C} ''(\phi)$, and so on.

For example, let $V$ be a vector space over $F$ (always assumed finite-dimensional).  Let $\hat{F}$ denote the ring of finite adeles, so that $\mathbb{A} = F_\infty \times \hat{F}$ with $F_\infty := \prod_{\mathfrak{p} | \infty} F_\mathfrak{p}$.  Similarly, write $V_\mathbb{A} = V_\infty \times \hat{V}$.  The Schwartz--Bruhat space $\mathcal{S}(V_\mathbb{A})$ factors as the algebraic tensor product $\mathcal{S}(V_\infty) \otimes \mathcal{S}(\hat{V})$.  Suppose given some quantities $a(\phi;t_1,t_2)$ and $b(\phi;t_1,t_2)$ depending upon $\phi \in \mathcal{S}(V_\mathbb{A})$ and some auxiliary parameters $t_1,t_2$.  The notation
\begin{equation}\label{eq:example-a-phi-bounded-by-C-phi}
  a(\phi;t_1,t_2)
  \ll b(\phi;t_1,t_2) \mathcal{C}(\phi) \text{ for fixed $t_2$ }
\end{equation}
means that for each $t_2$ and each $\phi_f \in \mathcal{S}(\hat{V})$, there is a finite collection $\mathcal{P}$ of polynomials on $V_\infty$ and a finite collection $\mathcal{D}$ of translation-invariant differential operators on $V_\infty$ (thus $\mathcal{D}$ consists of linear combinations of monomials $\frac{\partial }{\partial x_{i_1}} \dotsb \frac{\partial }{\partial x_{i_n}}$ with respect to some coordinates $x_j : V_\infty \rightarrow \mathbb{R}$) so that for all $\phi_\infty \in \mathcal{S}(V_\infty)$ and all $t_1$,
\[
  | a(\phi_\infty \otimes \phi_f;t_1,t_2) | \leq |b(\phi_\infty \otimes \phi_f;t_1,t_2)| \sum_{P \in \mathcal{P}} \sum_{D \in \mathcal{D}} \|P D \phi_\infty \|_{L^\infty(V^\infty)}.
\]
One could just as well write ``the functional $\mathcal{S}(V_\mathbb{A}) \ni \phi \mapsto a(\phi;t_1,t_2)/b(\phi;t_1,t_2)$ is defined and continuous, uniformly in $t_1$.''

\subsubsection{Simple estimates for lattice sums}\label{sec:simple-estimates-for-lattice-sums}
\begin{lemma}\label{lem:cheap-lattice-sum-bound-Rn-Zn}
  Let $n \in \mathbb{Z}_{\geq 0}$.  Let $A \geq 0$ and $t_0 > 0$ be fixed.  For each $\phi \in \mathcal{S}(\mathbb{R}^n)$ and $t > t_0$, one has $\sum_{v \in \mathbb{Z}^n - \{0\}} |\phi(t v)| \ll |t|^{-A} \mathcal{C}(\phi)$.
\end{lemma}
\begin{proof}
  The left hand side of the required estimate is bounded by
  \[C |t|^{-A} \sup_{x \in \mathbb{R}^n} |x|^{n+1+A} |\phi(x)|\] with $C := |t_0|^{-(n+1)} \sum_{v \in \mathbb{Z}^n - \{0\}} |v|^{-(n+1+A)} < \infty$.
\end{proof}
\begin{lemma}\label{lem:cheap-lattice-sum-bound-adelic}
  Let $V$ be a vector space over $F$.  Let $A \geq 0$ and $t_0 > 0$ be fixed.  For $\phi \in \mathcal{S}(V_\mathbb{A})$ and $y \in \mathbb{A}^\times$ with $|y| > t_0$, one has $\sum_{v \in V - \{0\}} |\phi(y v)| \ll |y|^{-A} \mathcal{C}(\phi)$.
\end{lemma}
\begin{proof}
  Let $\mathbb{A}^{(1)} := \{y \in \mathbb{A}^\times : |y| = 1\}$ denote the subgroup of norm one ideles.  By the compactness of $\mathbb{A}^{(1)} / F^\times$ and the continuity of the dilation action of $\mathbb{A}^\times$ on $\mathcal{S}(V_\mathbb{A})$, it suffices to consider the case that $y_\mathfrak{p} = 1$ for all finite $\mathfrak{p}$ and $y_\mathfrak{p} = t$ for all infinite $\mathfrak{p}$, where $t \in \mathbb{R}_+^\times$ satisfies $t > t_1 := t_0^{1/[F:\mathbb{Q}]}$.  Each $\phi_f \in \mathcal{S}(\hat{V})$ is bounded and satisfies $\supp(\phi_f) \cap V \subseteq L$ for some lattice $L \leq V$, so it suffices to show
  that for fixed $A \geq 0$, fixed $L \leq V$, all $\phi \in \mathcal{S}(V_\infty)$ and all $t > t_1$, one has $\sum_{v \in L - \{0\}} |\phi(t v)| = \O(|t|^{-A} \mathcal{C}(\phi))$.  By choosing a $\mathbb{Z}$-basis of $L$, we reduce to Lemma \ref{lem:cheap-lattice-sum-bound-Rn-Zn}.
\end{proof}

\subsubsection{Simple estimates for theta functions}
\label{sec-3-5-4}
Recall that the \emph{Iwasawa decomposition} asserts that each $s \in \SL_2(\mathbb{A})$ may be written in the form $s = n(x) t(y) k$, where $x \in \mathbb{A}, y \in \mathbb{A}^\times$ and $k$ belongs to the standard maximal compact subgroup of $\SL_2(\mathbb{A})$.  The decomposition is not unique, but the quantities $x$ and $|y|$ depend only upon $s$.

We define $\htt: [\SL_2] \rightarrow \mathbb{R}_{>0}$ by $\htt(g) := \max_{\gamma \in \SL_2(F)} \htt_\mathbb{A}(\gamma g)$, where $\htt_{\mathbb{A}} : \SL_2(\mathbb{A}) \rightarrow \mathbb{R}_{>0}$ is defined with respect to the Iwasawa decomposition $s = n(x) t(y) k$ by $\htt_{\mathbb{A}}(s) := |y|^{1/2}$.  One has $\int_{[\SL_2]} \htt^{1-\eps} < \infty$ for $\eps > 0$.  \emph{Reduction theory} says that the image of $\htt$ is bounded from below by some $c > 0$ depending only upon $F$.  We extend $\htt$ via pullback to a map $[\Mp_2] \rightarrow \mathbb{R}_{>0}$.

Recall that the nontrivial unitary character $\psi$ of $\mathbb{A}/F$ is regarded as fixed.
\begin{lemma}\label{lem:crude-bound-for-cuspidal-theta-functions}
  Let $A \geq 0$ be fixed.  Let $\Psi \in L^1( [\PB^\times])$ with $\langle \Psi, 1 \rangle = 0$.  Let $\phi \in \mathcal{S}(B_\mathbb{A}^0)$.  For $s \in \Mp_2(\mathbb{A})$, one has $\theta_{\psi}(\phi,\Psi;s) \ll \htt(s)^{-A} \mathcal{C}(\phi) \|\Psi \|_{L^1}$.
\end{lemma}
\begin{proof}
  Since $\Psi$ has mean zero,
  \begin{equation}
    \theta_{\psi}(\phi,\Psi;s) = \int_{g \in [\PB^\times]} \Psi(g) \sum_{v \in V - \{0\}} \rho_{\Weil}^{\psi,V}(s,\Ad(g)) \phi(v) \,d g.
  \end{equation}
  By the Iwasawa decomposition and reduction theory, we may assume that $s = n(x) t(y) k$ with $|y| \gg 1$.  Since $B$ is non-split, we may fix a compact subset $U$ of $\PB^\times_\mathbb{A}$ containing a fundamental domain for $[\PB^\times]$.  Then
  \[
    |\theta_{\psi}(\phi,\Psi;s)| \leq \|\Psi \|_{L^1} |y|^{3/2} \sup_{g \in U} \sum_{v \in B^0 - \{0\}} |\rho_{\Weil}^{\psi,B^0}(k,\Ad(g)) \phi(y v)|.
  \]
  Since the Weil representation is continuous \cite[\S39]{MR0165033}, we may reduce to the case $k = 1$ and $g = 1$, in which the required estimate follows from Lemma \ref{lem:cheap-lattice-sum-bound-adelic} of \S\ref{sec:simple-estimates-for-lattice-sums}.
\end{proof}

\begin{lemma}\label{lem:crude-bound-for-elementary-theta-functions}
  For $\phi \in \mathcal{S}(\mathbb{A})$ and $s \in \Mp_2(\mathbb{A})$, one has $\theta_{\psi}(\phi;s) \ll \htt(s)^{1/4} \mathcal{C}(\phi)$.
\end{lemma}
\begin{proof}
  We argue as in the proof of Lemma \ref{lem:crude-bound-for-cuspidal-theta-functions}, but take into account the contribution from $0 \in F$ to the definition of $\theta_{\psi}(\phi)$.
\end{proof}

\subsubsection{Main estimate: the case of pure tensors}
\label{sec-3-5-3}
\begin{lemma*}\label{lem:main-result-for-theta-squared}
  Let $\phi_1', \phi_2' \in \mathcal{S}(\mathbb{A})$ and $\phi_1'', \phi_2'' \in \mathcal{S}(B_\mathbb{A}^0)$.  Let $\Psi_1, \Psi_2 : [\PB^\times] \rightarrow \mathbb{C}$ be integrable functions of mean zero.  Let $\tau_1, \tau_2 \in F^\times$ be fixed.  Abbreviate $\theta_i := \theta_{\psi^{\tau_i}}(\phi_i')$ and $h_i := \theta_{\psi^{\tau_i}}(\phi_i'',\Psi_i)$.  Then for all $s \in \Mp_2(\mathbb{A})$,
  \[
    \langle \theta_1 \cdot \rho_{\reg}(s) h_1, \theta_2 \cdot \rho_{\reg}(s) h_2 \rangle = \langle \theta_1, \theta_2 \rangle \langle h_1, h_2 \rangle + \O\left(\Xi(s) \prod_{i=1,2} \mathcal{C}(\phi_i') \mathcal{C}(\phi_i'') \|\Psi_i\|_{L^1} \right).
  \]
\end{lemma*}
\begin{proof}
  The main result of \cite{nelson-theta-squared} gives an estimate essentially of the required shape, but instead with the error term $\Xi(s) \mathcal{S}(\phi_1') \mathcal{S}(\phi_2') \mathcal{S}^{\mathbf{X}}(h_1 \overline{h_2})$, where $\mathcal{S}, \mathcal{S}^{\mathbf{X}}$ are the adelic Sobolev norms defined in \emph{loc.\ cit.}\ and \cite[\S2]{michel-2009}.  By the cuspidality of $h_1, h_2$ and axioms (S3b) and (S4e) of \cite{michel-2009}, we may replace the expression $\mathcal{S}^{\mathbf{X}}(h_1 \overline{h_2})$ first with $\mathcal{S}^{\mathbf{X}}(h_1) \mathcal{S}^{\mathbf{X}}(h_2)$ and then with $\mathcal{S}(h_1) \mathcal{S}(h_2)$.  Our task thereby reduces to showing for $i=1,2$ that $\mathcal{S}(\phi_i') \ll \mathcal{C}(\phi_i')$ and $\mathcal{S}(h_i) \ll \mathcal{C}(\phi_i'') \|\Psi_i\|_{L^1}$.  The first of these estimates, i.e., that the norms $\mathcal{S}$ are continuous, follows readily from the definitions of those norms.  The second estimate follows similarly, using Lemma \ref{lem:crude-bound-for-cuspidal-theta-functions} of \S\ref{sec-3-5-4}.
\end{proof}

\subsubsection{Factorization}\label{sec:quantitative-factorization}
Let $V', V''$ be vector spaces over $F$ and $V := V' \oplus V''$.

\begin{lemma*}\label{lem:quantitative-factorization}
  Let $\ell : \mathcal{S}(V_\mathbb{A} ') \otimes \mathcal{S}(V_\mathbb{A} '') \rightarrow \mathbb{C}$ be an algebraic linear functional on the algebraic tensor product of Schwartz--Bruhat spaces satisfying an estimate of the form
  \begin{equation*}
    \ell(\phi' \otimes \phi '') \ll \mathcal{C}'(\phi ') \mathcal{C}''(\phi '')
  \end{equation*}
  on pure tensors.
  Then $\ell$ extends to a continuous functional $\ell : \mathcal{S}(V_\mathbb{A}) \rightarrow \mathbb{C}$ satisfying
  \[
    \ell(\phi) \ll \mathcal{C}(\phi)
  \]
  for all $\phi \in \mathcal{S}(V_\mathbb{A})$, where $\mathcal{C}$ depends only upon $\mathcal{C} '$ and $\mathcal{C} ''$.
\end{lemma*}
\begin{proof}
  This is essentially the Schwartz kernel theorem, as extended by Bruhat \cite[\S5]{MR0140941}, and follows from a standard ``square-root of a partition of unity'' argument.
\end{proof}

\subsubsection{Main estimate: the general case\label{sec:main-estimate-for-equidistribution-pairs-theta}}
\label{sec-3-5-5}
Temporarily denote by $\mathcal{A}_0$ the space of integrable functions $\Psi : [\PB^\times] \rightarrow \mathbb{C}$ of mean zero.  Let $\mathcal{E}_{\tau_1,\tau_2} : \mathcal{S}(B_\mathbb{A}) \otimes \mathcal{S}(B_\mathbb{A}) \otimes \mathcal{A}_0 \otimes \mathcal{A}_0 \rightarrow \mathbb{C}$ denote the sesquilinear form given for $\phi_i = \phi_i' \otimes \phi_i'' \in \mathcal{S}(B_\mathbb{A})$ with $\phi_1',\phi_2' \in \mathcal{S}(\mathbb{A}), \phi_1'',\phi_2'' \in \mathcal{S}(B_\mathbb{A}^0)$ by
\begin{align*}
  \mathcal{E}_{\tau_1,\tau_2}(\phi_1,
  \phi_2,
  \Psi_1, \Psi_2)
  :=
  \langle
  \theta_1 h_1, \theta_2 h_2 \rangle
  - \langle \theta_1, \theta_2 \rangle
  \langle h_1, h_2 \rangle,
\end{align*}
where we abbreviate $\theta_i := \theta_{\tau_i}(\phi_i')$ and $h_i := \theta_{\tau_i}(\phi_i'',\Psi_i)$.  (The definition makes sense: \emph{a priori} estimates as in \S\ref{sec-3-5-4} and the density of $\mathcal{S}(\mathbb{A}) \otimes \mathcal{S}(B_\mathbb{A}^0)$ in $\mathcal{S}(B_\mathbb{A})$ allow us to extend $\mathcal{E}_{\tau_1,\tau_2}$ continuously from its initial domain.)

\begin{proposition}\label{prop:main-error-estimate-global-adelic-general}
  For $\phi_1, \phi_2 \in \mathcal{S}(B_\mathbb{A}), \Psi_1,\Psi_2 \in \mathcal{A}_0$ and $s \in \Mp_2(\mathbb{A})$, one has with $\rho_0^{\tau}(s) := \rho_{\Weil}^{\psi^{\tau}, B^0}(s)$ the estimate
  \begin{equation}
    \mathcal{E}_{\tau_1,\tau_2}((1 \otimes \rho_0^{\tau_1}(s))\phi_1,
    (1 \otimes \rho_0^{\tau_2}(s))\phi_2,
    \Psi_1,\Psi_2)
    \ll \Xi(s)
    \prod_{j=1,2}
    \Sob(\phi_j)
    \|\Psi_j\|_{L^1}.
  \end{equation}
  The implied constant and the uniformity of the continuity of $\Sob(\phi_j)$
  depend at most upon $\psi,\tau_1,\tau_2,F,B$.  The operator $1 \otimes \rho_0^{\tau_1}(s)$ is defined as in \S\ref{sec:factorization-weil-repn}.
\end{proposition}
\begin{proof}
  The lemma of \S\ref{sec:quantitative-factorization} reduces the general case of Proposition \ref{prop:main-error-estimate-global-adelic-general} to the special case in which $\phi_i = \phi_i' \otimes \phi_i''$ for $i=1,2$, which follows from the lemma of \S\ref{sec-3-5-3} upon recalling from \eqref{eq:equivariance-ttf} that $\theta_{\psi_\tau}$ intertwines $\rho_{0}^{\tau}$ with $\rho_{\reg}$.
\end{proof}

\subsubsection{Invariance properties}\label{sec:equivariance-summary-for-E-tau-tau}
We record these for later use.
\begin{lemma*}\label{lem:equivariance-summary-for-E-tau-tau}
  For $g_1, g_2 \in \PB^\times_\mathbb{A}$ and $s \in \Mp_2(\mathbb{A})$, one has
  \begin{align*}
    \mathcal{E}_{\tau_1,\tau_2}(\phi_1, \phi_2, \Psi_1, \Psi_2)
    &=
      \mathcal{E}_{\tau_1,\tau_2}(\mathfrak{S} \phi_1, \phi_2,
      \Psi_1, \Psi_2)
    \\
    &=
      \mathcal{E}_{\tau_1,\tau_2}(\phi_1, \mathfrak{S} \phi_2,
      \Psi_1, \Psi_2)
    \\
    &=
      \mathcal{E}_{\tau_1,\tau_2}(\Ad(g_1) \phi_1, \Ad(g_2) \phi_2,
      \rho_{\reg}(g_1) \Psi_1,       \rho_{\reg}(g_2)\Psi_2)
    \\
    &=
      \mathcal{E}_{\tau_1,\tau_2}(\rho_{\Weil}^{\psi^{\tau_1},B}(s) \phi_1, \rho_{\Weil}^{\psi^{\tau_2},B}(s) \phi_2,
      \Psi_1, \Psi_2).
  \end{align*}
\end{lemma*}
\begin{proof}
  The first two identities follow from \eqref{eq:action-of-S-on-pure-tensors-for-B} and \eqref{eqn:O1-invariance-elementary-theta-fn}, the remaining from \eqref{eq:equivariance-etf}, \eqref{eq:equivariance-ttf}, \eqref{eq:equivariance-ttf-2}, \eqref{eq:action-of-S-on-pure-tensors-for-B-0}, \eqref{eq:action-of-S-on-pure-tensors-for-B} and the translation invariance of the Petersson inner product.
\end{proof}

\subsection{Similitude theta functions\label{sec:non-traditional-theta-lifts}}
\label{sec-3-2-9}

\subsubsection{Weil representation\label{sec:similitudes-global}}
\label{sec-3-2-4}
For each place $\mathfrak{p}$ of $F$, let $\Omega_{\mathfrak{p}}$ denote the representation of $\PGL_2(F_{\mathfrak{p}}) \times \PGO(B_\mathfrak{p})$ attached as in \S\ref{sec:defn-local-omega} to the tuple $(F_\mathfrak{p},B_{\mathfrak{p}},\psi_{\mathfrak{p}})$.  Let $\Omega$ denote the restricted tensor product of the spaces $\Omega_{\mathfrak{p}}$ with respect to the distinguished elements, which we denote now by $\phi_{\mathfrak{p}}^0 \in \Omega_{\mathfrak{p}}$.  We may and shall identify $\Omega$ with the space of functions $\phi : \mathbb{A}^\times \times B_\mathbb{A} \rightarrow \mathbb{C}$ such that
\begin{itemize}
\item For each $t \in \mathbb{A}^\times$, the function $\phi[t] : B_\mathbb{A} \rightarrow \mathbb{C}$ given by $\phi[t](x) := \phi(t,x)$ belongs to the Schwartz--Bruhat space $\mathcal{S}(B_\mathbb{A})$;
\item $\phi(z^2 t, x) = \phi(t, z x)$ for all $z,t \in \mathbb{A}^\times, x \in B_\mathbb{A}$.
\item There is a compact subset $C$ of $\mathbb{A}^\times / \mathbb{A}^{\times 2}$ such that $\phi[t] = 0$ for all $t \notin C$ (i.e., for all $t \in \mathbb{A}^\times$ whose image in $\mathbb{A}^\times / \mathbb{A}^{\times 2}$ lies outside $C$);
\item There is an open subgroup $U$ of $\mathbb{A}^\times / \mathbb{A}^{\times 2}$ such that $\phi[t u] = \phi[t]$ for all $t \in \mathbb{A}^\times, u \in U$.
\end{itemize}
We equip $\Omega$ with the invariant hermitian norm $\|.\|$ obtained by tensoring those on the factors $\Omega_\mathfrak{p}$, thus
\begin{equation}\label{eqn:inner-product-on-Omega-1-adelic}
  \|\phi\|^2_{\Omega}
  :=
  \int_{t \in \mathbb{A}^\times / \mathbb{A}^{\times 2}}
  |t|^2
  \int_{x \in B_\mathbb{A}}
  |\phi|^2(t,x) \, d x \, d_2^\times t.
\end{equation}
The group $\PGL_2(\mathbb{A}) \times \PGO(B_\mathbb{A})$ acts on $\Omega$ by the representation $\rho_{\Weil}$ obtained as the restricted tensor product of those defined in \S\ref{sec:defn-local-omega}.  We define $\mathfrak{S} : \Omega \rightarrow \Omega$ and (for $g \in \PB^\times_\mathbb{A}$) $\Ad(g) : \Omega \rightarrow \Omega$ as in \S\ref{sec:defn-local-omega}.  Note that $\mathfrak{S}$ does \emph{not} preserve pure tensors: for $\phi = \otimes \phi_\mathfrak{p} \in \Omega$,
\begin{align*}
  \mathfrak{S} \phi (t,x)
  &= (\phi (t,x) + \phi (t,x -
    \tr(x)))/2,
  \\
  \otimes \mathfrak{S} \phi_\mathfrak{p}(x,t)
  &= \prod  (\phi (t_\mathfrak{p},x_\mathfrak{p}) + \phi (t_\mathfrak{p},x_\mathfrak{p} -
    \tr(x_\mathfrak{p} )))/2,
\end{align*}
and in general, these are not the same.  They \emph{do} coincide if $\# \{ \mathfrak{p} : \mathfrak{S} \phi_\mathfrak{p} \neq \phi_\mathfrak{p} \} \leq 1$.

\subsubsection{Theta functions}\label{sec:35ac3e570a}
For $\phi \in \Omega$, $s \in \PGL_2(\mathbb{A}), g \in \PGO(B_\mathbb{A})$, set
\begin{equation}
  \label{defn:theta-knerle-for-Omega}
  \Theta(\phi;s,g)
  :=
  \frac{1}{2} \sum_{\tau \in F^\times / F^{\times 2}}
  \sum_{x \in B}
  \rho_{\Weil}(s,g)
  \phi(\tau,x).
\end{equation}
The sum is well-defined, converges absolutely and defines a smooth function $\Theta(\phi)$ on $[\PGL_2] \times [\PGO(B)]$.  For a cusp form $\Psi : [\PB^\times] \rightarrow \mathbb{C}$ and $s \in \PGL_2(\mathbb{A})$, set
\begin{equation}
  \Theta(\phi,\Psi;s) := \int_{g \in [\PB^\times]} \Psi(g) \Theta(\phi;s,\Ad(g)) \, d g.
\end{equation}
The integral (together with similar integrals below) converges absolutely and defines a cusp form $\Theta(\phi,\Psi) : [\PGL_2] \rightarrow \mathbb{C}$.

\begin{remark*}
  $\Theta(\phi,\Psi)$ is not a theta lift in the traditional sense: the integral in its definition is with respect to the orthogonal group of $B^0$ rather than that of $B$.
\end{remark*}

\subsubsection{Fourier expansion}\label{sec:four-expans}
Let $\phi \in \Omega$, and let $\Psi : [\PB^\times] \rightarrow \mathbb{C}$ be a cusp form.
\begin{lemma*}
  For $x \in \mathbb{A}$, $y \in \mathbb{A}^\times$, one has
  \[
    \Theta(\phi,\Psi;n(x) a(y)) = \sum_{\tau \in F^\times} \psi(\tau x) W(\Theta(\phi,\Psi), \tau y)
  \]
  where $W(\Theta(\phi,\Psi),y) := \int_{g \in [\PB^\times]} \Psi(g) \sum_{\gamma \in \PB^\times} |y| \phi(y \nr(\gamma)^{-1}, g^{-1} \gamma g) \,d g$.
\end{lemma*}
\begin{proof}
  By direct unfolding as in \cite{MR0333081}, one has for $g \in \PB^\times_{\mathbb{A}}$ that
  \begin{equation*}
    \begin{split}
      \Theta(\phi, n(x) a(y);\Ad(g)) &= \frac{1}{2} \sum_{\tau \in F^\times / F^{\times 2}}
      |y| \phi(\tau y,0) \\
      &\quad + \sum_{\tau \in F^\times} \psi(\tau x) W(\Theta(\phi),\Ad(g),\tau y),
    \end{split}
  \end{equation*}
  where $W(\Theta(\phi),\Ad(g),y) := \sum_{\gamma \in \PB^\times} |y| \phi(y \nr(\gamma)^{-1}, g^{-1} \gamma g)$.  We conclude by integrating against $\Psi$.
\end{proof}

\subsubsection{Restriction to $\SL_2$}\label{sec:35ac3e570e}
Let $\phi, \Psi$ be as above.

\begin{lemma*}
  Let $y \in \mathbb{A}^\times$.  Suppose that $\phi [y] = \phi '[y] \otimes \phi ''[y]$
  for some $\phi '[y] \in \mathcal{S}(\mathbb{A})$, $\phi ''[y] \in \mathcal{S}(B_\mathbb{A}^0)$.  Then for $s \in \SL_2(\mathbb{A})$,
  \begin{equation}\label{eqn:simliitude-theta-expand-as-sum-of-pure-tensors-and-over-a}
    \Theta(\phi,\Psi;s a(y))
    =
    \frac{1}{2}
    \sum_{\tau \in F^\times \backslash F^{\times 2}}
    |y|
    \theta_{\psi^\tau}(\phi'[\tau y];s)
    \theta_{\psi^\tau}(\phi''[\tau y],\Psi;s).
  \end{equation}
\end{lemma*}
\begin{proof}
  We derive first using \eqref{defn:theta-knerle-for-Omega} that for $g \in \O(B_{\mathbb{A}})$,
  \[
    \Theta(\phi;s a(y),g) = \frac{1}{2} \sum_{\tau \in F^\times \backslash F^{\times 2}} |y| \theta_{\psi^\tau}(\phi[\tau y];s,g),
  \]
  hence by \eqref{eqn:Factorization-of-theta-fns} that for $g \in \O(B_\mathbb{A}^0)$,
  \[
    \Theta(\phi;s a(y),g) = \frac{1}{2} \sum_{\tau \in F^\times \backslash F^{\times 2}} |y| \theta_{\psi^\tau}(\phi'[\tau y];s) \theta_{\psi^\tau}(\phi''[\tau y];s,g).
  \]
  We integrate against $\Psi$ to conclude.
\end{proof}

\subsubsection{Unfolding the inner product}\label{sec:unfold-inner-product-nontraditional-theta-over-sl2}
Let $\phi_1, \phi_2 \in \Omega$.  Let $\Psi_1, \Psi_2 : [\PB^\times] \rightarrow \mathbb{C}$ be cusp forms.

\begin{lemma*}
  Suppose that for each $y \in \mathbb{A}^\times$, one has $\phi_i [y] = \phi_i '[y] \otimes \phi_i ''[y]$
  for some $\phi_i '[y] \in \mathcal{S}(\mathbb{A})$, $\phi_i ''[y] \in \mathcal{S}(B_\mathbb{A}^0)$.
  Then the following identity holds, with both sides converging absolutely:
  \begin{equation}\label{eq:inner-product-pgl2-made-into-sl2}
    \begin{split}
      &\langle \Theta(\phi_1,\Psi_1),
    \Theta(\phi_2,\Psi_2) \rangle_{\PGL_2} \\
    &\quad =
    \int_{y \in \mathbb{A}^\times / F^\times \mathbb{A}^{\times 2}}
    |y|^2
    \frac{1}{2^2}
    \sum_{\tau_1,\tau_2 \in F^\times / F^{\times 2}}
    \langle \theta_1 h_1, \theta_2 h_2 \rangle_{\SL_2}
    \, d_2^\times y,
  \end{split}
\end{equation}
\begin{equation}\label{eq:thet-:=-thet}
    \theta_i := \theta_{\psi^{\tau_i}}(\phi_i'[\tau_i y]), \quad h_i := \theta_{\psi^{\tau_i}}(\phi_i''[\tau_i y],\Psi_i).
  \end{equation}
\end{lemma*}
\begin{proof}
  The left hand side is an inner product of cusp forms, hence convergent.  On the right hand side, we may replace the $y$-integral by a finite sum, since the domain $\mathbb{A}^{\times} / F^{\times } \mathbb{A}^{\times 2}$ is compact and the integrand is invariant under an open subgroup.  For individual $y$, the sum over $\tau_1, \tau_2$ has only finitely many nonzero summands, each of which consists of an inner product whose convergence is clear (see \S\ref{sec-3-5-4}).  The expansion \eqref{eqn:simliitude-theta-expand-as-sum-of-pure-tensors-and-over-a} implies for $y \in \mathbb{A}^\times, s \in \SL_2(\mathbb{A})$ that $\Theta(\phi_i,\Psi_i,a(y) s) = \frac{1}{2} \sum _{\tau_i \in F^\times / F^{\times 2} } |y| \theta_i(s) h_i(s)$, so the required identity follows from the formula \eqref{eq:sl2-vs-pgl2-integrals} relating integrals over $[\PGL_2]$ and $[\SL_2]$.
\end{proof}

\begin{remark*}
  In this paper, we consider several expressions shaped like the right hand side of \eqref{eq:inner-product-pgl2-made-into-sl2}.  On a first (or perhaps on any) reading, one should focus on the contributions from $y = \tau_1 = \tau_2 = 1$; under some class number and unit group restrictions, these turns out to be the relevant ones for our applications.  (We considered imposing such restrictions for the sake of presentation, but found that doing so obfuscated rather than simplified.)
\end{remark*}

\subsection{Inner product formulas}
\label{sec-3-7}

\subsubsection{Elementary theta functions}\label{sec:elem-theta-ipf}
We recall part of \cite[Thm 2]{nelson-theta-squared}.
\begin{lemma*}
  Suppose $\phi_1, \phi_2 \in \mathcal{S}(\mathbb{A})$ satisfy $\phi_1(x) = \phi_1(-x), \phi_2(x) = \phi_2(-x)$.  Let $\tau_1, \tau_2 \in F^\times$.  Set $\theta_1 := \theta_{\psi^{\tau_1}}(\phi_1), \theta_2 := \theta_{\psi^{\tau_1}}(\phi_2)$.  Then $\langle \theta_1, \theta_2 \rangle_{\SL_2}= 0$ unless $\tau_1 = \tau_2$, in which case $\langle \theta_1, \theta_2 \rangle_{\SL_2} = 2 \langle \phi_1, \phi_2 \rangle_{L^2(\mathbb{A})}$.
\end{lemma*}
\subsubsection{Ternary theta lifts}\label{sec:ternary-theta-ipf}
As in \cite[\S12.3]{nelson-variance-73-2}, we explicate Gan--Takeda \cite[Thm 6.6]{MR2837015} (compare with \cite[Prop 2.8 (i)]{MR3291638}).
\begin{lemma*}
  Let $\pi_1, \pi_2 \subseteq L^2([\PB^\times])$ be cuspidal automorphic representations that are not one-dimensional.  Let $\Psi_1 \in \pi_1, \Psi_2 \in \pi_2$ and $\phi_1, \phi_2 \in \mathcal{S}(B_\mathbb{A}^0)$.  Let $\tau \in F^\times$.  Set $h_i := \theta_{\psi^{\tau}}(\phi_i,\Psi_i)$ for $i=1,2$.
  \begin{enumerate}
  \item If $\pi_1 \neq \pi_2$, then $\langle h_1, h_2 \rangle_{\SL_2} = 0$.
  \item Suppose $\pi_1 = \pi_2 =: \pi$.  Let $S$ be a finite set of places of $F$ that contains all archimedean places, as well as any finite places at which $B$ ramifies, and that is sufficiently large in terms of $\Psi_i, \phi_i$.
    Then $\langle h_1, h_2 \rangle_{\SL_2}$ equals
    \begin{equation}
      \frac{L^{(S)}(\pi,\tfrac{1}{2})}{\zeta_F^{(S)}(2)} \left(\prod_{\mathfrak{p} \notin S} \vol(K_\mathfrak{p}) \right) \int_{g \in \PB^\times_S} \langle \Ad(g) \phi_1, \phi_2 \rangle_{L^2(B_\mathbb{A}^0)} \langle \pi(g) \Psi_1, \Psi_2 \rangle_{\PB^\times} \, d g
    \end{equation}
    with $L^{(S)}(\pi,\tfrac{1}{2})$ as in \S\ref{sec:standard-l-function}.
  \end{enumerate}
\end{lemma*}

\subsubsection{Induction to $\Omega$}\label{sec:induction-omega}
We now combine the previous two lemmas and sum them up.
Temporarily denote by $\mathcal{A}_0$ the space of cusp forms $\Psi : [\PB^\times] \rightarrow \mathbb{C}$ that are orthogonal to all one-dimensional representations.  For $\tau_1, \tau_2 \in F^\times$, let $\mathfrak{m} : \Omega \otimes \Omega \otimes \mathcal{A}_0 \otimes \mathcal{A}_0 \rightarrow \mathbb{C}$ denote the sesquilinear form given for $\phi_1, \phi_2 \in \Omega$ admitting factorizations $\phi_i[y] = \phi_i'[y] \otimes \phi_i''[y]$ by
\begin{equation}
  \mathfrak{m}(\phi_1,\phi_2,\Psi_1,\Psi_2)
  :=
  \int_{y \in \mathbb{A}^\times / F^{\times }
    \mathbb{A}^{\times 2}}
  |y|^2
  \frac{1}{2^2}
  \sum_{\tau_1,\tau_2 \in F^\times / F^{\times 2}}
  \langle \theta_1, \theta_2 \rangle
  \langle h_1, h_2 \rangle \, d_2^\times y,
\end{equation}
with $\theta_i, h_j$ as in \eqref{eq:thet-:=-thet}.  The relevance of $\mathfrak{m}$ may be inferred from \S\ref{sec:unfold-inner-product-nontraditional-theta-over-sl2}.

The definition makes sense: as in the proof of \S\ref{sec:unfold-inner-product-nontraditional-theta-over-sl2}, the $y$-integral is really a finite sum, and the sum over $\tau_1, \tau_2$ has only finitely many nonzero summands.  Each summand defines a sesquilinear form on $\mathcal{S}(\mathbb{A}) \otimes \mathcal{S}(B_\mathbb{A}^0) \otimes \mathcal{A}_0$ that extends continuously to $\mathcal{S}(B_\mathbb{A}) \otimes \mathcal{A}_0$ by the \emph{a priori} estimates of \S\ref{sec-3-5-4}.

\begin{lemma*}
  Let $\pi_1, \pi_2 \subseteq L^2([\PB^\times])$ be cuspidal automorphic representations that are not one-dimensional.  Let $\Psi_1 \in \pi_1, \Psi_2 \in \pi_2$.  Let $\phi_1, \phi_2 \in \Omega$.
  \begin{enumerate}
  \item If $\pi_1 \neq \pi_2$ then $\mathfrak{m}(\phi_1,\phi_2,\Psi_1,\Psi_2) = 0$.
  \item Suppose $\pi_1 = \pi_2 =: \pi$.  Let $S$ be a large enough finite set of places.
    Then $\mathfrak{m}(\phi_1,\phi_2,\Psi_1,\Psi_2)$ equals
    \begin{equation*}
      \frac{L^{(S)}(\pi,\tfrac{1}{2})}{\zeta_F^{(S)}(2)} \left(\prod_{\mathfrak{p} \notin S} \vol(K_\mathfrak{p}) \right) \int_{g \in \PB^\times_S} \langle \Ad(g) \mathfrak{S} \phi_1, \mathfrak{S} \phi_2 \rangle_{\Omega} \langle \pi(g) \Psi_1, \Psi_2 \rangle_{\PB^\times} \,d g.
    \end{equation*}
  \end{enumerate}
\end{lemma*}
\begin{proof}
  It suffices to consider the case that $\phi_1, \phi_2$ admit factorizations as in the definition of $\mathfrak{m}$.  By \eqref{eqn:O1-invariance-elementary-theta-fn} and \eqref{eq:action-of-S-on-pure-tensors-for-B}, we may assume that $\mathfrak{S} \phi_i = \phi_i$, or equivalently, that $\phi_i'(t,x) = \phi_i'(t,-x)$.  By the lemmas of \S\ref{sec:elem-theta-ipf} and \S\ref{sec:ternary-theta-ipf}, we have $\langle \theta_1, \theta_2 \rangle = 0$ unless $\tau_1 = \tau_2$ and then $\langle h_1, h_2 \rangle = 0$ unless $\pi_1 = \pi_2$; in that case, the formulas from those lemmas and the identities
  \[
    \langle \phi_1'[y \tau], \phi_2'[y \tau] \rangle_{L^2(\mathbb{A})} \langle \Ad(g) \phi_1''[y \tau], \phi_2''[y \tau] \rangle_{L^2(B_\mathbb{A}^0)}
  \]
  \[
    = \langle \Ad(g) \phi_1[y \tau], \phi_2[y \tau] \rangle_{L^2(B_\mathbb{A})}
  \]
  and (see \eqref{eq:integral-formula-involving-squares}, \eqref{eqn:inner-product-on-Omega-1-adelic})
  \begin{equation}
    \int_{y \in \mathbb{A}^{\times} / F^\times \mathbb{A}^{\times 2} } |y|^2 \frac{1}{2} \sum_{\tau \in F^\times / F^{\times 2}} \langle \Ad(g) \phi_1[y \tau], \phi_2[y \tau] \rangle_{L^2(B_\mathbb{A})} \, d_2^\times y
  \end{equation}
  \[
    = \int_{y \in \mathbb{A}^{\times} / \mathbb{A}^{\times 2} } |y|^2 \langle \Ad(g) \phi_1[y], \phi_2[y] \rangle_{L^2(B_\mathbb{A})} \, d_2^\times y = \langle \Ad(g) \phi_1, \phi_2 \rangle_{\Omega}
  \]
  combine to give the required conclusion.
\end{proof}

\section{Estimates for general quantum variance sums\label{sec:estimates-general-var}}
\label{sec-4}
In this section, we introduce general families of quantum variance sums, propose a candidate for their leading asymptotics, and state a general ``estimate'' comparing the two.
\subsection{Notation}
\label{sec-4-1}
Let $F$ be a number field with adele ring $\mathbb{A}$.  Fix a nontrivial unitary character $\psi$ of $\mathbb{A}/F$.  Let $B$ be a non-split quaternion algebra over $F$.  Fix a maximal order $R \subseteq B$ and a finite set $S$ of places of $F$, containing all archimedean places as well as any finite places at which $B$ ramifies.  Retain the (unsurprising) notation of \S\ref{sec-3-1}.

Since $B$ is non-split, the quotient $[\PB^\times]= \PB^\times \backslash \PB^\times_\mathbb{A}$ is compact, and so $L^2([\PB^\times])$ is completely reducible.  Let $A^{\flat}$ denote the set of irreducible subrepresentations of the Hilbert space $L^2([\PB^\times])$.  For each $\pi^{\flat} \in A^{\flat}$, let $\pi \leq \pi^{\flat}$ denote the subspace of smooth vectors.  Set $A := \{\pi : \pi^\flat \in A^{\flat}\}$.  Let $\mathcal{A}$ denote the algebraic direct sum $\oplus_{\pi \in A} \pi$, regarded as a pre-unitary representation of the group $\PB^\times_\mathbb{A}$.

We introduce the following additional notation:
\begin{itemize}
\item $K = \prod K_\mathfrak{p}$: a maximal compact subgroup of $\PB^\times_\mathbb{A}$.  For $\mathfrak{p} \notin S$, we assume that $K_\mathfrak{p} \leq \PB^\times_{\mathfrak{p}}$ is the image of $R_\mathfrak{p}^\times$.
\item $\mathcal{A}_0 \leq \mathcal{A}$: the orthogonal complement of the one-dimensional subrepresentations.  (We had earlier, in \S\ref{sec-3-5} and \S\ref{sec:induction-omega}, used the same symbol to denote some \emph{larger} spaces than what we call here $\mathcal{A}_0$.  This abuse of notation should introduce no confusion.)
\item $A_0 := \{\pi \in A : \pi \subseteq \mathcal{A}_0\} = \{\pi \in A : \dim(\pi) > 1\}$, so that $\mathcal{A}_0 = \oplus_{\pi \in A_0} \pi$.
\item $\mathcal{A}^S := \{\varphi \in \mathcal{A} : \rho_{\reg}(k) \varphi = \varphi \text{ for all } k \in K_\mathfrak{p}, \mathfrak{p} \notin S \}, \mathcal{A}_0^S := \mathcal{A}_0 \cap \mathcal{A}^S$: the ``unramified outside $S$'' subspaces of $\mathcal{A}, \mathcal{A}_0$.
\item $A^S := \{\pi \in A : \pi \cap \mathcal{A}^S \neq \{0\}\}, {A}_0^S := {A}^S \cap {A}_0$: the subsets consisting of those $\pi$ that are unramified outside $S$.
\item $\mathcal{B}(V)$, for $V$ a $K$-invariant subspace of $\mathcal{A}$: an orthonormal basis for the closure of $V$ that consists of $K$-isotypic elements of $V$.
  
\end{itemize}
Fix Haar measures on $\PB^\times_S$ and $[\PB^\times]$; we do not require any compatibility between them.  Because $B$ is non-split, each $\pi \in A_0$ is cuspidal.  Let $L^{(S)}(\pi,s)$, $L^{(S)}(\ad \pi,s)$ be as in \S\ref{sec:standard-l-function}.
\subsection{Key definitions\label{sec:main-general-estimate-key-defns}}
\label{sec-4-2}
By \eqref{eq:rapid-convergence-of-pretrace-formula}, \eqref{eq:polynomial-growth-of-reps} and \eqref{eq:HL}, the sums considered in the definitions to follow converge absolutely.

\subsubsection{The basic distributions}
\label{sec:omega-pi}
For $\pi \in A^S$, define
\begin{equation*}
  \omega_\pi : C_c^\infty(\PB^\times_S) \otimes \mathcal{A}^S \rightarrow \mathbb{C}
\end{equation*}
by
\begin{equation*}
  \omega_\pi(f,\Psi) := \sum_{\varphi \in \mathcal{B}(\pi \cap \mathcal{A}^S)} \langle \varphi, \Psi \cdot \pi(f) \varphi \rangle.
\end{equation*}
The definition is independent of the choice of orthonormal basis.

\begin{example*}\label{example:harmonic-sum-attached-to-orth-proj}
  If $\pi(f) = 0$, then $\omega_\pi(f,\Psi) = 0$.  If $\pi(f)$ is the orthogonal projector onto a one-dimensional subspace $\mathbb{C} \varphi$ of $\pi$ with unit basis vector $\varphi$, then $\omega_\pi(f,\Psi) = \langle \varphi, \Psi \varphi \rangle$.
\end{example*}

\subsubsection{Quantum variance sums}\label{sec:35ac3e5744}
For $f \in C_c^\infty(\PB^\times_S)$, define the sesquilinear form $\mathcal{V}_f : \mathcal{A}_0^S \otimes \mathcal{A}_0^S \rightarrow \mathbb{C}$ by
\[\mathcal{V}_f(\Psi_1,\Psi_2)
  := \sum_{\pi \in A_0^S} L^{(S)}(\ad \pi,1) \omega_\pi( f,\Psi_1) \overline{ \omega_\pi( f,\Psi_2) }.
\]

\subsubsection{Proposed limiting variance}\label{sec:35ac3e5745}
For $f \in C_c^\infty(\PB^\times_S)$, define the sesquilinear form $\mathcal{M}_f : \mathcal{A}_0^S \otimes \mathcal{A}_0^S \rightarrow \mathbb{C}$ by requiring for $\Psi_1 \in \pi_1 \in A_0^S, \Psi_2 \in \pi_2 \in A_0^S$ that $\mathcal{M}_f(\Psi_1,\Psi_2) := 0$ unless $\pi_1 = \pi_2 =: \pi$, in which case
\[
  \mathcal{M}_f(\Psi_1,\Psi_2) := c_3 L^{(S)}(\pi,\tfrac{1}{2}) \int_{g \in \PB^\times_S} \langle \Ad(g) \mathfrak{S} f, \mathfrak{S} f \rangle_{L^2(\PB^\times_S)} \langle \pi(g) \Psi_1, \Psi_2 \rangle_{\PB^\times} \, d g
\]
where
\begin{equation}\label{eq:defn-of-c}
  c_3 :=
  \zeta_F^{(S)}(2)  \vol([\PB^\times])^{-1}.
\end{equation}
The integral converges absolutely (see \S\ref{sec:local-convergence-lemmas},
\S\ref{sec:local-integrals-for-rallis-ipf}).

\subsubsection{Thickening $\PB^\times$ inside $B$}\label{sec:heartsuit}

Fix, once and for all, a nonzero element $W_S \in C_c^\infty(F_S^\times)$.  For $\tau \in F^\times$, define the linear map $\heartsuit^{\tau} : C_c^\infty(\PB^\times_S) \rightarrow \mathcal{S}(B_S)$ by
\[
  \heartsuit^{\tau} f(x) := \frac{W_S(\tau \nr(x))}{|\tau \nr(x)|_S} 1_{B_S^\times}(x) f(\pr(x)),
\]
where $\pr : B_S^\times \rightarrow \PB^\times_S$ denotes the natural projection.

\subsection{Statement of main result}
\label{sec-4-3}
The statement involves the metaplectic group (\S\ref{sec:global-metaplectic-gp}) and the Weil representation (\S\ref{sec:weil-repn-global}).  For $s \in \Mp_2(F_S)$, we abbreviate $\rho^{\tau}(s) := \rho_{\Weil}^{\psi^{\tau},B}(s)$ and $\rho_0^\tau(s) := \rho_{\Weil}^{\psi^{\tau},B^0}(s)$; these operators act respectively on $\mathcal{S}(B_S)$ and $\mathcal{S}(B_S^0)$.  The operators $1 \otimes \rho_0^{\tau_i}(s)$ on $\mathcal{S}(B_S)$ are defined using the decomposition $B_S = F_S \oplus B_S^0$, as in \S\ref{sec:factorization-weil-repn}.
\begin{theorem}\label{thm:main-estimate-general-variance}
  There is a finite subset $X$ of $F^\times$ and a finite collection $(\eps_{\tau_1,\tau_2})_{\tau_1,\tau_2 \in X}$ of sesquilinear forms $\eps_{\tau_1,\tau_2} : \mathcal{S}(B_S) \otimes \mathcal{S}(B_S) \otimes \mathcal{A}_0^S \otimes \mathcal{A}_0^S \rightarrow \mathbb{C}$, depending only upon $F$, $\psi$, $S$ and $W_S$, with the following properties:
  \begin{enumerate}
  \item {\bf Relevance.}  For $f \in C_c^\infty(\PB^\times_S)$, one has the following identity of sesquilinear forms on $\mathcal{A}_0^S$:
    \begin{equation}\label{eq:relevance}
      \mathcal{V}_f
      = \mathcal{M}_f
      + \sum_{\tau_1,\tau_2 \in X} \eps_{\tau_1,\tau_2}(\heartsuit^{\tau_1} f, \heartsuit^{\tau_2} f, \cdot, \cdot).
    \end{equation}
  \item {\bf $\O_1(F)$-invariance.}
    \[
      \eps_{\tau_1,\tau_2}(\mathfrak{S} \phi_1, \phi_2, \Psi_1, \Psi_2) = \eps_{\tau_1,\tau_2}( \phi_1, \phi_2, \Psi_1, \Psi_2),
    \]
    \[
      \eps_{\tau_1,\tau_2}(\phi_1, \mathfrak{S}\phi_2, \Psi_1, \Psi_2) = \eps_{\tau_1,\tau_2}(\phi_1, \phi_2, \Psi_1, \Psi_2).
    \]
    
  \item {\bf $\SO(B_S^0)$-invariance.}  For $g_1,g_2 \in \PB^\times_S$,
    \[\eps_{\tau_1,\tau_2}(\Ad(g_1) \phi_1, \Ad(g_2) \phi_2, \rho_{\reg}(g_1) \Psi_1,
      \rho_{\reg}(g_2) \Psi_2)
    \]
    \[
      = \eps_{\tau_1,\tau_2}(\phi_1, \phi_2, \Psi_1, \Psi_2).
    \]
  \item {\bf Metaplectic invariance.}  For $s \in \Mp_2(F_S)$,
    \[
      \eps_{\tau_1,\tau_2}(\rho^{\tau_1}(s) \phi_1, \rho^{\tau_2}(s) \phi_2, \Psi_1, \Psi_2) = \eps_{\tau_1,\tau_2}(\phi_1, \phi_2, \Psi_1, \Psi_2).
    \]
    
  \item {\bf Main estimate.}  For $s \in \Mp_2(F_S)$,
    \[
      \eps_{\tau_1,\tau_2}((1 \otimes \rho_0^{\tau_1}(s)) \phi_1, (1 \otimes \rho_0^{\tau_2}(s)) \phi_2, \Psi_1, \Psi_2) \]
    \[
      \ll \Xi(s) \prod_{i=1,2} \Sob(\phi_i) \|\Psi_i\|_{L^1},
    \]
    where $\Xi$ denotes the Harish--Chandra function (\S\ref{sec:Xi-global}) and $\Sob(\phi_i)$ denotes a quantity that varies continuously with $\phi_i$ (see \S\ref{sec:some-asympt-notat}).  The implied constants and the uniformity in the continuity of $\Sob(.)$
    depend at most upon $F,\psi,S,W_S$.
  \end{enumerate}
\end{theorem}
The proof of Theorem \ref{thm:main-estimate-general-variance} occupies \S\ref{sec:proof-theor-refthm:m}.

\begin{remark}
  In applications of Theorem \ref{thm:main-estimate-general-variance}, the crucial assertions are the relevance and the main estimate.  The $\SO(B_S^0)$-invariance may be employed to obtain quantitatively stronger estimates, while the $\O_1(F)$-invariance and metaplectic invariance may be computationally convenient (see \cite{nelson-variance-II}).
\end{remark}
\begin{remark}
  The formulation of Theorem \ref{thm:main-estimate-general-variance} is independent of the choice of measures on $\PB_S^\times$ and on $[\PB^\times]$.
\end{remark}
\begin{remark}
  Theorem \ref{thm:main-estimate-general-variance} minus the ``main estimate'' is like a trace formula: $\mathcal{V}_f$ is a sum over automorphic forms, $\mathcal{M}_f$ is the ``identity'' or ``diagonal'' contribution, and the $\eps_{\tau_1,\tau_2}$ are the ``interesting'' contributions which one often wishes to show are small.  One difference is that $\mathcal{V}_f$ has a quadrilinear (rather than bilinear) dependence upon the automorphic forms $\varphi$.
\end{remark}
\begin{remark}
  Theorem \ref{thm:main-estimate-general-variance} likely extends to the split case $B = M_2(F)$ after incorporating contributions from the continuous spectrum into the definitions of \S\ref{sec:main-general-estimate-key-defns} and replacing $\|\Psi_i\|_{L^1}$ with $\|\htt^A \Psi_i\|_{L^1}$ for some fixed large enough $A > 0$.
\end{remark}

\subsection{Proof of Theorem \ref{thm:main-estimate-general-variance}}\label{sec:proof-theor-refthm:m}

\subsubsection{Measures}
\label{sec-4-5-1}
With a view to applications, we have formulated Theorem \ref{thm:main-estimate-general-variance} in a measure-independent fashion.  For the proof, it is convenient to take on $[\PB^\times]$ the Tamagawa measure, so that
\begin{equation}\label{eq:defn-of-c-2}
  c_3 = \frac{1}{2}  \zeta_F^{(S)}(2),
\end{equation}
and to fix measures on $\PB^\times_{\mathbb{A}}, \PB^\times_\mathfrak{p}$ and hence on $\PB^\times_S = \prod_{\mathfrak{p} \in S} \PB^\times_\mathfrak{p}$ as in \S\ref{sec:global-measures}.

\subsubsection{The $\heartsuit$ operator: local}
\label{sec-2-5}
Suppose temporarily (for \S\ref{sec-2-5} only) that $k$ is a local field, $\psi$ is a nontrivial unitary character of $k$, $B$ is a quaternion algebra over $k$, $G := B^\times / k^\times$, and $W \in C_c^\infty(k^\times)$.  Recall from \S\ref{sec:defn-local-omega} the definition of $\Omega$.  We define a linear map $\heartsuit : C_c^\infty(G) \rightarrow \Omega$ by
\begin{equation}
  \heartsuit f(t,x)
  :=
  \frac{W(t \nr(x))}{|t \nr(x)|}
  1_{B^\times}(x)
  f(x).
\end{equation}
(By abuse of notation, we write $f(x)$ for the value taken by $f$ at the image of $x$ under the natural projection $B^\times \rightarrow G$.)

By inspecting the definitions, one has the identities of maps $C_c^\infty(G) \rightarrow \Omega$
\[
  \mathfrak{S} \heartsuit = \heartsuit \mathfrak{S}, \quad \Ad(g) \heartsuit = \heartsuit \Ad(g) \text{ (for $g \in G$)}.
\]
By inspecting the definitions, one has for $y \in k^\times, b \in B^\times$ that
\begin{equation}\label{eqn:local-heartsuit-formula}
  \rho_{\Weil}(a(y)) \heartsuit f(\nr(b)^{-1},b)
  =
  |y| \heartsuit f(y \nr(b)^{-1}, b)
  =
  W(y) f(b).
\end{equation}
By the formula \eqref{eqn:inner-product-on-Omega-2} for $\|.\|_{\Omega}$, one obtains
\begin{equation}
  \label{eqn:local-norm-heartsuit-f}
  \|\heartsuit f\|_{\Omega} =
  \|f\|_{L^2(G)} \,
  \|W\|_{L^2(k^\times, |x|^{-1} \, d x)}.
\end{equation}

\subsubsection{The $\heartsuit$ operator: global\label{sec:heartsuit-global}}
\label{sec-3-6}
We revert to the global setting of \S\ref{sec:estimates-general-var}.
Let $\pi \in A_0^S$.
Recall from \S\ref{sec:similitudes-global} the definition of $\Omega$.  We define a linear map $\heartsuit : C_c^\infty(\PB^\times_S) \rightarrow \Omega$ by
\[
  \heartsuit f(t,x) := \frac{W_S(t_S \nr(x_S))}{|t_S \nr(x_S)|} 1_{B_S^\times}(x_S) f(x_S) \prod_{\mathfrak{p} \notin S} \vol(K_{\mathfrak{p}})^{-1} \phi_\mathfrak{p}^0(t_\mathfrak{p},x_\mathfrak{p}),
\]
where $\phi_\mathfrak{p}^0 \in \Omega_\mathfrak{p}$ is defined with respect to $R_\mathfrak{p}$ (see \S\ref{sec:dist-elem}).  This definition and that of \S\ref{sec:heartsuit} are obviously similar; we record their precise relationship below in \S\ref{sec-4-5-5}.

By \eqref{eqn:local-heartsuit-formula} and the lemma of \S\ref{sec:hecke-kernels-local}, one has for $y \in \mathbb{A}^\times$, $b \in B_\mathbb{A}^\times$ that
\begin{equation}\label{eqn:formulas-for-heartsuit-f-global}
  |y| \heartsuit f(
  y \nr(b)^{-1}, b)
  =
  W_S(y_S) f(b_S)
  \prod_{\mathfrak{p} \notin S}
  T_{y_\mathfrak{p}}(b_\mathfrak{p}),
\end{equation}
with $T_{y_\mathfrak{p}}$ as in \S\ref{sec:hecke-kernels-local}.  By combining \eqref{eqn:local-norm-heartsuit-f} with Lemma \ref{lem:norm-of-distinguished-vector-in-Omega-local} of \S\ref{sec:dist-elem}, one obtains
\begin{equation}\label{eqn:norm-of-heartsuit-f-global}
  \|\heartsuit f\|^2_{\Omega}
  =
  \|f\|_{L^2(\PB^\times_S)}^2
  \int_{t \in F_S^\times}
  |W_S|^2(t)
  \, \frac{d t}{|t|}
  \prod_{\mathfrak{p} \notin S}
  \frac{\vol(R_\mathfrak{p})}{\vol(K_\mathfrak{p})^2}.
\end{equation}

\begin{lemma*}\label{lem:main-term-eval-for-heartsuit}
  Let $\pi \in A_0^S$.  Let $\Psi_1,\Psi_2 \in \pi$ be $\prod_{\mathfrak{p} \notin S} K_\mathfrak{p}$-invariant vectors.  For $f \in C_c^\infty(\PB^\times_S)$, the quantity $\mathfrak{m}(\heartsuit f, \heartsuit f, \Psi_1, \Psi_2)$ (see \S\ref{sec:induction-omega}) equals
  \[
    c_2 L^{(S)}(\pi,\tfrac{1}{2}) \int_{g \in \PB_S^\times} \langle \Ad(g) \mathfrak{S} f, \mathfrak{S} f \rangle_{L^2(\PB_S^\times)} \langle \pi(g) \Psi_1, \Psi_2 \rangle_{\PB^\times} \,d g,
  \]
  where
  \begin{equation}\label{eq:defn-c2-frak-W}
    c_2
    := 
    \frac{1}{\zeta^{(S)}_F(2)}
    \left(\prod_{\mathfrak{p} \notin S}
    \frac{\vol(R_\mathfrak{p})}{\vol(K_\mathfrak{p})}\right)
    \int_{y \in F_S^\times}
    |W_S|^2(y)
    \, \frac{d y}{|y|}.
  \end{equation}
\end{lemma*}
\begin{proof}
  By the lemma of \S\ref{sec:induction-omega}, the polarization of \eqref{eqn:norm-of-heartsuit-f-global} and the commutativity $\heartsuit \Ad(g) = \Ad(g) \heartsuit$, the required identity holds if we replace $S$ with some possibly larger finite set of places $S' \supseteq S$.  To deduce the identity as written, we apply \eqref{eqn:local-rallis-ipf-unram-calc} (using \eqref{eqn:generic-at-each-place-if-not-1-diml} to verify its hypotheses).
\end{proof}

\subsubsection{A specific Eichler/Jacquet--Langlands lift}\label{sec:Phi-pi}
For $\pi \in A_0^S$, let $\Phi_\pi \in \pi_{\JL}$ denote the element of the Jacquet--Langlands lift of $\pi$ having the Fourier expansion $\Phi_\pi(n(x) a(y)) = \sum_{\tau \in F^\times} \psi(\tau x) W_\pi(\tau y)$, where the Whittaker function $W_\pi : \mathbb{A}^\times \rightarrow \mathbb{C}$ is given by $W_\pi(y) := W_S(y_S) \prod_{\mathfrak{p} \notin S} W_{\pi_\mathfrak{p}}^0(y_\mathfrak{p})$ (see \S\ref{sec:aut-forms-fourier-exp}, \S\ref{sec-3-3-2}).

\begin{lemma*}
  One has $\|\Phi_\pi\|^2 = c_1 L^{(S)}(\ad \pi,1)$, where
  \begin{equation}\label{eq:c-of-frak-W-defn}
    c_1
    :=
    \frac{2}{ \zeta_F^{(S)}(2)}
    \left(\prod_{\mathfrak{p} \notin S}
    \Delta_{\psi_{\mathfrak{p}}}^{-1/2}\right)
    \int_{y \in F_S^\times} |W_S|^2(y) \, \frac{d y}{|y|}.
  \end{equation}
  If $\pi, \pi ' \in A_0^S$ are distinct, then $\langle \Phi_{\pi}, \Phi_{\pi '} \rangle = 0$.
\end{lemma*}
\begin{proof}
  The conclusion in the case $\pi \neq \pi '$ is the multiplicity one theorem for $\PB^\times$ combined with the injectivity of $\pi \mapsto \pi_{\JL}$.  The formula \eqref{eq:c-of-frak-W-defn} is a consequence of the lemma of \S\ref{sec-3-3-2} and the corresponding local calculation (\S\ref{sec-2-3}).
\end{proof}

\subsubsection{The normalizing scalar}\label{sec:35ac3e5750}
Recall from \eqref{eq:c-of-frak-W-defn}, \eqref{eq:defn-c2-frak-W} and \eqref{eq:defn-of-c-2} the scalars $c_1,c_2,c_3$.  By the local volume formulas of \S\ref{sec:local-vol-formulas},
\begin{equation}\label{eqn:relation-c1-c2-c}
  c_1^{-1} c_2 = c_3.
\end{equation}

\subsubsection{Application of the pretrace formula}
\label{sec-4-5-2}
Recall the theta functions $\Theta(\phi,\Psi)$ attached in \S\ref{sec:non-traditional-theta-lifts} to each $\phi \in \Omega, \Psi \in \mathcal{A}_0$.
Let $f \in C_c^\infty(\PB^\times_S)$, $\Psi \in \mathcal{A}_0^S$.

\begin{lemma}\label{lem:abs-conv-of-silly-sum}
  $\sum_{\pi \in A_0^S} |\omega_\pi( f,\Psi)| \| \Phi_\pi \|_{L^p([\PGL_2])} < \infty$ for $p=2,\infty$.
\end{lemma}
\begin{proof}
  By \eqref{eq:rapid-convergence-of-pretrace-formula} and \eqref{eq:polynomial-growth-of-reps}, it suffices to show that $\| \Phi_\pi \|_{L^p([\PGL_2])} \ll C(\pi)^{O(1)}$.  The case $p=2$ follows from the lemma of \S\ref{sec:Phi-pi} and \eqref{eq:HL}.  The case $p=\infty$ reduces to the case $p=2$ by axioms (S2a) and (S3b) of \cite[\S2.4]{michel-2009}, wherein the quantities $\mathcal{S}_d(\Phi_\pi)$ may be estimated using \cite[\S3.2.5]{michel-2009}.  A direct proof of this convergence also follows by a rearrangement of the arguments given below.
\end{proof}

\begin{lemma}\label{lem:fourier-coefficients-of-nontraditional-theta-lifts}
  $\Theta(\heartsuit f,\Psi) = \sum_{\pi \in A_0^S} \omega_\pi( f,\Psi) \Phi_\pi$.
\end{lemma}
\begin{proof}
  Set $\Phi_1 := \Theta(\heartsuit f,\Psi)$ and $\Phi_2 := \sum_{\pi \in A_0^S} \omega_\pi( f,\Psi) \Phi_\pi$; we must show that $\Phi_1 = \Phi_2$.  Since $\Phi_1, \Phi_2$ are cuspidal, it will suffice to demonstrate the equality of their Whittaker functions $W_1, W_2 : \mathbb{A}^\times \rightarrow \mathbb{C}$ as defined in \S\ref{sec:aut-forms-fourier-exp}.  By the lemma of \S\ref{sec:four-expans} and \eqref{eqn:formulas-for-heartsuit-f-global}, we have
  \[
    W_1(y) = \sum_{\gamma \in \PB^\times} \int_{g \in [\PB^\times]} \Psi(g) W_S(y_S) f(g_S^{-1} \gamma g_S) \prod_{\mathfrak{p} \notin S} T_{y_\mathfrak{p}}(g_\mathfrak{p}^{-1} \gamma g_\mathfrak{p}) \,d g.
  \]
  The definition of $\Phi_\pi$ implies (using Lemma \ref{lem:abs-conv-of-silly-sum} to justify the interchange of summation with the Fourier integral over the compact group $\mathbb{A}/F$) that
  \[
    W_2(y) = \sum_{\pi \in A_0^S} \omega_\pi( f,\Psi) W_S(y_S) \prod_{\mathfrak{p} \notin S} W_{\pi_\mathfrak{p}}^0(y_\mathfrak{p}).
  \]
  Since $\Psi \in \mathcal{A}_0^S$, we have $\omega_\pi(f,\Psi) = 0$ for all $\pi \in A^S$ with $\pi \notin A_0^S$, so it suffices to establish for all $y \in \mathbb{A}^\times$, $g \in \PB^\times_{\mathbb{A}}$ the pointwise identity
  \begin{equation*}
    \sum_{\gamma \in \PB^\times} f(g_S^{-1} \gamma g_S) \prod_{\mathfrak{p} \notin S} T_{y_\mathfrak{p}}(g_\mathfrak{p}^{-1} \gamma g_\mathfrak{p}) = \sum_{\pi \in A^S}  \sum_{\varphi \in \mathcal{B}(\pi \cap \mathcal{A}^S)} \overline{\varphi}(g) \pi(f) \varphi(g) \prod_{\mathfrak{p} \notin S} W_{\pi_\mathfrak{p}}^0(y_\mathfrak{p}),
  \end{equation*}
  which follows from the pretrace formula \eqref{eq:cnsez1neo0} and the identity $W_{\pi_\mathfrak{p}}^0(y_\mathfrak{p}) = \lambda_{\pi_\mathfrak{p}}(T_{y_\mathfrak{p}})$ (see \eqref{eqn:key-local-identity-relating-whittaker-and-hecke}).
\end{proof}

\begin{remark*}
  Lemma \ref{lem:fourier-coefficients-of-nontraditional-theta-lifts} and its proof are in the spirit of arguments of Shimizu \cite[\S4]{MR0333081}, but we were unable to relate them precisely (e.g., by deducing one from the other).
\end{remark*}

\subsubsection{Some sesquilinear forms}\label{sec:35ac3e5754}
Define $\mathcal{V}, \mathcal{M}, \mathcal{E} : \Omega \otimes \Omega \otimes \mathcal{A}_0 \otimes \mathcal{A}_0 \rightarrow \mathbb{C}$ by requiring that for $\phi_1, \phi_2 \in \Omega$ satisfying $\phi_i[y] = \phi_i'[y] \otimes \phi_i''[y]$ with $\phi_i'[y] \in \mathcal{S}(\mathbb{A})$ and $ \phi_i''[y] \in \mathcal{S}(B_\mathbb{A}^0)$ for all $y \in \mathbb{A}^\times$, one has with the abbreviations $\theta_i := \theta_{\psi^{\tau_i}}(\phi_i'[\tau_i y])$ and $h_i := \theta_{\psi^{\tau_i}}(\phi_i''[\tau_i y],\Psi_i)$ that
\[
  \mathcal{V}(\phi_1,\phi_2,\Psi_1,\Psi_2) := c_1^{-1} \int_{y \in \mathbb{A}^\times / F^\times \mathbb{A}^{\times 2}} |y|^2 \frac { 1 } { 2^2 } \sum_{\tau_1,\tau_2 \in F^\times / F^{\times 2}} \langle \theta_1 h_1, \theta_2 h_2 \rangle \, d_2^\times y,
\]
\[
  \mathcal{M}(\phi_1,\phi_2,\Psi_1,\Psi_2) := c_1^{-1} \int_{y \in \mathbb{A}^\times / F^\times \mathbb{A}^{\times 2}} |y|^2 \frac { 1 } { 2^2 } \sum_{\tau_1,\tau_2 \in F^\times / F^{\times 2}} \langle \theta_1, \theta_2 \rangle \langle h_1, h_2 \rangle  \, d_2^\times y
\]
and $\mathcal{E} := \mathcal{V} - \mathcal{M}$, or equivalently,
\begin{equation}
  \mathcal{E}(\phi_1,\phi_2,\cdot ,\cdot )
  := c_1^{-1}
  \int_{y \in \mathbb{A}^\times / F^\times \mathbb{A}^{\times 2}} |y|^2
  \frac
  {
    1
  }
  {
    2^2
  }
  \sum_{\tau_1,\tau_2 \in F^\times / F^{\times 2}}
  \mathcal{E}_{\tau_1,\tau_2}(\phi_1[\tau_1 y], \phi_2[\tau_2
  y],\cdot,\cdot )  \, d_2^\times y,
\end{equation}
where $\mathcal{E}_{\tau_1,\tau_2} : \mathcal{S}(B_\mathbb{A}) \otimes \mathcal{S}(B_\mathbb{A}) \otimes \mathcal{A}_0 \otimes \mathcal{A}_0 \rightarrow \mathbb{C}$ is as in \S\ref{sec:main-estimate-for-equidistribution-pairs-theta}.  The definitions makes sense for the same reasons as in \S\ref{sec:induction-omega}.  The identity
\begin{equation}\label{eqn:c1-times-error-equals-bla}
  \mathcal{V}(\phi_1,\phi_2,\Psi_1,\Psi_2)
  = c_1^{-1} \langle \Theta(\phi_1, \Psi_1), \Theta(\phi_2, \Psi_2) \rangle.
\end{equation}
follows from \S\ref{sec:unfold-inner-product-nontraditional-theta-over-sl2} when $\phi$ is a pure tensor, hence in general by linearity.

\subsubsection{The main identities}\label{sec:main-identity}

\begin{proposition}\label{prop:after-extracting-main-term}
  Let $f \in C_c^\infty(\PB^\times_S)$ and $\Psi_1,\Psi_2 \in \mathcal{A}_0^S$.  Then, with $\mathcal{V}_f$ and $\mathcal{M}_f$ defined as in \S\ref{sec:35ac3e5744} and \S\ref{sec:35ac3e5745}, we have
  \begin{equation}\label{eq:V-equals-Vf}
    \mathcal{V}(\heartsuit f,\heartsuit f,\Psi_1,\Psi_2)
    = \mathcal{V}_f(\Psi_1,\Psi_2),
  \end{equation}
  \begin{equation}\label{eq:M-equals-Mf}
    \mathcal{M}(\heartsuit f,\heartsuit f,\Psi_1,\Psi_2)
    = \mathcal{M}_f(\Psi_1,\Psi_2),
  \end{equation}
  \begin{equation}\label{eq:Vf-equals-Mf-plus-E}
    \mathcal{V}_f(\Psi_1,\Psi_2)
    =
    \mathcal{M}_f(\Psi_1,\Psi_2)
    +
    \mathcal{E}(\heartsuit f, \heartsuit f,
    \Psi_1, \Psi_2).
  \end{equation}
\end{proposition}
\begin{proof}
  \eqref{eq:V-equals-Vf}: By \eqref{eqn:c1-times-error-equals-bla}, Lemma \ref{lem:fourier-coefficients-of-nontraditional-theta-lifts} of \S\ref{sec-4-5-2}, and the lemma of \S\ref{sec:Phi-pi},
  \begin{align*}
    \mathcal{V}(\heartsuit f,\heartsuit f,\Psi_1,\Psi_2)
    &=
      c_1^{-1} \langle   \Theta(\heartsuit f, \Psi_1),   \Theta(\heartsuit f,
      \Psi_2) \rangle
    \\
    &=
      c_1^{-1} \sum_{\pi_1,\pi_2 \in A_0^S}
      \left\langle
      \omega_{\pi_1}(f, \Psi_1),
      \omega_{\pi_2}(f, \Psi_2)
      \right\rangle
    \\
    &=
      \sum_{\pi \in A_0^S}
      L^{(S)}(\ad \pi,1)
      \omega_{\pi}( f, \Psi_1),
      \overline{\omega_{\pi}( f, \Psi_2)}
    \\
    &=
      V_f(\Psi_1,\Psi_2).
  \end{align*}

  \eqref{eq:M-equals-Mf}: by the lemma of \S\ref{sec-3-6} and \eqref{eqn:relation-c1-c2-c}.

  \eqref{eq:Vf-equals-Mf-plus-E}: by \eqref{eq:V-equals-Vf}, \eqref{eq:M-equals-Mf} and the definition of $\mathcal{E}$.
\end{proof}

\subsubsection{Completion of the proof}
\label{sec-4-5-5}
We now apply Proposition \ref{prop:main-error-estimate-global-adelic-general} (see \S\ref{sec-3-5-5}) and Proposition \ref{prop:after-extracting-main-term} to prove Theorem \ref{thm:main-estimate-general-variance}.  The purpose
of this final, purely technical part of the argument is to recast the content of those propositions in terms of $\mathcal{S}(B_S)$ rather than the less ``user-friendly'' space $\Omega$.

By weak approximation, we may choose a compact fundamental domain $Y \subset \mathbb{A}^\times / \mathbb{A}^{\times 2}$ for $\mathbb{A}^\times / F^{\times} \mathbb{A}^{\times 2}$ with the property that $y_\mathfrak{p} = 1$ for all $y \in Y$ and $\mathfrak{p} \in S$.  Choose a finite set $X \subseteq F^\times$ of representatives for the finite set
\begin{equation}\label{eq:elements-such-that-multiplying-by-something-in-Y-gives-squares-mod-units-outside-S}
  \{\tau \in F^\times / F^{\times 2} :
  \text{
    there exists }
  y \in Y
  \text{ so that }
  y_\mathfrak{p} \tau \in F_\mathfrak{p}^{\times 2}  \mathcal{O}_\mathfrak{p}^\times
  \text{ for all }
  \mathfrak{p} \notin S
  \}.
\end{equation}
For $y \in Y, \tau \in X$, let $\diamondsuit^{\tau y} : \mathcal{S}(B_S) \hookrightarrow \mathcal{S}(B_\mathbb{A})$ denote the map $\diamondsuit^{y \tau} \Phi := \Phi \otimes (\otimes_{\mathfrak{p} \notin S} \phi_\mathfrak{p}^0[\tau y_\mathfrak{p} ])$, where $\phi_{\mathfrak{p}}^0 \in \Omega_{\mathfrak{p}}$ denotes as usual the distinguished element.  For $f \in C_c^\infty(\PB_S^\times)$, one then has $\diamondsuit^{\tau y} \heartsuit^{\tau} f = \heartsuit f[\tau y]$.  Observe that for $\mathfrak{p} \notin S$ and $t \in F_\mathfrak{p}^\times$, one has $\phi_\mathfrak{p}^0[t] = 0$ unless $t \in F_\mathfrak{p}^{\times 2} \mathcal{O}_\mathfrak{p}^\times$.  It follows that for $\tau \in F^\times$ and $y \in Y$, one has $\heartsuit f[\tau y] = 0$ unless $\tau$ belongs to the set \eqref{eq:elements-such-that-multiplying-by-something-in-Y-gives-squares-mod-units-outside-S}, hence that
\begin{equation}\label{eq:penultimate-identity-before-proving-main-general-thm}
  \mathcal{E}(\heartsuit f, \heartsuit f, \cdot, \cdot)
  =
  \frac{c_1^{-1}}{2^2}
  \int_{y \in Y}
  |y|^2
  \sum_{\tau_1,\tau_2 \in X}
  \mathcal{E}_{\tau_1,\tau_2}(\diamondsuit^{\tau_1 y}
  \heartsuit^{\tau_1} f,
  \diamondsuit^{\tau_2 y} \heartsuit^{\tau_2} f,
  \cdot,\cdot) \, d_2^\times y.
\end{equation}
Define $\eps_{\tau_1,\tau_2} : \mathcal{S}(B_S) \otimes \mathcal{S}(B_S) \otimes \mathcal{A}_0^S \otimes \mathcal{A}_0^S \rightarrow \mathbb{C}$ by
\begin{equation}\label{eq:defn-of-the-eps-guys-yay}
  \eps_{\tau_1,\tau_2}(\Phi_1,\Phi_2,\Psi_1,\Psi_2)
  :=
  \frac{c_1^{-1}}{2^2}
  \int_{y \in Y}
  |y|^2
  \mathcal{E}_{\tau_1,\tau_2}(\diamondsuit^{\tau_1 y}
  \Phi_1,
  \diamondsuit^{\tau_2 y} \Phi_2,
  \Psi_1,\Psi_2) \, d_2^\times y.
\end{equation}
We verify the assertions made in Theorem \ref{thm:main-estimate-general-variance}:
\begin{enumerate}
\item The ``relevance'' follows from \eqref{eq:penultimate-identity-before-proving-main-general-thm}, \eqref{eq:defn-of-the-eps-guys-yay} and Proposition \ref{prop:after-extracting-main-term}.
\item Since $\mathfrak{S} \phi_\mathfrak{p}^0 = \phi_\mathfrak{p}^0$, one has $\mathfrak{S} \diamondsuit^{\tau_i y} = \diamondsuit^{\tau_i y} \mathfrak{S}$.  For $g \in \PB^\times_S$ and $s \in \Mp_2(F_S)$, one has $\Ad(g) \diamondsuit^{\tau_i y} = \diamondsuit^{\tau_i y} \Ad(g)$ and $\rho^{\tau_i}(s) \diamondsuit^{\tau_i y} = \diamondsuit^{\tau_i y} \rho^{\tau_i}(s)$.  Thus the ``$\O_1(F)$-invariance,'' ``$\SO(B_S^0)$-invariance'' and ``metaplectic invariance'' follow from \S\ref{sec:equivariance-summary-for-E-tau-tau}.
\item The ``main estimate'' is the content of Proposition \ref{prop:main-error-estimate-global-adelic-general}.
\end{enumerate}

\subsection{Classicalization}
\label{sec-4-4}
We now specialize to the setting of Theorem \ref{sec-1} and record how Theorem \ref{thm:main-estimate-general-variance} specializes to Theorem \ref{thm:var-2}.

\subsubsection{Specialization to a single place\label{sec:general-estimates-specialized-single-place}}
\label{sec-4-4-1}
We specialize the definitions of \S\ref{sec:main-general-estimate-key-defns} to the case that ramification is concentrated at a single place $\mathfrak{q}$ of $F$, finite or infinite.  This is the case required for the proof of Theorem \ref{thm:var-3} as well as its non-archimedean analogue pursued in \cite{nelson-variance-II}.

Assume that $S$ is the set of places $\mathfrak{p}$ for which either
\begin{itemize}
\item $\mathfrak{p}$ is infinite,
\item $\mathfrak{p}$ is a finite place at which $B$ ramifies, or
\item $\mathfrak{p} = \mathfrak{q}$.
\end{itemize}
Assume that for each $\mathfrak{p} \in S - \{\mathfrak{q}\}$, the completion $B_\mathfrak{p}$ is non-split, or equivalently, that $\PB^\times_\mathfrak{p}$ is compact.  There are the following possibilities:
\begin{enumerate}
\item $\mathfrak{q}$ is real, in which case $F$ is totally real and $B$ ramifies at every infinite place other than $\mathfrak{q}$.
\item $\mathfrak{q}$ is complex, in which case $F$ is real and $B$ ramifies at every infinite place other than $\mathfrak{q}$.
\item $\mathfrak{q}$ is finite, in which case $F$ is totally real and $B$ is totally definite.
\end{enumerate}
For each place $\mathfrak{p}$, define the compact open subgroup $J_\mathfrak{p} \leq \PB^\times_\mathfrak{p}$ as in \S\ref{sec-2-4} by taking for $J_\mathfrak{p}$ the image of $R_\mathfrak{p}^\times$ if $\mathfrak{p}$ is finite and taking $J_\mathfrak{p} := \PB^\times_\mathfrak{p}$ if $\mathfrak{p}$ is infinite.  Set
\begin{equation*}
  J := \prod_{\mathfrak{p} \neq \mathfrak{q}} J_\mathfrak{p}.
\end{equation*}
In addition to the notation of \S\ref{sec-4-1}, we now introduce a superscripted $J$, as in $\mathcal{A}^J, \mathcal{A}_0^J, \pi^J$ to denote the $J$-fixed subspace.  We denote by $\mathcal{A}^J_+ \subseteq \mathcal{A}^J$, $\mathcal{A}^J_{0+} \subseteq \mathcal{A}_0^J$ the ``even'' subspaces consisting of $\varphi$ that are $\PB^\times_{\mathfrak{p}}$-invariant for all $\mathfrak{p} \in S - \{\mathfrak{q}\}$.  Thus, for instance, $\mathcal{A}_{0+}^J \subseteq \mathcal{A}_0^J \subseteq \mathcal{A}_0^S \subseteq \mathcal{A}_0 \subseteq \mathcal{A}$.  We denote by $A_0, A^J, A_0^J, A^J_+, A^J_{0+}$ the set of all $\pi \in A$ having nonzero intersection with the space having the corresponding scripted notation.

Set $G := \PB^\times_\mathfrak{q}$, and let $f \in C_c^\infty(G)$.  For $\mathfrak{p} \in S - \{q\}$, set
\begin{equation*}
  e_{J_\mathfrak{p}} := \vol(J_\mathfrak{p})^{-1} 1_{J_\mathfrak{p}} \in C_c^\infty(\PB^\times_\mathfrak{p}).
\end{equation*}
Define $\tilde{f} \in C_c^\infty(\PB^\times_S)$ by the formula
\[
  \tilde{f} (g) := f(g_\mathfrak{q}) \prod_{\mathfrak{p} \in S - \{\mathfrak{q} \}} e_{J_\mathfrak{p}}(g_\mathfrak{p}).
\]
For $\pi \in A_0^J$ and $\Psi \in \mathcal{A}_0^J$, set
\begin{equation*}
  \omega_\pi(f,\Psi) := \sum_{\varphi \in \mathcal{B}(\pi \cap \mathcal{A}^J)} \langle \varphi, \Psi \cdot \pi(f) \varphi \rangle.
\end{equation*}
Then $\omega_\pi(\tilde{f},\Psi) = \omega_\pi(f,\Psi)$ (see \S\ref{sec:omega-pi} for the definition of the left hand side).  Since $\dim(\pi_\mathfrak{p}) = 1$ for all $\mathfrak{p} \in S - \{\mathfrak{q} \}$, one has for $\Psi \in \pi' \in \mathcal{A}_0^J$ that $\omega_\pi(f,\Psi) = 0$ unless $\pi ' \in \mathcal{A}_{0+}^J$.

Let $V_{f}, M_{f} : \mathcal{A}_{0+}^J \otimes \mathcal{A}_{0+}^J \rightarrow \mathbb{C}$ denote the sesquilinear forms obtained by restricting the forms $\mathcal{V}_{\tilde{f}}, \mathcal{M}_{\tilde{f}} : \mathcal{A}_0^S \otimes \mathcal{A}_0^S \rightarrow \mathbb{C}$.  Then
\begin{equation}
  V_{f}(\Psi_1,\Psi_2)
  =
  \sum_{\pi \in A_0^J}
  L^{(S)}(\ad \pi,1)
  \omega_\pi(f,\Psi_1)
  \overline{\omega_\pi(f,\Psi_2)}.
\end{equation}
By the observation that $\mathfrak{S} e_{J_\mathfrak{p} } = e_{J_\mathfrak{p}}$ for $\mathfrak{p} \in S - \{\mathfrak{q}\}$ and the local calculations \eqref{eqn:local-rallis-ipf-integral-nonsplit-finite} and \eqref{eqn:local-rallis-ipf-integral-nonsplit-real}, we see that for $\Psi_1 \in \pi_1 \in A_{0+}^J$ and $\Psi_2 \in \pi_2 \in A_{0+}^J$, we have $M_{f}(\Psi_1,\Psi_2) = 0$ unless $\pi_1 = \pi_2 =: \pi$, in which case
\begin{equation}
  M_{f}(\Psi_1,\Psi_2)
  =
  c_4
  L^{(S)}(\pi,\tfrac{1}{2})
  \int_{g \in G}
  \langle \Ad(g) \mathfrak{S} f, \mathfrak{S} f \rangle_{L^2(G)}
  \langle \pi(g) \Psi_1, \Psi_2 \rangle_{\PB^\times} \,d g 
\end{equation}
where
\begin{equation}\label{eq:defn-of-c4}
  c_4 := 2^t
  \zeta_F^{(S)}(2) \vol([\PB^\times])^{-1}.
\end{equation}
with $t$ the number of finite primes $\mathfrak{p} \in S - \{\mathfrak{q}\}$.

\subsubsection{Strong approximation\label{sec:strong-approx}}
\label{sec-4-4-2}
Retaining the notation of \S\ref{sec:general-estimates-specialized-single-place}, we record here how the quotient $[\PB^\times]/J$ unadelizes under some assumptions.  Recall that $G := \PB^\times_{\mathfrak{q}}$.  Let $\Gamma \leq G$ denote the image of $\PB^\times \cap J$ under the inclusion $\PB^\times \hookrightarrow G$.  Then $\Gamma$ is a discrete cocompact subgroup of $G$, and the natural map $\iota : \Gamma \backslash G \rightarrow [\PB^\times]/J$
is injective.  We record a standard consequence of strong approximation.
\begin{lemma*}\label{lem:consequence-of-strong-approx}
  Suppose that $F$ has odd narrow class number and either that
  \begin{enumerate}
  \item $B_\mathfrak{q}$ is split, or that
  \item $\mathfrak{q}$ is infinite and $B$ has class number one.
  \end{enumerate}
  Then $\iota$ is bijective.
\end{lemma*}

\subsubsection{Proof of Theorem \ref{thm:var-2}}\label{sec:deduct-theor-refthm}
We assume now that $\mathfrak{q}$ is archimedean, so that the above setting recovers that of \S\ref{sec:intro-setup}.  We deduce Theorem \ref{thm:var-2} by specializing (parts of) Theorem \ref{thm:main-estimate-general-variance}, following the unadelization procedure discussed above.  For $f \in C_c^\infty(\PB_\mathfrak{q}^\times)$, we have
  \begin{equation*}
    \mathcal{V}(f) =
    V_f(\Psi_1,\Psi_2)
    =
    V_{\tilde{f}}(\Psi_1,\Psi_2),
  \end{equation*}
  and similarly for $\mathcal{M}(f)$.  Here $\mathcal{V}(f)$ is as defined in \S\ref{sec:orgb2d0a08}, while $\mathcal{V}_f$ and $\mathcal{V}_{\tilde{f}}$ are as in \S\ref{sec-4-4-1}.  We identify $\tau \in F ^\times$ with its image in $\mathbb{R}^\times$ via the given archimedean place $\mathfrak{q}$.  Inspecting the definitions, we have $\heartsuit^{\tau} \tilde{f} = (\heartsuit^{\tau} f) \otimes \Phi^{\tau}$, where $\heartsuit^{\tau} f \in \mathcal{S}(B_\mathfrak{q})$ is defined by $\heartsuit^{\tau} f (g) := W(\tau \nr(g)) f(g)$ and $\Phi^{\tau} = \otimes_{\mathfrak{p} \in S - \{\mathfrak{q} \}} \Phi_\mathfrak{p}^{\tau}$ for some $\Phi_\mathfrak{p}^{\tau} \in \mathcal{S} (B_\mathfrak{p} )$ not depending upon $f$.  The required identity \eqref{eq:var-2-main-identity} then holds with $\mathcal{E}_{\tau_1,\tau_2}(\phi) := \eps_{\tau_1,\tau_2} (\phi \otimes \Phi^{\tau_1}, \phi \otimes \Phi^{\tau_2}, \Psi_1, \Psi_2)$.

  To deduce the claimed error estimate \eqref{eqn:main-estimate}, we specialize the main estimate \eqref{eqn:main-estimate} of Theorem \ref{thm:main-estimate-general-variance} to the case that $s \in \Mp_2(F_S)$ lifts the diagonal element $\diag(y,y^{-1})$ of $\SL_2(F_\mathfrak{q})$.  Then $\rho_0^\tau(s)$ (\S\ref{sec:local-weil-repn}) acts on $\mathcal{S}(B_S)$ via the factor $\mathcal{S}(B_\mathfrak{q})$, where it acts by a constant multiple of the dilation operator $D_y$ defined in \S\ref{sec:orgb2d0a08}.  The relevant estimates for $\Xi(s)$ were recorded in \S\ref{sec:local-Xi}.

\part{Application to microlocal lifts}\label{part:appl-micr-lifts}
Our aim is now to prove Theorem \ref{thm:var-3} by application of Theorem \ref{thm:var-2}.  We retain the general notation of \S\ref{sec:intro-setup} (in particular, $G = \PGL_2(\mathbb{R})$,  $\Gamma < G$  is an arithmetic subgroup, $S$ is the finite set of finite primes of the number field $F$ at which the quaternion algebra $B$ ramifies, and $A_0$ is the set of nontrivial irreducible representations $\pi \subseteq L^2(\Gamma \backslash G)$ of both the group $G$ and the Hecke operators $T_p$) and the asymptotic notation and conventions of \S\ref{sec:sketch-deduct-theor} (in particular, the conventions concerning ``$\h$-dependent elements'' and ``fixed'').

The microlocal lift measures $\mu_\pi$ are defined (\S\ref{sec:constr-mu-pi-ZW}) using differential operators (namely, raising and lowering operators), but our methods apply most naturally to the distributions $\mu(\pi(f), \cdot)$ attached  in \S\ref{sec:orgb2d0a08} to integral operators $\pi(f)$, $f \in C_c^\infty(G)$.  Our first main task is thus to give an alternative construction of the microlocal lift measures using integral operators.  One input here is the microlocal calculus for Lie group representations developed in \cite{nelson-venkatesh-1}, whose specialization to $\PGL_2(\mathbb{R})$ we recall in \S\ref{sec:operator-calculus}.  Having constructed microlocal lift measures in this way, we are in good position to  apply Theorem \ref{thm:var-2}, and must then estimate the ``main'' and ``error'' terms that arise.  We refer back to \S\ref{sec:sketch-deduct-theor} for a more detailed survey of Part \ref{part:appl-micr-lifts}.

\section{Archimedean preliminaries}\label{sec:constr-irred-unit}

\subsection{Lie algebra}\label{sec:lie-algebra}
Let $\mathfrak{g}$ denote the Lie algebra of $G$.  We denote by $\mathfrak{g}_\mathbb{C} \cong \slLie_2(\mathbb{C})$ its complexification and by $\mathfrak{g}_\mathbb{C}^*$ the complex dual.  We will often identity $\mathfrak{g}_\mathbb{C}$ with the space of linear functions on $\mathfrak{g}_\mathbb{C}^*$.  We work with the following basis elements for $\mathfrak{g}_{\mathbb{C}}$:
\[
  X := \frac{1}{2 i}
  \begin{pmatrix}
    1 & i  \\
    i & -1
  \end{pmatrix}, \quad Y := \frac{1}{2 i}
  \begin{pmatrix}
    1  & -i \\
    -i & -1
  \end{pmatrix}, \quad W := \frac{1}{2 i}
  \begin{pmatrix}
    0  & 1 \\
    -1 & 0
  \end{pmatrix}.
\]
These satisfy $[X,Y] = - 2W$, $[W,X]= X$ and $[W,Y]= -Y$.  The map $\theta \mapsto e^{i \theta W}$ defines an isomorphism from $\mathbb{R}/2 \pi \mathbb{Z}$ to $K^1$.  The complex conjugation on $\mathfrak{g}_\mathbb{C}$ is given by $-\overline{X} = Y$ and $- \overline{W} = W$.

The center of the universal enveloping algebra of $\mathfrak{g}_\mathbb{C}$ is the one variable polynomial ring $\mathbb{C}[\Omega]$, where
\[
  \Omega := W^2 - \frac{X Y + Y X}{2} = W(W-1) - X Y = W(W+1) - Y X.
\]
The ring $\Sym(\mathfrak{g}_\mathbb{C})^G$ of $G$-invariant polynomials on $\mathfrak{g}_\mathbb{C}^*$ is generated by the polynomial $\Lambda := W^2 - X Y$.  The Harish--Chandra isomorphism $\mathbb{C}[\Omega] \xrightarrow{\cong } \mathbb{C}[\Lambda]$ is given in this case by $\Omega \mapsto \Lambda - 1/4$.

We identify $\mathfrak{g}_\mathbb{C}^*$ with $\mathfrak{g}_\mathbb{C}$ via the trace pairing $(x,\xi) \mapsto \langle x, \xi \rangle := \trace(x \xi)$.  We identify the real and imaginary duals $\mathfrak{g}^*$ and $i \mathfrak{g}^*$ of $\mathfrak{g}$ with the subspaces of $\mathfrak{g}_\mathbb{C}^*$ taking real and imaginary values on $\mathfrak{g}$, respectively.  We abbreviate $\mathfrak{g}^\wedge := i \mathfrak{g}^*$; it identifies with the Pontryagin dual of $\mathfrak{g}$ via the natural pairing $\mathfrak{g} \times \mathfrak{g}^\wedge \ni (x,\xi) \mapsto e^{\langle x, \xi \rangle} \in \mathbb{C}^{(1)}$.  We occasionally work with the coordinates and basis elements
\[
  \mathfrak{g} \ni x = \begin{pmatrix}
    x_1/2 & x_2    \\
    x_3 & -x_1/2
  \end{pmatrix}
  = \sum_{j=1,2,3} x_j e_j,
\]
\[
  \mathfrak{g}^\wedge \ni \xi = i
  \begin{pmatrix}
    \xi_1 & \xi_3  \\
    \xi_2 & -\xi_1
  \end{pmatrix}
  = \sum_{j=1,2,3} \xi_j e_j^*,
\]
so that the natural pairing is given by
\begin{equation*}
  (x,\xi) \mapsto e^{i \sum x_j \xi_j}.
\end{equation*}
We note that the invariant polynomial $\Lambda$ is given in this optic by $\Lambda(\xi) = \det(\xi/i) = - \xi_1^2 - \xi_2 \xi_3$.

The following elements, defined for $t \in \mathbb{R}_{\neq 0}$, will occur frequently in our analysis:
\begin{equation}\label{eqn:defn-xi-of-t}
  \xi(t) :=
  i \begin{pmatrix}
    t &  \\
    & -t
  \end{pmatrix} \in \mathfrak{g}^\wedge.
\end{equation}
We note that
\begin{equation*}
  X(\xi(t)) = Y(\xi(t)) = t,
\end{equation*}
\begin{equation*}
  W(\xi(t)) = 0, \qquad \Lambda(\xi(t)) = -t^2.
\end{equation*}
We note also that the $G$-stabilizer of $\xi(t)$ is the diagonal subgroup $H$.

We fix norms $|.|$ on all of the above spaces.

\subsection{Coadjoint orbits}\label{sec:coadjoint-orbits}
A \emph{coadjoint orbit} $\mathcal{O}$ is a $G$-orbit on $\mathfrak{g}^\wedge$; in particular, it is a smooth manifold.  The origin $\{0\}$ is a zero-dimensional coadjoint orbit.  The other coadjoint orbits are two-dimensional and of the form
\[
  \mathcal{O}(\lambda) := \{\text{$\xi \in\mathfrak{g}^\wedge - \{0\}$ with $\Lambda(\xi) = \lambda$}\}
\]
for some $\lambda \in \mathbb{R}$.  If $\lambda = 0$, then $\mathcal{O}(\lambda)$ is the regular subset of the nilcone; if $\lambda > 0$, it is a two-sheeted hyperboloid; if $\lambda < 0$, it is a one-sheeted hyperboloid.  For $t \neq 0$, the orbit of $\xi(t)$ is $\mathcal{O}(-t^2)$.

We equip any two-dimensional coadjoint orbit $\mathcal{O}$ with its normalized symplectic measure
\begin{equation*}
C_c(\mathfrak{g}^\wedge) \ni a \mapsto \int_{\mathcal{O}} a := \int_{\xi \in \mathcal{O}} a(\xi) \, d \omega(\xi),
\end{equation*}
defined by the $2$-form $\omega$ on $\mathcal{O}$ described as follows (see, e.g., \cite[\S6]{nelson-venkatesh-1} or \cite{MR1701415} for further details, and the calculations of \S\ref{sec:orge0a88b2} for some explicit formulas).  For each $\xi \in \mathcal{O}$, the tangent space $T_\xi \mathcal{O}$ identifies with the space of vectors $\{x \cdot \xi : x \in \mathfrak{g}\}$, where $x \cdot \xi := \ad^*(x) \xi \in \mathfrak{g}^\wedge$ is defined by differentiating the action of $G$ on $\mathfrak{g}^\wedge$.  The component $\omega_\xi$ of $\omega$ at $\xi$ is then given by
\begin{equation*}
  \omega_{\xi}(x \cdot \xi, y \cdot \xi) := \langle \xi, [x,y] \rangle/2 \pi i.
\end{equation*}
For each $a \in C_c(\mathfrak{g}^\wedge)$, the function $\mathbb{R} \ni \lambda \mapsto \int_{\mathcal{O}(\lambda)} a$ is continuous and compactly-supported (see, e.g., \cite[\S11.2]{nelson-venkatesh-1}).  The rescaling $\h \mathcal{O}$ is also a coadjoint orbit, and we have
\begin{equation*}
  \int_{\xi \in \h \mathcal{O}} a(\xi) \, d \omega (\xi ) = \h \int_{\xi \in \mathcal{O}} a(\h \xi) \, d \omega (\xi ).
\end{equation*}
We record a simple estimate:
\begin{lemma*}
  Let $\tau \in \mathfrak{g}^\wedge$ with $|\tau| \asymp 1$, and let $0 < r \leq 1$.  For any two-dimensional coadjoint orbit $\mathcal{O}$, the symplectic volume of the subset $\{\xi \in \mathcal{O} : |\xi - \tau| < r\}$ of $\mathcal{O}$ is $\O(r^2)$.
\end{lemma*}
\begin{proof}
  Each $\tau \in \mathfrak{g}^\wedge - \{0\}$ is regular, i.e., the differential of the invariant polynomial $\Lambda$ is nonzero at $\tau$.  On a small neighborhood of each such $\tau$, we may thus find local coordinates $(\tau_1,\tau_2,\Lambda)$ with respect to which the coadjoint orbits are the fibers of the projection onto the third coordinate, with the symplectic measures given by smooth multiples of Lebesgue measure in the first two coordinates.  The required estimate follows for $\tau$ in a small fixed neighborhood of each fixed element of $\mathfrak{g}^\wedge - \{0\}$, then in general by continuity and compactness.
\end{proof}

\subsection{Representations}\label{sec:prelims-representations}
Let $\pi$ be an irreducible unitary representation of $G$.  Then $\Omega$ acts on the smooth subspace of $\pi$ by some real scalar $\Omega_{\pi}$.  We set $\lambda_\pi := 1/4 + \Omega_\pi$, and refer to it as the infinitesimal character of $\pi$.  Up to isomorphism, we may classify $\pi$ as follows:
\begin{itemize}
\item The one-dimensional representations (the trivial representation $\mathbb{C}$ and the sign representation $\mathbb{C} (\sgn \circ \det)$), for which $\Omega_\pi = 0$ and $\lambda_\pi = 1/4$.
\item The discrete series representations $\pi(k)$ ($k \in \mathbb{Z}_{\geq 1}$), for which $\Omega_\pi = k(k-1)$ and $\lambda_\pi = (k-1/2)^2$.  (We note that $\pi(k)$ is often denoted $\mathcal{D}_{2k}$.)
\item The unitary principal series representations $\pi(t,\eps)$, with
  \begin{enumerate}[(i)]
  \item $t \in \mathbb{R}$ and $\eps \in \{\pm 1\}$ or
  \item $t \in i(-1/2,1/2) - \{0\}$ and $\eps = 1$ (the ``complementary series''),
  \end{enumerate}
  obtained by normalized parabolic induction of the character $\diag(y,1) \mapsto \sgn(y)^\eps |y|^{i t}$, for which $\Omega_\pi = -1/4 - t^2$ and $\lambda_\pi = - t^2$.
\end{itemize}
The only equivalences are that $\pi(t,\eps) \cong \pi(-t,\eps)$.  The tempered irreducibles are the $\pi(k)$ and $\pi(t,\eps)$ with $t \in \mathbb{R}$.

Suppose that $\pi$ is not one-dimensional.  We may then realize it as follows.  If $\pi = \pi(t,\eps)$, set $Q := \mathbb{Z}$; if $\pi = \pi(k)$, set $Q := \{q \in \mathbb{Z} : |q| \geq k\}$ and $\eps := 1$.  We regard $L^2(Q)$ as a Hilbert space with respect to the counting measure, with basis elements given by the $\delta$-masses $e_q$ at each $q \in Q$.  It contains the dense subspace $C_c(Q)$ consisting of the finitely-supported elements.  We verify readily that the following formulas define an infinitesimally unitary $(\mathfrak{g},K)$-module structure on $C_c(Q)$, corresponding to a representative for the isomorphism class of $\pi$:
\begin{equation*}
  X e_q = ( q(q+1) - \Omega_\pi)^{1/2} e_{q+1},
\end{equation*}
\[
  Y e_{q+1} = (q(q+1) - \Omega_\pi)^{1/2} e_{q},
\]
\[
  W e_q = q e_q, \quad e^{i \theta W} e_q = e^{i \theta q} e_q, \quad \diag(-1,1) e_q = (-1)^{\eps} e_{-q}.
\]

\subsection{Kirillov formula}\label{sec:kirillov-formula}
The character of an irreducible representation $\pi$ of $G$ is a generalized function $\chi_\pi : G \rightarrow \mathbb{C}$ (see, e.g., \cite[\S X]{MR855239}).  Fix a sufficiently small open neighborhood $\mathcal{G}$ of the origin in $\mathfrak{g}$.  The normalized Jacobian of the exponential map is the function $\jac : \mathcal{G} \rightarrow \mathbb{R}_{>0}$ for which
\begin{itemize}
\item $\jac(0) = 1$, and
\item if $d g$ is any Haar measure on $G$, then there is a unique Haar measure $d x$ on $\mathfrak{g}$ so that for $g = \exp(x)$ with $x \in \mathcal{G}$, we have $d g = \jac(x) d x$.  We say in this case that $d g$ and $d x$ are \emph{compatibly normalized}.
\end{itemize}
\begin{lemma*}
  Let $\pi$ be a tempered irreducible unitary representation of $G$.  Set $\mathcal{O}_\pi := \mathcal{O}(\lambda_\pi)$.  For $x \in \mathcal{G}$, we have the identity of generalized functions
  \[
    \chi_\pi(\exp(x)) = \jac(x)^{-1/2} \int_{\xi \in \mathcal{O}_\pi} e^{\langle x, \xi \rangle} \, d \omega (\xi ).
  \]
\end{lemma*}
See, e.g., \cite[\S6]{nelson-venkatesh-1} and references.  This says concretely that for each $\phi \in C_c^\infty(\mathcal{G})$ and Haar measure $d x$ on $\mathfrak{g}$, the operator $\int_{x \in \mathfrak{g}} \phi(x) \pi(\exp(x)) \, d x$ on $\pi$ belongs to the trace class and has trace $\int_{\xi \in \mathcal{O}_\pi} (\int_{x \in \mathfrak{g}} \phi(x) \jac(x)^{-1/2} e^{\langle x, \xi \rangle} \, d x) \, d \omega (\xi )$.

\subsection{Construction of $\mu_\pi$}\label{sec:constr-mu-pi-ZW}
Let $\pi \in A_0$ with $\lambda_\pi < 0$.  Then $\pi \cong \pi(t,\eps)$ with $t = \sqrt{-\lambda_\pi} > 0$.  Recall that we have chosen a unit vector $\varphi_\pi \in \pi$ invariant by $K^1$.  The microlocal lift $\mu_\pi$ of $\pi$ may be defined on $K$-finite smooth functions $\Psi : \Gamma \backslash G \rightarrow \mathbb{C}$ as follows (see for instance \cite[\S2--3]{MR1859345}, which gives the same definition up to notational differences).  Set $\varphi_0 := \varphi_\pi$ and $s := 1/2 + i t$.  Define $\varphi_q$ for $q \in \mathbb{Z}$ recursively by the formulas $i X \varphi_q = (s + q) \varphi_{q+1}$ and $i Y \varphi_q = (s - q) \varphi_{q-1}$.  Then
\[
  \mu_\pi(\Psi) := \sum_{q \in \mathbb{Z}} \langle \varphi_0 \Psi, \varphi_q \rangle.
\]

\subsection{Branching coefficients}\label{sec:branch-coeff}
Let $\pi, \sigma \in A_0$.
\begin{lemma}\label{lem:trivial-vanishing}
  If $\sigma$ is not even, then $\langle \varphi_1 \Psi, \varphi_2 \rangle = 0$ for all $\varphi_1, \varphi_2 \in \pi$ and $\Psi \in \sigma$.  In particular, $\mu_\pi(\Psi) = 0$.
\end{lemma}
We give the proof below after some otherwise relevant preliminaries.

Assume temporarily that for each $p \in S$, the involutory Hecke operator $T_{p}$ acts trivially (i.e., with eigenvalue $+1$ rather than $-1$) on $\sigma$.  The triple product formula \cite{MR2449948} then implies that for eigenfunctions $\varphi_1, \varphi_2 \in \pi$ and $\Psi \in \sigma$,
\begin{equation}\label{eq:branching-coeff-defn}
  |\langle \varphi_1 \Psi, \varphi_2  \rangle|^2
  =
  \mathcal{L}(\pi,\sigma)
  \int_{g \in G}
  \langle g \varphi_1, \varphi_1 \rangle
  \overline{\langle g \varphi_2, \varphi_2 \rangle}
  \langle g \Psi, \Psi  \rangle \, d g,
\end{equation}
where $\mathcal{L}(\pi,\sigma)$ is nonnegative real given explicitly in terms of special values of $L$-functions; in particular,
\begin{equation}\label{eq:defn-cal-L}
  \mathcal{L}(\pi,\sigma)
  \asymp
  \frac{L(\pi \otimes \overline{\pi } \otimes \sigma,
    \tfrac{1}{2})}{
    L(\ad \pi,1)^2 L(\ad \sigma,1)},
\end{equation}
where $L(\dotsb)$ denotes the finite part of an $L$-function.  The implied constant in \eqref{eq:defn-cal-L} may be made explicit, and depends at most upon $F$, $B$.

\begin{remark}
  We record in more detail how to obtain the above specialization of the triple product formula.  We may assume that each automorphic form $\varphi_j$ and $\Psi$ is nonzero.  The cited reference then gives \eqref{eq:branching-coeff-defn} with $\mathcal{L}(\pi, \sigma)$ given by the right hand side of \eqref{eq:defn-cal-L} times a product $c_0 \left( \prod_{v \in S_\infty - \mathfrak{q}} c_v \right)\left( \prod_{p \in S} c_p \right)$, where
  \begin{itemize}
  \item $c_0 > 0$ depends only upon measure normalizations,
  \item $S_\infty - \mathfrak{q}$ is the set of archimedean completions of $F$, excluding $\mathfrak{q}$,
  \item for $v \in S_\infty - \mathfrak{q}$ or for $v = p \in S$,
    \begin{equation*}
      c_v =
      \int_{g \in \mathbf{G}(F_v)}
      \frac{\langle g \varphi_1, \varphi_1 \rangle
        \overline{\langle g \varphi_2, \varphi_2 \rangle}
        \langle g \Psi, \Psi  \rangle}{ \langle \varphi_1, \varphi_1 \rangle \langle \varphi_2, \varphi_2 \rangle \langle \Psi, \Psi \rangle}
      \, d g.
    \end{equation*}
    (More precisely, recall from \S\ref{sec:intro-setup} that the $\varphi_j$ and $\Psi$ identify with functions on $\mathbf{G}(F) \backslash \mathbf{G}(\mathbb{A}) / J$; we may thus form their right translates by $g$ on $\mathbf{G}(F) \backslash \mathbf{G}(\mathbb{A})$ and then the indicated inner products.)
  \end{itemize}
  We need to check that this product is $\asymp 1$.  It suffices to show the same for each $c_v$.  We claim more precisely that if the Haar measure on $\mathbf{G}(F_v)$ is normalized to be a probability measure, then $c_v = 1$.  When $v$ is archimedean, this follows from our assumption that our automorphic forms are invariant by $J_\infty$, hence by $\mathbf{G}(F_v)$.  When $v = p \in S$, it is enough to verify that the integrand is identically $1$.  To that end, recall that our automorphic forms are assumed invariant by $J_p$.  That group has index $2$ inside $\mathbf{G}(F_p)$, with the nontrivial coset represented by any group element $g$ that defines the Hecke operator $T_p$.  For that element, the normalized matrix coefficients $\langle g \varphi_j, \varphi_j \rangle / \langle \varphi_j, \varphi_j \rangle$ are both equal to the sign $\pm 1$ of $T_p$ acting on $\pi$.  By our assumption that $T_p$ acts trivially on $\Psi$, the integrand is thus $(\pm 1)^2 \cdot 1 = 1$, as required.
\end{remark}

\begin{proof}[Proof of Lemma \ref{lem:trivial-vanishing}]
  Since the distributions $\mu_\pi$ are invariant by the involutory Hecke operators $T_p$ ($p \in S$), the conclusion is clear if some such operator acts nontrivially on $\sigma$, so suppose otherwise that each such operator acts trivially.  The global root number of $\sigma$ is then the same as the local root number at the distinguished real place $\mathfrak{q}$, which, by hypothesis, is $-1$.  Therefore $L(\sigma, \tfrac{1}{2} ) = 0$.  Since $L(\pi \otimes \overline{\pi } \otimes \sigma, \tfrac{1}{2} ) = L(\ad \pi \otimes \sigma, \tfrac{1}{2} ) L(\sigma, \tfrac{1}{2} )$, we have also $\mathcal{L}(\pi,\sigma) = 0$.  The conclusion follows now from \eqref{eq:branching-coeff-defn}.
\end{proof}

\begin{lemma}\label{lem:WW}
  Let $\sigma \in A_0$ be fixed and even.  Let $\pi$ be an $\h$-dependent element of $A_0$ with $\lambda_\pi < 0$ and $|\h^2 \lambda_\pi| \asymp 1$.
  \begin{enumerate}[(i)]
  \item\label{enumerate:cq21hkqsnp} Let $\Psi \in \sigma$ be a fixed eigenfunction.  Then
    \begin{equation}\label{eq:WW-upper-bound}
      |\mu_\pi(\Psi)|^2
      \ll \h \mathcal{L}(\pi,\sigma).
    \end{equation}
  \item \label{item:lower-bound-WW} There is a fixed eigenfunction $\Psi \in \sigma$ so that
    \begin{equation}\label{eq:WW-lower-bound}
      |\mu_\pi(\Psi)|^2 \gg \h \mathcal{L}(\pi,\sigma).
    \end{equation}
  \item We have
    \begin{equation}\label{eq:convexity-bound}
      \h \mathcal{L}(\pi,\sigma)
      \ll 1.
    \end{equation}
  \end{enumerate}
\end{lemma}
\begin{proof}
  The final estimate \eqref{eq:convexity-bound} follows from \eqref{eq:WW-lower-bound} and the trivial bound $\mu_\pi(\Psi) \ll 1$, so our main task is to verify the first two estimates \eqref{eq:WW-upper-bound} and \eqref{eq:WW-lower-bound}.  Recalling  our convention that ``eigenfunctions'' are $K$-finite, we may assume that $\Psi$ is a unit $K^1$-eigenvector of some fixed weight $q \in \mathbb{Z}$.  Then $\mu_\pi(\Psi) = \langle \varphi_0 \Psi, \varphi_q \rangle$, so our task is to show that the integral
  \begin{equation*}
    I(q) := \int_{g \in G} \langle g \varphi_0, \varphi_0 \rangle \overline{\langle g \varphi_q, \varphi_q \rangle} \langle g \Psi, \Psi \rangle \, d g
  \end{equation*}
  satisfies $I(q) \ll \h$ for each $q$ and $I(q) \gg \h$ for some $q$.

  For $\sigma$, there are two cases:
  \begin{enumerate}[(a)]
  \item\label{enumerate:cq21hkk9br} $\sigma$ is a principal series representation $\pi(t,\eps)$.  Our assumption that $\sigma$ is even then implies that $\eps = 1$.
  \item\label{enumerate:cq21hklbfu} $\sigma$ is a discrete series representation $\pi(k)$.
  \end{enumerate}

  Let $\mathcal{L}_\infty$ denote the ratio of archimedean $L$-factors like on the right hand side of \eqref{eq:defn-cal-L}.  In case \eqref{enumerate:cq21hkk9br}, by comparing the formulas of Watson \cite{watson-2008} and Ichino \cite{MR2449948}, we obtain $I(0) \asymp \mathcal{L}_\infty$.  In case \eqref{enumerate:cq21hklbfu}, Woodbury \cite[Appendix, Thm 3]{2013arXiv1303.6972S} computes $I(k)$ explicitly; dropping the factors in his formula that depend upon the fixed quantities $k$ and $\sigma$ (``$t_3$'' in his notation) gives $I(k) \asymp \mathcal{L}_\infty$.  In either case, a standard application of Stirling's formula as in \cite[\S4.2.1]{watson-2008} gives $\mathcal{L}_\infty \asymp \h$.  In particular, we obtain the required lower bound for $I(q)$ for some $q$.

  It remains to establish the upper bound $I(q) \ll \h$ for each fixed $q$.  This estimate is established case-by-case in the recent paper \cite{MR4879363}.  Alternatively, see \cite{PDN-QUE-implies-subconvexity} for a short uniform proof.
\end{proof}
\begin{remark}
  The ``lower bound'' part of the above proof relies on explicit formulas.  Since estimates would suffice for our purpose, one could ask for a softer treatment.  To that end, we sketch an alternative proof.  Using \eqref{eq:branching-coeff-defn}, we may write $|\mu_\pi(\Psi)|^2 = \mathcal{L}(\pi,\sigma) |\mu_\pi^{\loc}(\Psi)|^2$, say.  One can show by arguments as in \S\ref{sec:orge0a88b2} and \cite[\S6.3]{nelson-padic-que} that the leading order asymptotics as $\h \rightarrow 0$ of $|\mu_\pi^{\loc}(\Psi)|^2$ are given by a constant multiple of $\h \int_{s \in H} \langle s \Psi, \Psi \rangle \, d s$.  As in \cite[\S3.3.1]{michel-2009}, we may write $\int_{s \in H} \langle s \Psi, \Psi \rangle \, d s \asymp |\ell(\Psi)|^2$, where $\ell$ is described in the Kirillov model $\Psi \mapsto W_{\Psi}$ of $\sigma$ (with respect to some fixed nontrivial character) by the absolutely convergent integral $\ell(\Psi) = \int_{y \in \mathbb{R}^\times} W_\Psi(y) \frac{d y}{|y|}$.  Thus $|\mu_\pi^{\loc}(\Psi)|^2 \ll \h$; moreover, if $\ell(\Psi) \neq 0$, then we can replace ``$\ll$'' with ``$\asymp$''.
\end{remark}

\subsection{Operator calculus}\label{sec:operator-calculus}
In this subsection, we recall some properties of the operator calculus developed in \cite{nelson-venkatesh-1} (and refined in \cite{2020arXiv201202187N, PDNstandard}, but we do not require the refinements).  We denote by $\pi$ an $\h$-dependent unitary representation of $G$ and by $\pi^{\infty}$ its subspace of smooth vectors.

\subsubsection{The basic operator assignment}\label{sec:35ac3e5760}
We fix once and for all a cutoff $\chi \in C_c^\infty(\mathfrak{g})$ with the following properties:
\begin{itemize}
\item The support of $\chi$ is sufficiently small.  The precise meaning of ``sufficiently'' is not important for our purposes, but it suffices to require that $\supp(\chi) \subseteq \mathcal{G}_2 \subseteq \mathcal{G}_1 \subseteq \mathcal{G}$, with $\mathcal{G}$ as in \S\ref{sec:kirillov-formula}, where each $\mathcal{G}_j$ is an even bounded open neighborhood of the origin in $\mathfrak{g}$ such that the exponential map restricted to $\mathcal{G}_j$ defines an isomorphism onto its image and the closure of $\exp(\mathcal{G}_2) \exp(\mathcal{G}_2)$ is contained in $\exp(\mathcal{G}_1)$ (cf.\ \cite[\S2.1--\S2.5]{nelson-venkatesh-1}).
\item $\chi$ is $[0,1]$-valued, $\chi(-x) = \chi(x)$, and $\chi = 1$ in a neighborhood of the origin.
\end{itemize}
For any $\h$-dependent Schwartz function $a \in \mathcal{S}(\mathfrak{g}^\wedge)$, we may define the following objects (see \cite[\S2]{nelson-venkatesh-1} for details):
\begin{itemize}
\item $a^\vee : \mathfrak{g} \rightarrow \mathbb{C}$, the inverse Fourier transform of $a$.
\item $a_{\h} : \mathfrak{g}^\wedge \rightarrow \mathbb{C}$ the $\h$-dependent function given by rescaling: $a_{\h}(\xi) := a(\h \xi)$.
\item $a_{\h}^\vee : \mathfrak{g} \rightarrow \mathbb{C}$, the inverse Fourier transform of the rescaling, thus $a_{\h}^\vee(x) = \h^{-3} a^\vee(x/\h)$.
\item $\chi a_{\h}^\vee \in C_c^\infty(\mathfrak{g})$, the cutoff of $a_{\h}^\vee$.
\item The compactly-supported smooth distribution $\chi(x) a_{\h}^\vee (x) \, d x$ on $\mathfrak{g}$, which is supported near the origin.
\item The pushforward under the exponential map $x \mapsto g = \exp(x)$ of this distribution, which may be written $\widetilde{\Opp}_{\h}(a)(g) \, d g$ for some $\widetilde{\Opp}_{\h}(a) \in C_c^\infty(G)$ supported near the identity; explicitly, $\widetilde{\Opp}_{\h}(a)(\exp(x)) = \jac^{-1}(x) \chi(x) a_{\h}^\vee(x)$.
\item An $\h$-dependent integral operator $\Opp_{\h}(a:\pi)$ on $\pi^\infty$, abbreviated $\Opp_{\h}(a)$ when $\pi$ is clear by context, given by
  \[
    \Opp_{\h}(a:\pi) := \pi( \widetilde{\Opp}_{\h}(a) ) = \int_{x \in \mathfrak{g}} \chi(x) a_{\h}^\vee(x) \pi(\exp(x)) \, d x.
  \]
\end{itemize}

\subsubsection{Adjoints}\label{sec:35ac3e5763}
The operator $\Opp_{\h}(a)$ extends to a bounded operator on $\pi$ with adjoint $\Opp_{\h}(\overline{a})$.  In particular, if $a$ is real-valued, then $\Opp_{\h}(a)$ is self-adjoint and $\Opp_{\h}(a)^2$ is positive-definite.

\subsubsection{Symbol classes}\label{sec:35ac3e5764}
For $\xi$ belonging to any normed space (e.g., $\mathfrak{g}^\wedge$), we set
\begin{equation*}
  \langle \xi \rangle := (1 + |\xi|^2)^{1/2}.
\end{equation*}

For fixed $0 \leq \delta < 1/2$ and $m \in \mathbb{Z}$, we write $S^m_\delta$ (denoted ``$S^m[\h^\delta]$'' in \cite[\S4]{nelson-venkatesh-1}) for the space of $\h$-dependent functions $a : \mathfrak{g}^\wedge \rightarrow \mathbb{C}$ such that for each fixed multi-index $\alpha \in \mathbb{Z}_{\geq 0}^{\dim(\mathfrak{g})}$, the corresponding partial derivative $\partial^{\alpha} a$ enjoys for each $\xi \in \mathfrak{g}^\wedge$ the upper bound
\[
  \partial^\alpha a(\xi) \ll \h^{-\delta |\alpha|} \langle \xi \rangle^{m-|\alpha|}.
\]
(The implied constant is thus allowed to depend upon $\alpha$, but not upon $\h$ or $\xi$.)  We extend the definition to $m = \infty$ or $m = -\infty$ by taking unions or intersections.
For instance, an $\h$-independent Schwartz function defines an element of $S^{-\infty}_\delta$, while a polynomial of fixed degree $m \in \mathbb{Z}_{\geq 0}$ and coefficients $\O(1)$ defines an element of $S^m_\delta$.  Elements of $S^{-\infty}_\delta$ are in particular $\h$-dependent Schwartz functions on $\mathfrak{g}^\wedge$, so the operators $\Opp_{\h}(a) := \Opp_{\h}(a:\pi)$ may be defined as above.

\subsubsection{Smoothing operators}\label{sec:35ac3e5766}
For Lie algebra elements $x_1,\dotsc,x_m$, we write $x_1 \dotsb x_m$ for their product in the universal enveloping algebra and abbreviate $\pi(x_1 \dotsb x_m) := d \pi(x_1) \dotsb d \pi(x_m)$.  We denote by $\Psi^{-\infty} := \Psi^{-\infty}(\pi)$ the space of $\h$-dependent operators $T$ on $\pi^\infty$ with the property that for any fixed collection $x_1,\dotsc,x_m,y_1,\dotsc,y_n \in \mathfrak{g}$, the operator norm of $\pi(x_1 \dotsb x_m) T \pi(y_1 \dotsb y_n)$ is $\O(1)$.  This is easily seen to be equivalent to the definition of \cite[\S3]{nelson-venkatesh-1}.  It is verified in \cite[\S12.3]{nelson-venkatesh-1} (see part (iii) of Theorem 9) that if $\pi$ is irreducible, then
\begin{equation}\label{eq:smoothing-ops-have-controlled-trace-norm}
  T \in \Psi^{-\infty}
  \implies
  \text{ the trace norm of $T$ is $\O(1)$}.
\end{equation}

Given an $\h$-dependent scalar $c$ and vector space $V$ consisting of $\h$-dependent quantities, we denote by $c V$ the vector space of $\h$-dependent quantities of the form $c v$, with $v \in V$.  We write $\h^\infty V$ for the intersection of $\h^\eta V$ taken over all fixed $\eta \in \mathbb{R}$.  In particular, we may define $\h^\infty \Psi^{-\infty}$; we will regard it as the space of ``negligible'' operators on $\pi^\infty$.

\subsubsection{Composition}\label{sec:35ac3e576a}
For $\phi_1, \phi_2 \in C_c^\infty(\mathfrak{g})$ supported near the origin, let $\phi_1 \star \phi_2 \in C_c^\infty(\mathfrak{g})$ denote the function for which the distribution $(\phi_1 \star \phi_2)(x) \, d x$ on $\mathfrak{g}$ is the pullback of the convolution on $G$ of the images under pushforward of the distributions $\phi_1(x) \, d x$ and $\phi_2(x) \, d x$ on $\mathfrak{g}$.  For $a, b \in S^{-\infty}_{\delta}$, it is verified in \cite[\S2.5, \S4.6]{nelson-venkatesh-1} that the (rescaled) \emph{star product} $a \star_{\h} b$, characterized by the identity
\begin{equation*}
  (a \star_{\h} b)_{\h} = (\chi a_{\h}^\vee \star \chi b_{\h}^\vee)^\wedge,
\end{equation*}
defines an element of $S^{-\infty}_\delta$ that enjoys the composition formula
\begin{equation}\label{eq:composition-formula-for-opp}
  \Opp_{\h}(a)
  \Opp_{\h}(b)
  \equiv 
  \Opp_{\h}(a \star_{\h} b)
  \pmod{\h^\infty \Psi^{-\infty}}.
\end{equation}
The failure of \eqref{eq:composition-formula-for-opp} to be an equality is an artefact of the cutoff $\chi$.

\subsubsection{Equivariance}\label{sec:35ac3e576c}
It follows from \cite[\S5.5]{nelson-venkatesh-1} that for $g \in G$ belonging to a fixed compact subset,
\begin{equation}\label{eq:opp-equivariance}
  \Opp_{\h}(g \cdot a)
  \equiv
  \pi(g)
  \Opp_{\h}(a)
  \pi(g)^{-1}
  \pmod{\h^\infty \Psi^{-\infty}},
\end{equation}
where $g \cdot a(\xi) := a(g^{-1} \cdot \xi)$.  The error comes from the failure of the cutoff $\chi$ to be exactly $G$-invariant.  It will be convenient to assume that $\chi$ is exactly $K$-invariant (by averaging a given cutoff, for instance).  We then have
\begin{equation}\label{eq:exact-equivariance-K}
  \Opp_{\h}(g \cdot a)
  =
  \pi(g)
  \Opp_{\h}(a)
  \pi(g)^{-1}
  \text{ for all } g \in K.
\end{equation}
More precisely, $\widetilde{\Opp}_{\h}(g \cdot a)$ is the conjugate by $g$ of $\widetilde{\Opp}_{\h}(a)$.

\subsubsection{Star product extension and asymptotics}\label{sec:35ac3e576d}
It is shown in \cite[\S4.6]{nelson-venkatesh-1} that the star product extends to a compatible family of maps $\star_{\h} : S^m_\delta \times S^n_\delta \rightarrow S^{m+n}_\delta$ enjoying the asymptotic expansion: for fixed $J \in \mathbb{Z}_{\geq 0}$,
\begin{equation}\label{eq:asymp-expn-star-prod}
  a \star_{\h} b
  \equiv \sum_{0 \leq j < J}
  \h^j a \star^j b
  \pmod{\h^{(1-2 \delta ) J} S_{\delta}^{m+n-J}},
\end{equation}
with $\star^j$ a fixed polynomial-coefficient differential operator, of order $j$ in each variable, homogeneous of degree $j$, satisfying the mapping property $\star^j : S^{m}_\delta \times S^{n}_\delta \rightarrow \h^{- 2 \delta j} S^{m+n-j}_\delta$ and given in the simplest case by $a \star^0 b = a b$.

\subsubsection{Extended operator assignment}\label{sec:35ac3e5770}
It is shown in \cite[\S5.6]{nelson-venkatesh-1} that $\Opp_{\h}$ extends to a compatible family of maps
\[
  \Opp_{\h} : S^m_\delta \rightarrow \{\text{$\h$-dependent operators on $\pi^\infty$}\}
\]
for which the composition and equivariance properties \eqref{eq:composition-formula-for-opp}, \eqref{eq:opp-equivariance} and thus \eqref{eq:exact-equivariance-K} remain valid.

\subsubsection{Polynomial symbols}\label{sec:35ac3e5771}
It is verified in \cite[\S5.2]{nelson-venkatesh-1} that if $p \in S^m_\delta$ is an $\h$-dependent polynomial function (corresponding to some $\h$-dependent element of $\Sym(\mathfrak{g}_\mathbb{C})$), then
\begin{equation}\label{eq:Opp-polynomial-symb}
  \Opp_{\h}(p)
  = \pi(\sym(p_{\h})),
\end{equation}
where $\sym$ denotes the symmetrization map from $\Sym(\mathfrak{g}_\mathbb{C})$ to the the universal enveloping algebra of $\mathfrak{g}_\mathbb{C}$ and (as above) $p_{\h}(\xi) = p(\h \xi)$.

\subsubsection{Trace estimates}\label{sec:opp-trace-estimates}
It is shown in \cite[\S12.3]{nelson-venkatesh-1} that if $\pi$ is irreducible and tempered (so that the coadjoint orbit $\mathcal{O}_{\pi}$ as well as its rescaling $\h \mathcal{O}_\pi$ may be defined), then for $a \in S^{-2}_\delta$, the operator $\h \Opp_{\h}(a)$ is trace-class, with trace asymptotics described for each fixed $J \in \mathbb{Z}_{\geq 0}$ by
\begin{equation}\label{eq:kirillov-expanded}
  \trace(\h \Opp_{\h}(a))
  = 
  \sum_{0 \leq j < J}
  \h^j \int_{\h \mathcal{O}_\pi }
  \mathcal{D}_j a
  + \O(\h^{(1-\delta) J}),
\end{equation}
where $\mathcal{D}_j$ is a fixed constant coefficient differential operator of pure degree $j$, with $\mathcal{D}_0 a = a$.  In particular,
\begin{equation}\label{eq:kirillov-expanded-2}
  \trace(\h \Opp_{\h}(a))
  \ll 1.
\end{equation}

\subsubsection{Clean-up}\label{sec:clean-up}
It follows from \cite[\S10.3]{nelson-venkatesh-1} that if $\pi$ is irreducible (so that its infinitesimal character $\lambda_\pi \in \mathbb{R}$ may be defined) and $a \in S^{\infty}_\delta$ has the property that the image under the invariant polynomial $\Lambda : \mathfrak{g}^\wedge \rightarrow \mathbb{R}$ (see \S\ref{sec:lie-algebra}) of the support of $a$ is separated by at least $\h^{1/2-\eps}$ from $\h^2 \lambda_\pi$ for some fixed $\eps > 0$, then
\begin{equation*}
  \Opp_{\h}(a) \in
  \langle \lambda_{\pi} \rangle^{-\infty}
  \h^\infty \Psi^{-\infty},
\end{equation*}
where as usual $\langle \lambda_\pi \rangle := (1 + |\lambda_\pi|^2)^{1/2}$.  In particular, the trace norm of $\Opp_{\h}(a)$ is $\O(\langle \lambda_{\pi} \rangle^{-\infty} \h^\infty)$.

(We note a potential point of notational confusion: the rescaled infinitesimal character that we denote here by $\h^2 \lambda_\pi \in \mathbb{R}$ is written ``$\h \lambda_\pi \in [\mathfrak{g}^\wedge] \cong \mathbb{R}$'' in \cite{nelson-venkatesh-1}; see \cite[\S9]{nelson-venkatesh-1} for details.)

\section{Microlocal lifts}\label{sec:35ac3e579c}

\subsection{Characterizing microlocal lifts via their symmetry}\label{sec:states-approx-microlocal}
Theorem \ref{thm:var-2} applies to the distributions $\Psi \mapsto \mu(\pi(f),\Psi)$ attached to integral operators $\pi(f)$ with $f \in C_c^\infty(G)$, but the construction of the microlocal lift $\mu_\pi$ recorded in \S\ref{sec:constr-mu-pi-ZW} is in terms of differential operators.  We thus encounter the problem of constructing $\mu_\pi$, or at least some asymptotically equivalent distributions, using integral operators.

We begin with some motivational remarks.  Recall the asymptotic notation and terminology set in \S\ref{sec:orgb2d0a08}.  Let $\pi$ be an $\h$-dependent element of $A_0$ with $\lambda_\pi < 0$ and $|\h^2 \lambda_\pi| \asymp 1$.  Set
\begin{equation*}
  v := c \sum_{q \in \mathbb{Z} : |q| \leq \h^{-1/2}} \varphi_q \in \pi,
\end{equation*}
where $c > 0$ is chosen so that $v$ is a unit vector, and $T := v \otimes \overline{v} \in \pi \otimes \overline{\pi}$.  It follows from calculations of Wolpert \cite[\S5]{MR1838659} (see \cite[\S3]{MR1859345} for a concise account) that for fixed eigenfunctions $\Psi$,
\begin{equation}\label{eq:mu-T-Psi-vs-mu-pi-Psi-simple}
  \mu(T,\Psi)
  =
  \langle v \Psi, v \rangle
  =
  \mu_\pi(\Psi)
  + \O(\h^{1/2}).
\end{equation}
Set $t := \sqrt{- \h^2 \Omega_\pi } = \sqrt{- \h^2 \lambda_\pi } + \O(\h) \asymp 1$.  We verify readily that
\begin{equation*}
  \pi(\h X) v = t v + \O(\h^{1/2}),
\end{equation*}
\begin{equation*}
  \pi(\h Y) v = t v + \O(\h^{1/2})
\end{equation*}
and
\begin{equation*}
  \pi(\h W) v = \O(\h^{1/2});
\end{equation*}
equivalently, for fixed $Z \in \mathfrak{g}$, we have
\begin{equation*}
  \pi(\h Z) v = Z(\xi(t)) v + \O(\h^{1/2}),
\end{equation*}
with $\xi(t) \in \mathfrak{g}^\wedge$ as in \eqref{eqn:defn-xi-of-t}; in other words, $v$ is an approximate eigenvector under the first-order differential operators on $\pi^\infty$ defined by Lie algebra elements, with eigenvalue described by $\xi(t)$.  We will verify below that some variants of these observations concerning $v$, phrased in terms of $T$, give sufficient conditions for (more precise forms of) the estimate \eqref{eq:mu-T-Psi-vs-mu-pi-Psi-simple} to hold.  Turning to details:
\begin{definition*}
  Let $\pi$ be an $\h$-dependent irreducible unitary representation of $G$.  Let $T$ be an $\h$-dependent positive-definite trace class operator on $\pi$ such that $\trace(T) \ll 1$.  Let $\tau$ be an $\h$-dependent element of $\mathfrak{g}^\wedge$ with $|\h \tau| \ll 1$.  Let $0 < \delta \leq 1/2$ be fixed.  We say that $T$ is \emph{$\delta$-localized at $\tau$} if for each fixed $n \in \mathbb{Z}_{\geq 0}$ and each $\h$-dependent polynomial function $p : \mathfrak{g}_\mathbb{C}^* \rightarrow \mathbb{C}$ of degree $\O(1)$ and coefficients $\O(1)$ that vanishes to order at least $n$ at $\h \tau$, we have
  \begin{equation}
    \trace(\Opp_{\h}(p) T) \ll \h^{n \delta},
  \end{equation}
  where $\Opp_{\h}(p) := \Opp_{\h}(p:\pi)$ is as given by \eqref{eq:Opp-polynomial-symb}.
\end{definition*}
One can verify that the $T$ considered above is $1/2$-localized at $\xi(t)$; we will not need this fact, so we omit the proof.

We may construct integral operators satisfying the above definition:
\begin{lemma}[Integral operators attached to localized symbols are localized]
  Let $0 < \delta < 1/2$ be fixed.  Let $\tau \in \mathfrak{g}^\wedge$ with $|\h \tau| \asymp 1$.  Let $a \in \h^{-\delta} S^{-\infty}_{\delta}$ be real-valued.  Let $\pi$ be an $\h$-dependent tempered irreducible unitary representation of $G$.  Set $T := \h \Opp_{\h}(a:\pi)^2$.
  
  Suppose that every element of $\supp(a) \cap \h \mathcal{O}_\pi$ is of the form $\h \tau + \O(\h^\delta)$.  Then $T$ is $\delta$-localized at $\tau$, and
  \begin{equation}\label{eq:trace-estimate-in-localized-case}
    \trace(T)
    = \int_{\h \mathcal{O}_\pi} a^2 + \O(\h^{1-\delta})
    = \O(1).
  \end{equation}
\end{lemma}
\begin{proof}
  We have $a^2(\xi) \ll \h^{-2 \delta}$.  By the lemma of \S\ref{sec:coadjoint-orbits} and the hypotheses concerning $|\h \tau|$ and the support of $a$, the set $\h \mathcal{O}_\pi \cap \supp(a)$ has symplectic volume $\O(\h^{2 \delta})$.  The required trace estimate \eqref{eq:trace-estimate-in-localized-case} thus follows from \S\ref{sec:opp-trace-estimates}.  In particular, the operator $T$ is positive-definite with $\trace(T) \ll 1$.

  To verify the localization property, fix $n \in \mathbb{Z}_{\geq 0}$ and let $p$, as above, be an $\h$-dependent polynomial of degree $\O(1)$ and coefficients $\O(1)$ that vanishes to order $\geq n$ at $\h \tau$.  We must check then that $\trace(\Opp_{\h}(p) T) \ll \h^{n \delta}$.
  
  We pause to observe that for each $q \in S^{\infty}_\delta$ and each fixed $J \in \mathbb{Z}_{\geq 0}$,
  \begin{equation}\label{eq:trace-opp-q-T}
    \trace(\Opp_{\h}(q) T )
    =
    \sum_{0 \leq j_1, j_2 < J}
    \h^{j_1 + j_2}
    \int_{\h \mathcal{O}_\pi}
    q \star^{j_1} (a \star^{j_2} a)
    + \O(\h^{(1 - 2 \delta) J}).
  \end{equation}
  This estimate follows from the composition formula \eqref{eq:composition-formula-for-opp}, the star product asymptotics \eqref{eq:asymp-expn-star-prod} and the trace estimate \eqref{eq:kirillov-expanded}, using \eqref{eq:smoothing-ops-have-controlled-trace-norm} and \eqref{eq:kirillov-expanded-2} to clean up the remainders.  Since $\h \mathcal{O}_\pi \cap \supp(a)$ has symplectic volume $\O(\h^{2 \delta})$, we have also for fixed $j_1, j_2 \geq 0$ that
  \begin{equation*}
    \int_{\h \mathcal{O}_\pi}
    q \star^{j_1} (a \star^{j_2} a)
    \ll
    \h^{2 \delta}
    \|q \star^{j_1} (a \star^{j_2} a)\|_{L^\infty(\h \mathcal{O}_\pi)}.
  \end{equation*}

  Returning to the proof of the lemma, choose a ball $B_1$ with origin $\h \tau$ and radius $\asymp \h^\delta$ so that $\supp(a) \cap \h \mathcal{O}_\pi \subseteq B_1$.  Let $B_2$ denote the ball with the same origin as $B_1$ but twice the radius.  Choose $\phi \in S^{-\infty}_{\delta}$ taking the value $1$ on $B_1$ and the value $0$ on the complement of $B_2$.  We may then decompose $p = \phi p + (1 - \phi) p$.  We apply the above estimates with $q =\phi p$ and $q = (1 - \phi) p$:
  \begin{itemize}
  \item Our assumptions on $p$ imply that $\phi p \in \h^{n \delta} S^{-\infty}_\delta$.  By \eqref{eq:trace-opp-q-T} and the mapping properties of $\star^j$, the symbol $\h^{j_1 + j_2} \phi p \star^{j_1} (a \star^{j_2} a)$ belongs to $\h^{-2 \delta + n \delta} S^{-\infty}_\delta$ and thus has $L^\infty$-norm $\O(\h^{- 2 \delta + n \delta})$.  It follows that $\trace(\Opp_{\h}(\phi p) T) = \O(\h^{n \delta} + \h^{J'})$.
  \item By construction, $\h \mathcal{O}_\pi \cap \supp(1-\phi) \cap \supp(a) = \emptyset$, so $(1 - \phi) p \star^{j_1} (a \star^{j_2} a)$ vanishes identically on $\h \mathcal{O}_\pi$, and thus $\trace(\Opp_{\h}((1-\phi) p T) = \O(\h^{J'})$.
  \end{itemize}
  We conclude by combining these estimates and taking $J$ large enough.
\end{proof}

We verify next the promised relationship between the above definition and $\mu_\pi$.
\begin{proposition}[Some localized operators define microlocal lifts]\label{thm:characterize-mu-pi}
  Fix a mean-zero even eigenfunction $\Psi \in \sigma \in A_0$.  Let $\pi$ be an $\h$-dependent element of $A_0$ such that $\lambda_\pi < 0$ and $|\h^2 \lambda_\pi| \asymp 1$.  Abbreviate $\mathcal{L} := \mathcal{L}(\pi,\sigma)$.  Let $T$ be an $\h$-dependent positive-definite trace class operator on $\pi$ with $\trace(T) \ll 1$.  Set $\tau := \xi(\sqrt{-\Omega_\pi}) \in \mathfrak{g}^\wedge$, so that $|\h \tau| \asymp 1$.  Fix $0 < \delta \leq 1/2$.

  Suppose that $T$ is $\delta$-localized at $\tau$.  Then
  \begin{equation}\label{eq:mu-T-Psi-approximates-etc}
    \mu(T,\Psi)
    =
    \trace(T)
    \mu_\pi(\Psi)
    + \O (\h^\delta \sqrt{\h \mathcal{L} } + \h^\infty ).
  \end{equation}
\end{proposition}
The proof is given in \S\ref{sec:calc-with-rais}.

\begin{remark}
  The estimate \eqref{eq:mu-T-Psi-approximates-etc} implies in particular that $\mu(T,\Psi) = \trace(T) \mu_\pi(\Psi) + \O(\h^\delta)$, but this weaker estimate is inadequate for our applications, in which we exploit crucially that $\mathcal{L}$ is ``bounded on average'' (see \S\ref{sec:defn-f}).
\end{remark}

\begin{remark}
  Although Proposition \ref{thm:characterize-mu-pi} is formulated in terms of $L$-values, it does not fundamentally exploit the arithmeticity of $\Gamma \backslash G$.  What matters are the properties of $\mathcal{L}$ enunciated in \S\ref{sec:branch-coeff}, which make sense for general finite volume quotients (see \cite[\S1.4]{nelson-venkatesh-1}, \cite{MR2726097}).
\end{remark}

\begin{remark}
  Proposition \ref{thm:characterize-mu-pi} may be used to give a proof of the asymptotic $H$-invariance of the measures $\mu_\pi$, roughly in the spirit of \cite[\S26.5]{nelson-venkatesh-1}; the relevant observations are that
  \begin{itemize}
  \item if $T$ is $\delta$-localized at $\tau$ and $g \in G$ is fixed, then $\pi(g) T \pi(g)^{-1}$ is $\delta$-localized at $g \cdot \tau$, and
  \item $H$ centralizes the elements $\xi(t)$.
  \end{itemize}
\end{remark}

\subsection{Calculations with raising and lowering operators}\label{sec:calc-with-rais}
Here we record the proof of Proposition \ref{thm:characterize-mu-pi} of \S\ref{sec:states-approx-microlocal}.  The proof is a bit tedious, but not difficult, and unrelated to the main novelties of this work.  It is basically a quantification of the arguments used to prove \eqref{eq:mu-T-Psi-vs-mu-pi-Psi-simple}.  (Indeed, it is instructive to note that \eqref{eq:mu-T-Psi-approximates-etc} recovers \eqref{eq:mu-T-Psi-vs-mu-pi-Psi-simple}.)  For these reasons, the reader might wish to skim or skip this section on a first reading.

\subsubsection{Setup and notation}

We recall that our task is to verify, under certain assumptions, the estimate \eqref{eq:mu-T-Psi-approximates-etc}.
We have $\pi \cong \pi(t,\eps)$ with $t > 0$.  We realize $\pi(t,\eps)$ as $L^2(\mathbb{Z})$ as in \S\ref{sec:constr-irred-unit}.  There is then a unique equivariant (isometric) isomorphism $j_\pi : L^2(\mathbb{Z}) \rightarrow \pi$ that maps the basis element $e_0$ to $\varphi_\pi$.  Thus $\varphi_q$, as in the construction of $\mu_\pi$, is equal to $b(q) j_\pi(e_q)$, where $b(q)$ is defined recursively by
\[
  b(q+1) = b(q) \frac{i}{s+q} \sqrt{q(q+1) - \Omega_\pi}.
\]
Since $t \in \mathbb{R}$, we have $|s+q|^2 = q(q+1) - \Omega_\pi$, and so $|b(q)| = 1$ for all $q$.  Moreover, since $t > 0$, we have for fixed $q$ that $b(q) = 1 + \O(1/t)$.  Thus the vectors $\varphi_q$ are asymptotically quite close to the $j_\pi(e_q)$.

By a limiting argument, it will suffice to consider the case that $T$ is a finite rank operator $T = \sum_i j_\pi(v_i) \otimes \overline{j_\pi(v_i)}$ attached to some finite orthogonal subset $\{v_i\}$ of $L^2(\mathbb{Z})$.  For $q,\ell \in \mathbb{Z}$, we set $T(q,\ell) := \sum_i v_i(q) \overline{v_i(q+\ell)}$, so that
\begin{equation*}
  T = \sum_{q,\ell} T(q,\ell) j_\pi(e_q) \otimes \overline{j_\pi(e_{q+\ell})},
\end{equation*}
and
\begin{equation*}
  \Psi(q,\ell) := \langle j_\pi(e_{q}) \Psi, j_\pi(e_{q+\ell}) \rangle,
\end{equation*}
so that
\[
  \mu_\pi(\Psi) = \sum_{\ell} b(\ell) \Psi(0,\ell)
\]
and
\[
  \mu(T,\Psi) = \sum_{q,\ell} T(q,\ell) \Psi(q,\ell)
\]
and
\[
  \trace(T) = \sum_{i,q} |v_i(q)|^2 = \sum_q T(q,0).
\]
By Cauchy--Schwarz and the assumed trace estimate for $T$,
\begin{equation*}
  \sum_q |T(q,\ell)|
  \leq
  \sum_q \sum_i |v_i(q) \overline{v_i(q + \ell)}|
  \leq \sum_{q,i} |v_i(q)|^2
  = \trace(T)
  \ll 1.
\end{equation*}

\subsubsection{Consequences of localization}

Set $r := \sqrt{- \Omega_\pi}$, so that $\tau = \ell(r)$.  We temporarily abbreviate $X,Y,W := \pi(X), \pi(Y), \pi(W)$.  We will use the following consequences of our assumption that $T$ is $\delta$-localized at $\tau$:
\begin{equation}\label{eqn:W-cond-for-T}
  \trace((\h W)^n T) \ll \h^{n \delta}
  \text{ for each fixed $n \in \mathbb{Z}_{\geq 0}$},
\end{equation}
\begin{equation}\label{eqn:X-cond-for-T}
  \trace((\h X - \h r) T) \ll \h^\delta.
\end{equation}
Indeed, we have $X(\h \tau) = Y(\h \tau) = \h r$ and $W(\h \tau) = 0$, so the polynomial $p = W^n$ vanishes to order $n$ at $\h \tau$ and satisfies $\sym(p_{\h}) = (\h W)^n$, while the polynomial $p = X - \h r$ vanishes to order $1$ at $\h \tau$ and satisfies $\sym(p_{\h}) = \h X - \h r$.

By \eqref{eqn:W-cond-for-T} with $n = 2$, we have
\[
  \sum_{i,q} |\h q|^2 |v_i(q)|^2 = \trace((\h W)^2 T) \ll \h^{2 \delta}.
\]
Using Cauchy--Schwarz as above, it follows that for fixed $\ell$,
\begin{equation}\label{eqn:T-W-decay-bulk}
  \sum_q
  |T(q,\ell)| \cdot |h q|
  \ll
  \h^{\delta}.
\end{equation}
We now fix $0 < \delta ' < \delta$, and argue using \eqref{eqn:W-cond-for-T} for arbitrary fixed $n$ that
\begin{equation}\label{eqn:T-W-decay}
  \sum_{q :  |\h q| \geq \h^{\delta '}}
  |T(q,\ell)| \ll \h^\infty.
\end{equation}

We now investigate the consequences of \eqref{eqn:X-cond-for-T}.  We have
\[
  \h X T = \sum_{q,\ell} T(q,\ell) \h \sqrt{r^2 + q(q+1)} e_{q + 1} \otimes \overline{e_{q + \ell}},
\]
thus
\begin{equation}\label{eqn:formula-trace-h-X-T}
  \trace(\h X T) = \sum_{q} \h \sqrt{r^2 + q(q+1)} T(q,1).
\end{equation}
Our assumptions on $\pi$ imply that $\h r \asymp 1$.  We estimate the latter sum in the range $|\h q | \geq \h^{\delta '}$ using \eqref{eqn:T-W-decay} and in the range $|\h q | < \h ^{\delta '}$ using the Taylor expansion
\[
  \h \sqrt{r^2 + q(q+1)} = \h r + \O(|\h q|).
\]
The contribution to \eqref{eqn:formula-trace-h-X-T} of the remainder in this expansion is treated using \eqref{eqn:T-W-decay-bulk}.  We extend the sum to all $q$ using \eqref{eqn:T-W-decay} once again.  We obtain in this way that
\begin{equation*}
  \trace(\h X T) = \h {r} \sum_{q} T(q,1)
  + \O(\h^\delta).
\end{equation*}
Since $T(q,1) = \sum_i v_i(q) \overline{v_i(q+1)}$, the estimate \eqref{eqn:X-cond-for-T} thus translates to
\begin{equation}\label{eqn:X-cond-for-T-2}
  \sum_{q,i}
  |v_i(q)|^2 = \sum_{q,i} v_i(q) \overline{v_i(q+1)}
  + \O(\h^\delta).
\end{equation}
By iterating this estimate, one may deduce more generally that for each fixed $\ell$,
\begin{equation}\label{eqn:T-X-cons-0}
  \sum_{q,i}
  |v_i(q)|^2 = \sum_{q,i} v_i(q) \overline{v_i(q+\ell)}
  + \O(\h^\delta),
\end{equation}
which translates to
\begin{equation}\label{eqn:T-X-cons}
  \sum_q T(q,\ell) = \trace(T) + \O(\h^\delta).
\end{equation}
We record details concerning this deduction at the end of \S\ref{sec:calc-with-rais}.

\subsubsection{Statement of crucial estimates}

Recall that $\Psi(q,\ell) = 0$ unless $|\ell| \leq C$ for some fixed $C$.  We have the trivial bound
\begin{equation}\label{eqn:trivial-for-Psi-pi}
  |\Psi(q,\ell)| \leq \|\Psi \|_{L^\infty(\Gamma \backslash G)}
  \ll 1
\end{equation}
for all $q, \ell$.  Suppose now that $|\h q| \leq \h^{\delta '}$.  We claim then that
\begin{equation}\label{eqn:Psi-pi-triv}
  \Psi(q,\ell) \ll \sqrt{\h \mathcal{L} } + \h^{\infty}
\end{equation}
and
\begin{equation}\label{eqn:Psi-pi-lipschitz}
  \Psi(q,\ell) =
  \Psi(0,\ell)
  + \O(|\h q| \sqrt{\h \mathcal{L}}  + \h^{\infty}).
\end{equation}

\subsubsection{Proof of \eqref{eqn:Psi-pi-triv} and \eqref{eqn:Psi-pi-lipschitz}}

We will prove these when $q \geq 0$; an analogous argument applies to negative $q$.

For $j \in \mathbb{Z}_{\geq 0}$ and $q,\ell \in \mathbb{Z}$, let $\Psi^j(q,\ell)$ be defined like $\Psi(q,\ell)$, but with $\Psi$ replaced with $X^j \Psi$.  We will work in what follows with fixed values of $j$, so that $\|X^j \Psi \| \ll 1$.

By \eqref{eq:WW-upper-bound}, we have
\begin{equation}\label{eqn:woodbury-stirling}
  \Psi^j(0,\ell)
  \ll \sqrt{\h \mathcal{L} }.
\end{equation}
We also have the trivial bound
\begin{equation}\label{eqn:trivial-Psi-j}
  \Psi^j(q,\ell)
  \ll 1,
\end{equation}
as in \eqref{eqn:trivial-for-Psi-pi}.

We now argue recursively using the following instance of ``partial integration'': the integral over $\Gamma \backslash G$ of $X (j_\pi(e_q) \overline{j_\pi(e_{q+\ell})} X^j \Psi)$ vanishes.  Expanding this out, we obtain with $f(q) := \h \sqrt{r^2 + q(q+1)}$ that
\[
  f(q+\ell) \Psi^j(q,\ell) = f(q) \Psi^j(q+1,\ell) + \h \Psi^{j+1}(q,\ell+1).
\]
For $q$ in the indicated range and $\ell \ll 1$, we have $f(q) \asymp 1$ and $f(q+\ell) = f(q) + \O(\h)$.  Hence for such $q$ and $\ell$,
\begin{equation}\label{eqn:Psi-pi-j-main-recursive-bound}
  \Psi^j(q+1,\ell)
  -
  \Psi^j(q,\ell)
  \ll
  \h (
  \left\lvert \Psi^j(q,\ell) \right\rvert
  +
  \left\lvert \Psi^j(q+1,\ell) \right\rvert
  +
  \left\lvert \Psi^{j+1}(q,\ell+1) \right\rvert
  ).
\end{equation}

Fix $J \in \mathbb{Z}_{\geq 0}$ and then $C \in \mathbb{R}_{\geq 1}$ sufficiently large, and set
\[
  \beta_j(q) := C (1 + C \h)^{-q} \sup_{\ell} |\Psi^j(q,\ell)|.
\]
We consider $q \geq 0$ with $|\h q| \leq \h^{\delta '}$.  Having chosen $C$ large enough, the estimate \eqref{eqn:Psi-pi-j-main-recursive-bound} implies
\begin{equation}\label{eqn:recursion-for-beta-j}
  \beta_j(q + 1) \leq \beta_j(q)
  + C \h \beta_{j+1}(q)
  \quad (0 \leq j < J).
\end{equation}
Similarly, by \eqref{eqn:trivial-Psi-j},
\begin{equation*}
  \beta_J(q) \leq 1.
\end{equation*}
Thus the sequence of $(J+1)$-dimensional row vectors
\[\beta(q) := (\beta_0(q), \beta_1(q), \dotsc, \beta_{J-1}(q), 1)\]
satisfy
\[\beta(q) \leq \beta(0) M^q,
  \quad M := \begin{pmatrix}
    1    &        &        &      &   \\
    C \h & 1      &        &          \\
    & \dotsb & \dotsb &      &   \\
    &        & \dotsb & 1    &   \\
    & & & C \h & 1
  \end{pmatrix}.
\]
We also have, by \eqref{eqn:woodbury-stirling} and the estimate $(1 + C \h)^q = 1 + \O(\h^{\delta '})$, the initial bound
\begin{equation*}
  \beta_j(0) \ll \sqrt{\h \mathcal{L}}
  \quad
  (0 \leq j < J).
\end{equation*}
These lead to
\begin{equation}
  \beta_j(q) \ll \sqrt{\h \mathcal{L} }
  + \h^{(J-j) \delta}.
\end{equation}
Taking $j = 0$ and recalling that $J$ was arbitrary, we obtain \eqref{eqn:Psi-pi-triv}, and also its analogue for $\Psi^1$; inserting the latter into the $j=0$ case of \eqref{eqn:Psi-pi-j-main-recursive-bound} then gives \eqref{eqn:Psi-pi-lipschitz}.

\subsubsection{Conclusion}

We now combine the above estimates to conclude.  Expanding the definitions, we have
\[
  \mu(T,\Psi) = \trace(T) \mu_{\pi}(\Psi) + S_1 + S_2 + S_3
\]
where
\begin{equation}
  S_1 :=
  \sum_{\ell}
  \left(\sum_q T(q,\ell) - \trace(T)\right) \Psi(0,\ell),
\end{equation}
\begin{equation}
  S_2 :=
  \sum_{q,\ell}
  T(q,\ell) (\Psi(q,\ell) - \Psi(0,\ell)),
\end{equation}
\begin{equation}
  S_3 :=
  \trace(T)
  \sum_{\ell} (1 - b(\ell)) \Psi(0,\ell).
\end{equation}
Using \eqref{eqn:Psi-pi-triv} and \eqref{eqn:T-X-cons}, we see that $S_1 \ll \h^\delta \sqrt{\h \mathcal{L} } + \h^\infty$.  To bound $S_2$, we estimate the contribution from $|\h q| \geq \h^{\delta '}$ via \eqref{eqn:T-W-decay} and \eqref{eqn:Psi-pi-triv}.  We then estimate the remaining contribution via \eqref{eqn:Psi-pi-lipschitz}.  We obtain
\begin{equation}
  S_2 \ll
  \sqrt{\h \mathcal{L} }
  \sum_{\ell : |\ell| \leq C}
  \sum_{q}
  |T(q,\ell)|
  \cdot
  |\h q|
  +
  \h^{\infty}
  \ll
  \h^{\delta} \sqrt{\h \mathcal{L} }
  + \h^\infty.
\end{equation}
For $S_3$, we use that $\Psi(0,\ell) \ll \sqrt{h \mathcal{L}}$ and that $\Psi(0,\ell) \neq 0$ only if $|\ell| = \O(1)$, in which case $b(\ell) = 1 + \O(\h)$; thus $S_3 \ll \h \sqrt{\h \mathcal{L}}$.  This completes the proof of \eqref{eq:mu-T-Psi-approximates-etc}.

\subsubsection{Proof of \eqref{eqn:T-X-cons-0}}

We finally record the promised details concerning the deduction of \eqref{eqn:T-X-cons-0}.  Let $(V, \langle , \rangle)$ be a complex inner product space, with unitary group $\U(V)$.
\begin{lemma*}
  For $a, b \in \U(V)$ and $v \in V$, we have
  \begin{equation*}
    \lvert \langle a b v - v, v \rangle \rvert^{1/2} \leq \lvert \langle a v - v, v \rangle \rvert^{1/2} + \lvert \langle b v - v, v \rangle \rvert ^{1/2}.
  \end{equation*}
\end{lemma*}
\begin{proof}
  Introduce the abbreviations
  \begin{equation*}
    \eps_a := \lvert \langle a v - v, v \rangle \rvert^{1/2},
    \quad 
    \eps_b := \lvert \langle b v - v, v \rangle \rvert^{1/2}.
  \end{equation*}
  We observe first, by expanding the square and invoking the assumed unitarity, that
  \begin{equation}\label{eqn:expand-square-avminusv}
    \lVert a v - v \rVert^2 = 2 \Re \langle v - a v, v \rangle \leq 2 \eps_a^2,
  \end{equation}
  and similarly for $b$.  Consider now the inner product
  \begin{equation*}
    I := \langle a (b v - v), a v - v \rangle.
  \end{equation*}
  On the one hand, by Cauchy--Schwarz and \eqref{eqn:expand-square-avminusv}, we have $|I| \leq 2 \eps_a \eps_b$.
  On the other hand, by expanding out and invoking the unitarity of $a$ and $b$, we see that
  \begin{equation*}
    \langle a b v - v, v \rangle =
    \langle a v - v , v \rangle
    +
    \langle b v - v, v \rangle
    - I,
  \end{equation*}
  and thus
  \begin{equation*}
    |\langle a b v - v, v \rangle| \leq \eps_a^2 + \eps_b^2 + 2 \eps_a \eps_b = (\eps_a + \eps_b)^2.
  \end{equation*}
  The required inequality follows.
\end{proof}
\begin{corollary*}
  For $a \in \U(V)$, $v \in V$ and $n \in \mathbb{Z}$, we have
  \begin{equation}\label{eq:lvertlangle-an-v}
    \lvert\langle a^n v - v, v \rangle \rvert \leq n^2  \lvert \langle a v - v, v \rangle \rvert.
  \end{equation}
\end{corollary*}
\begin{proof}
  It suffices to consider the case $n \geq 0$; otherwise, replace $(n,a)$ with $(-n, a^{-1})$ and use that
  \begin{equation*}
    \langle a v - v, v \rangle = \overline{\langle a^{-1} v - v, v \rangle}.
  \end{equation*}
  The required inequality is clear when $n=0,1$.  For $n \geq 2$, we induct; the induction step is given by the lemma with $b := a^{n-1}$.
\end{proof}
\begin{remark*}
  By taking $V$ to be $\mathbb{C}^2$ with the standard inner product and $a$ to be rotation by some small angle, we may see that \eqref{eq:lvertlangle-an-v} is sharp, i.e., the factor $n^2$ cannot be replaced by anything smaller.
\end{remark*}

We now explain how the corollary implies \eqref{eqn:T-X-cons-0}.  Let $V$ be the inner product space consisting of collections of complex numbers $v = \{v_i(q)\}_{i,q}$, where $i$ ranges over some index set $I$ and $q$ ranges over $\mathbb{Z}$, with the inner product given by
\begin{equation*}
\langle u, v \rangle = \sum_{i,q} u_i(q) \overline{v_i(q)}.
\end{equation*}
 Let $a$ denote the unitary operator on $V$ given by
 \begin{equation*}
(a v)_i(q) = v_i(q+1).
\end{equation*}
Our hypothesis \eqref{eqn:X-cond-for-T-2} then reads $\langle v, v \rangle = \langle v, a v \rangle + \O(\h^\delta)$, while the desired conclusion \eqref{eqn:T-X-cons-0} reads $\langle v, a^{\ell} v \rangle = \langle v, v \rangle + \O(\h^\delta)$ for fixed $\ell$.  Since $\ell^2 = \O(1)$, the required implication follows immediately from \eqref{eq:lvertlangle-an-v}.

\subsection{Constructing microlocal lifts via integral operators}
\label{sec:org8bcf1bd}
Recall the coordinates $\xi = i \begin{pmatrix}
  \xi_1 & \xi_3 \\
  \xi_2 & - \xi_1
\end{pmatrix}$ on $\mathfrak{g}^\wedge$.  We henceforth fix some $0 < \delta < 1/2$.  We denote by $\tilde{\mathcal{K}}$ the set of all real-valued $a \in \h^{-\delta} S^{-\infty}_\delta$ with the following properties:
\begin{itemize}
\item $a (\xi) = 0$ unless $\xi_1 > 0$ and $\xi_1 \asymp 1$ and $\xi_2,\xi_3 \ll \h^\delta$; equivalently, $a$ is supported in a $\O(\h^\delta)$ neighborhood of some fixed compact subset of $\{\xi(t) : t > 0\}$.  In particular, $a$ vanishes identically on $\mathcal{O}(\lambda)$ unless $\lambda < 0$ and $|\lambda| \asymp 1$.
\item We have
  \begin{equation}\label{eq:Weyl-negates-symbol}
    a(-(w \cdot \xi)) = a(\xi)
  \end{equation}
  for all $\xi \in \mathfrak{g}^\wedge$, where $w = \begin{pmatrix} 0 & 1 \\ -1 & 0 \end{pmatrix} \in G$.  Equivalently, $a(\xi)$ is invariant under swapping the coordinates $\xi_2$ and $\xi_3$.
\end{itemize}

To each $a \in \tilde{\mathcal{K}}$, we attach an $\h$-dependent element $k$ of $C_c^\infty(\mathbb{R}_{<0})$ by the formula
\[
  k(\lambda) := \int_{\mathcal{O}(\lambda)} a^2.
\]
We note that the support condition for $k$ enunciated in \S\ref{sec:sketch-deduct-theor} follows from that for $a$.  We denote by $\mathcal{K}$ the set of $\h$-dependent functions $k$ arising in this way.  We verify readily that $\mathcal{K}$ has the ``control'' and ``richness'' properties enunciated in \S\ref{sec:orgb2d0a08}.

We henceforth work with one such $k$ together with a corresponding symbol $a \in \tilde{\mathcal{K}}$.  We fix $C \geq 10$ sufficiently large that $a (\xi) = 0$ unless $3/C < \xi_1 < C/3$.  Then $k(\lambda) = 0$ unless $2/C < \sqrt{-\lambda} < C/2$.

We say that an $\h$-dependent irreducible unitary representation $\pi$ of $G$ is \emph{good} if $\lambda_\pi < 0$ and $1/C \leq \sqrt{-\h^2 \lambda_\pi} \leq C$, and otherwise that $\pi$ is \emph{bad}.  Informally, the bad $\pi$ are those whose rescaled infinitesimal characters are sufficiently separated from the support of $a$ that they play little role in our analysis.

\begin{lemma*}
  Let $\pi$ be an $\h$-dependent irreducible unitary representation of $G$.  Set
  \[
    T := \h \Opp_{\h}(a)^2,
  \]
  where $\Opp_{\h}(a) := \Opp_{\h}(a:\pi)$.
  \begin{enumerate}[(i)]
  \item\label{itm:props-of-T-pi} $T$ defines a positive-definite trace-class operator on $\pi$.  If $\pi$ is good, then
    \begin{equation}\label{eq:trace-estimate-good-case}
      \trace(T)
      =
      k(\h^2  \lambda_\pi)
      + \O(\h^{1-\delta}).
    \end{equation}
    If $\pi$ is bad, then
    \begin{equation}\label{eq:bad-trace-pi-estimate}
      \trace(T)
      =
      \O(\h^\infty \langle \lambda_\pi  \rangle^{-\infty}).
    \end{equation}
  \item\label{itm:estimate-mu-T-pi} Fix a mean-zero even eigenfunction $\Psi \in \sigma \in A_0$, and suppose that $\pi \in A_0$.  If $\pi$ is good, then
    \begin{equation}\label{eq:compare-with-microlocal}
      \mu(T,\Psi)
      =
      k(\h^2 \lambda_\pi)
      \mu_\pi(\Psi)
      + \O(\h^{\delta}
      \sqrt{\h \mathcal{L}}
      + \h^\infty),
    \end{equation}
    where $\mathcal{L} := \mathcal{L}(\pi,\sigma)$.  If $\pi$ is bad, then
    \begin{equation}\label{eq:bad-pi-estimate-for-T-pi-Psi}
      \mu(T,\Psi)
      =
      \O(\h^\infty \langle \lambda_\pi  \rangle^{-\infty}).
    \end{equation}
  \end{enumerate}
\end{lemma*}
\begin{proof}
  If $\pi$ is bad, then the estimate \eqref{eq:bad-trace-pi-estimate} follows from \S\ref{sec:clean-up}, while \eqref{eq:bad-pi-estimate-for-T-pi-Psi} follows from \eqref{eq:bad-trace-pi-estimate} and the inequality $|\mu(T,\Psi)| \leq \trace(T) \|\Psi\|_{L^\infty}$.  Suppose that $\pi$ is good.  Then $\pi$ is tempered, $\lambda_\pi < 0$ and $|\h^2 \lambda_\pi| \asymp 1$.  Moreover, every element of $\supp(a) \cap \h \mathcal{O}_\pi$ is of the form $\h \tau + \O(\h^\delta)$ with $\tau = \xi(t)$, $t = \sqrt{- \Omega_\pi}$.  The hypotheses of \eqref{eq:trace-estimate-in-localized-case} are thus satisfied, while the conclusion gives \eqref{eq:bad-trace-pi-estimate}.  To deduce \eqref{eq:compare-with-microlocal}, we combine the lemmas of \S\ref{sec:states-approx-microlocal}.
\end{proof}

\section{Proofs of main results}\label{sec:35ac3e57a1}
Following the sketch of \S\ref{sec:sketch-deduct-theor}, we must establish the three key estimates \eqref{eqn:req-V}, \eqref{eqn:req-I} and \eqref{eqn:req-E}.  These are recalled and verified in the following subsections, followed by the modifications needed for the holomorphic case.

\subsection{Quantum variance
  sums via integral operators}\label{sec:defn-f}
Recall from \S\ref{sec:org8bcf1bd} that we have chosen $k \in \mathcal{K}$ and a corresponding symbol $a \in \tilde{\mathcal{K}}$.  Then
\[
  f := \h^{3/2} \widetilde{\Opp}_{\h}(a) \ast \widetilde{\Opp}_{\h}(a)
\]
(here $\ast$ denotes convolution in $C_c^\infty(G)$ with respect to our chosen Haar measure) is an $\h$-dependent positive-definite element of $C_c^\infty(G)$, supported in a fixed small neighborhood of the identity element.  For each unitary representation $\pi$ of $G$, we have $\pi(f) = \h^{1/2} T_\pi$ with $T_\pi := \h \Opp_{\h}(a:\pi)^2$ as in the lemma of \S\ref{sec:org8bcf1bd}.  Recall from \S\ref{sec:orgb2d0a08} that
\begin{equation}\label{eq:recap-defn-cal-V}
  \mathcal{V}(f)
  =
  \sum_{\pi \in A_0}
  \iota_{\pi}
  \mu(\pi(f),\Psi_1)
  \overline{\mu(\pi(f),\Psi_2)}
  =
  \h \sum_{\pi \in A_0} \iota_\pi \mu(T_\pi,
  \Psi_1) \overline{ \mu(T_\pi, \Psi_2) }.
\end{equation}
The purpose of this section is to verify the claimed estimate \eqref{eqn:req-V} relating $\mathcal{V}(f)$ to the quantum variance of microlocal lifts, which we copy here for convenience:
\begin{equation}
  \mathcal{V}(f)
  = \h \sum_{\pi \in A_0}
  \iota_{\pi} k(\h^2 \lambda_\pi)^2
  \mu_\pi(\Psi_1)
  \overline{\mu_\pi(\Psi_2)}
  + \O(\h^\delta).
  \tag{\ref{eqn:req-V}}
\end{equation}
We begin with an \emph{a priori} bound:
\begin{lemma*}
  Fix $C \geq 1$, and fix an even $\sigma \in A_0$.  Then
  \begin{equation}\label{eq:SZ-upper-bound-2}
    \h^2
    \sum_{
      \substack{
        \pi \in A_0 : \\
        C^{-1} \leq - \h^2 \lambda_\pi \leq C
      }
    } \iota_{\pi}
    \mathcal{L}(\pi,\sigma)
    \ll 1.
  \end{equation}
\end{lemma*}
\begin{proof}
  This can be deduced using an approximate functional equation and the Kuznetsov--Bruggeman--Miatello formula as in the work of Luo--Sarnak--Zhao (who in fact obtain and require asymptotic formulas with strong error terms rather than merely upper bounds \eqref{eq:SZ-upper-bound-2} of the expected order of magnitude).  For completeness, we record a self-contained proof of \eqref{eq:SZ-upper-bound-2}.
  We assume $k \in \mathcal{K}$ chosen so that $k(\lambda) \geq 1$ whenever $C^{-1} \leq -\lambda \leq C$.  Let $\pi$ be as in \eqref{eq:SZ-upper-bound-2}, so that $|k(\h^2 \lambda_\pi)|^2 \geq 1$.  Set $T_\pi := \h \Opp_{\h}(a:\pi)^2$.  Fix an eigenfunction $\Psi \in \sigma$ for which  \eqref{eq:WW-lower-bound} holds, so that $|k(\h^2 \lambda_\pi)|^2 |\mu_\pi(\Psi)|^2 \gg \h \mathcal{L}(\pi,\sigma)$.
  It follows then by \eqref{eq:compare-with-microlocal} that $|\mu(T_\pi,\Psi)|^2 \gg \h \mathcal{L}(\pi,\sigma)$.  Thus the left hand side of \eqref{eq:SZ-upper-bound-2} is bounded by a fixed multiple of $\h \sum_{\pi \in A_0 } \iota_{\pi} |\mu(T_\pi,\Psi)|^2$, which is just $\mathcal{V}(f)$ specialized to $\Psi_1 = \Psi_2 = \Psi$.  The identity \eqref{eq:var-2-main-identity} and the estimates \eqref{eqn:req-I} and \eqref{eqn:req-E} give an asymptotic formula for $\mathcal{V}(f)$ which implies in particular that $\mathcal{V}(f) \ll 1$.  This completes the proof.  (We note that the proofs of the estimates \eqref{eqn:req-I} and \eqref{eqn:req-E}, given below, do not depend upon the lemma that we are proving, so our argument is non-circular.)
\end{proof}

We now verify \eqref{eqn:req-V}.  The contribution from bad $\pi$ to \eqref{eq:recap-defn-cal-V} is adequately estimated using \eqref{eq:bad-pi-estimate-for-T-pi-Psi} and the very weak Weyl law $\h^{100} \sum_{\pi \in A_0} \iota_\pi \langle \lambda_\pi \rangle^{-10} \ll \h^{10}$, say.  If $\pi$ is good, then we see by \eqref{eq:WW-upper-bound} and \eqref{eq:compare-with-microlocal} that
\begin{equation}\label{eqn:approx-for-mu-mu}
  \begin{split}
    \mu(T_\pi,\Psi_1) \overline{\mu(T_\pi,\Psi_2)} &= k(\h^2 \lambda_\pi)^2 \mu_\pi(\Psi_1) \overline{\mu_\pi(\Psi_2)}
    \\
    &\quad + \O(\h^{1+\delta} \sqrt{\mathcal{L}_1 \mathcal{L}_2} + \h^{\infty})
  \end{split}
\end{equation}
where $\mathcal{L}_j := \mathcal{L}(\pi,\sigma_j)$.  To discard the error, we apply Cauchy--Schwarz followed by the above lemma, which gives for $j=1,2$ that
\[
  \h \sum_{\text{good }\pi \in A_0} \iota_{\pi} (\h^{1+\delta} \mathcal{L}_j + \h^\infty) \ll \h^\delta.
\]
The proof of \eqref{eqn:req-V} is now complete.  In the following sections we will verify \eqref{eqn:req-I} and \eqref{eqn:req-E}, thereby completing the proof of Theorem \ref{thm:var-3}.

We note for future reference that for $x \in \mathcal{G}$ (see \S\ref{sec:kirillov-formula}),
\begin{equation}\label{eq:cnsez7nean}
  f(\exp(x)) = \h^{3/2} \jac^{-1/2}(x) b^\vee_{\h}(x),
\end{equation}
where $b \in \h^{-\delta} S^{-\infty}_{\delta}$ is characterized by the identity $b_{\h}^{\vee} = \jac^{-1/2} (a \star_{\h} a)_{\h}^\vee$.  By \cite[\S7.8]{nelson-venkatesh-1}, it admits an asymptotic expansion $b \sim \sum_{j \geq 0} b_j$ with $b_0 = a^2$ and $b_j \in \h^{-2 \delta + (1 - 2 \delta) j} S^{-\infty}_{\delta}$; by this we mean that $b - \sum_{0 \leq j < J} b_j \in \h^{- 2 \delta + (1 - 2 \delta )J} S^{-\infty}_{\delta}$ for each fixed $J \in \mathbb{Z}_{\geq 0}$.

\subsection{Main term estimates}
\label{sec:orge0a88b2}
Before beginning the proof of \eqref{eqn:req-I}, we establish a relevant integral formula.  Recall that we have fixed a Haar measure $d g$ on $G$.  Let $d x$ denote the compatibly normalized Haar measure on $\mathfrak{g}$ (\S\ref{sec:kirillov-formula}), and let $d \xi$ denote the corresponding dual measure on $\mathfrak{g}^\wedge$, so that for instance $\phi(0) = \int_{x \in \mathfrak{g}} (\int_{\xi \in \mathfrak{g}^\vee} \phi(\xi) e^{\langle x, \xi \rangle}\, d \xi ) \, dx$ for $\phi \in C_c^\infty(\mathfrak{g}^\wedge)$.  Let $\Phi \in C_c(G)$ and $\phi_1,\phi_2 \in C_c(\mathfrak{g}^\wedge)$.  The integral
\begin{equation*}
  I := \int_{g \in G} \Phi(g) \int_{\xi \in \mathfrak{g}^\wedge} \phi_1(g \cdot \xi) \phi_2(\xi) \, d \xi \, d g
\end{equation*}
is then independent of these choices of measure.  Our immediate aim is to express $I$ in terms of the normalized symplectic measures on coadjoint orbits.  We do this under the assumption that the $\phi_j$ are supported on the ``negative cone'' $\{\xi : \Lambda(\xi) < 0\} = \cup_{t > 0} \mathcal{O}(-t^2)$, which is the case relevant for our applications.

To state our result requires some notation.  Let $t > 0$.  Recall that $\xi(t)$ (see \eqref{eqn:defn-xi-of-t}) has $G$-orbit $\mathcal{O}(-t^2)$ and stabilizer $H$ (the diagonal subgroup).  For $\xi,\eta \in \mathcal{O}(-t^2)$, set
\[G_{\xi \leftarrow \eta} := \{g \in G : g \cdot \eta = \xi \}.\] For any choice of elements $x,y \in G$ with $x \cdot \xi(t) = \xi$ and $y \cdot \xi(t) = \eta$, we obtain a bijection $H \cong G_{\xi \leftarrow \eta}$ given by $s \mapsto x s y^{-1}$.  We equip $G_{\xi \leftarrow \eta}$ with the transport of the Haar measure on $H$, as normalized in \S\ref{sec:intro-setup}.  The measure so-defined on $G_{\xi \leftarrow \eta}$ is independent of the choice of $x,y$.
\begin{lemma*}
  Let $I$ be as defined above, with $\phi_1,\phi_2$ supported on $\{\xi : \Lambda(\xi) < 0\}$.  Then
  \begin{equation}
    I = \int_{t > 0}
    \int_{\xi,\eta \in \mathcal{O}(-t^2)}
    \phi_1(\xi) \phi_2(\eta)
    \left(\int_{G_{\xi \leftarrow \eta}} \Phi\right)
    \, d \omega (\xi ) \, d \omega (\eta )
    \, \frac{d t}{2 \pi }.
  \end{equation}
\end{lemma*}
\begin{proof}
  Although both sides of the identity are independent of all choices of Haar measure, it is convenient to make explicit choices for the proof.  Recall the coordinates and notation of \S\ref{sec:lie-algebra}.  We assume that $d x = \frac{1}{ \pi } d x_1 \, d x_2 \, d x_3$.  This normalizes a Haar measure $d g$ on $G$, as well as the dual measure $d \xi = \frac{1}{8 \pi^2} d \xi_1 \, d \xi_2 \, d \xi_3$ on $\mathfrak{g}^\wedge$.  We equip $G/H$ with the quotient measure $d g$.  We note in passing that $H$ meets both connected components of $G$, so the quotient $G/H$ could be replaced by the quotient $G^1/H^1$ of connected groups in what follows.

  Let $t > 0$.  We first explicate the value $\omega_{\xi(t)}$ of the canonical $2$-form $\omega$ on $\mathcal{O}(-t^2)$ (\S\ref{sec:coadjoint-orbits}) at the point $\xi(t)$.  Under the differentiated orbit map $\mathfrak{g} \rightarrow T_{\xi(t)}(\mathcal{O}(-t^2))$, we have $e_2 \mapsto - 2 t e_3^*$ and $e_3 \mapsto 2 t e_2^*$.  Thus $\omega_{\xi(t)}(- 2 t e_3^* \wedge 2 t e_2^*) = \langle \xi(t), [e_2, e_3] \rangle/2 \pi i = 2 t / 2 \pi$, and so $\omega_{\xi(t)} = \frac{1}{4 \pi t } d \xi_2 \wedge d \xi_3$.

  Let $\beta = \frac{1}{8 \pi^2} d \xi_1 \wedge d \xi_2 \wedge d \xi_3$ denote the differential form on $\mathfrak{g}^\wedge$ corresponding to $d \xi$.  Then $\beta_{\xi(t)} = \frac{t}{2 \pi } d \xi_1 \wedge \omega_{\xi(t)}$.  This implies the integral formula
  \begin{equation*}
    \int_{\xi \in \mathfrak{g}^\wedge} \phi(\xi) \, d \xi
    = \int_{t > 0}
    \left(\int_{\mathcal{O}(-t^2)} \phi \right)
    \, \frac{t \, d t}{ 2 \pi }
  \end{equation*}
  for $\phi$ supported on $\cup_{t>0} \mathcal{O}(-t^2)$.

  On the other hand, the Haar measure on $H$ corresponds to the Haar measure on $\mathfrak{h} = \Lie(H)$ given by $d x_1$.  The induced quotient measure on $G/H$ corresponds to the differential form on $\mathfrak{g}/\mathfrak{h}$ given at the origin by $\frac{1}{\pi } d x_2 \wedge d x_3$.  Under the orbit map isomorphism $\mathfrak{g}/\mathfrak{h} \cong T_{\xi(t)} \mathcal{O}(-t^2)$, we have $\frac{1}{ \pi } d x_2 \wedge d x_3 \mapsto \pm \frac{1}{4 \pi t^2} d \xi_2 \wedge d \xi_3 = \pm \frac{1}{t} \omega_{\xi(t)}$.  This implies the integral formula
  \begin{equation*}
    \int_{\mathcal{O}(-t^2)} \phi 
    = t
    \int_{g \in G/H} \phi(g \cdot \xi(t)) \, d g.
  \end{equation*}

  Combining the formulas established thus far, we obtain
  \begin{align*}
    I
    &=
      \int_{g \in G}
      \Phi(g)
      \int_{t > 0}
      \int_{\xi \in \mathcal{O}(-t^2)}
      \phi_1(g \cdot \xi) \phi_2(\xi)
      \, d \omega (\xi )
      \, \frac{t \, d t}{2 \pi } \, d g
    \\
    &=
      \int_{g \in G}
      \Phi(g)
      \int_{t > 0}
      \int_{y \in G/H}
      \phi_1(g  \cdot \xi(t)) \phi_2(y \cdot \xi(t))
      \, d y
      \, \frac{t^2 \, d t}{2 \pi } \, d g
    \\
    &=
      \int_{t > 0}
      \int_{x,y \in G/H}
      \phi_1(x \cdot \xi(t)) \phi_2(y \cdot \xi(t))
      \int_{s \in H}
      \Phi(x s y^{-1}) \, d s \, d x \, d y
      \, \frac{t^2 \, d t}{2 \pi },
  \end{align*}
  which simplifies to the required formula.
\end{proof}

We now verify \eqref{eqn:req-I}, which we copy here for convenience:
\begin{equation}
  \mathcal{I}(f)
  = \int_{s \in H} \langle s \Psi_1^w, \Psi_2^w
  \rangle \,d s 
  \int_{t > 0}
  k(-t^2)^2 \, \frac{d t}{2 \pi}
  + \O(\h^\delta).
  \tag{\ref{eqn:req-I}}
\end{equation}
Using \eqref{eq:exact-equivariance-K} and \eqref{eq:Weyl-negates-symbol}, we compute that $\Ad(w) f(g) = f(g^{-1})$, hence $\mathfrak{S} f = \frac{f + \Ad(w) f}{2}$, and so
\begin{equation}\label{eq:I-f-after-simpl}
  \mathcal{I}(f)
  =
  \int_{g \in G}
  \langle \Ad(g) f, f \rangle_{G}
  \Phi(g) \,d g ,
  \quad 
  \Phi(g) := \langle g \Psi_1^w, \Psi_2^w \rangle.
\end{equation}
We compatibly normalize Haar measures on $G$, $\mathfrak{g}$ and $\mathfrak{g}^\wedge$ as above.  By change of variables, \eqref{eq:cnsez7nean} and Parseval, we then have
\begin{align*}
  \langle \Ad(g) f, f \rangle_G
  &=
    \int_{x \in \mathfrak{g}}
    \jac(x)
    f(g^{-1} \exp(x) g ) \overline{f(\exp(x))}
    \, d x
  \\
  &=
    \h^3
    \langle g \cdot b_{\h}^\vee, b_{\h}^\vee
    \rangle_{\mathfrak{g}}
  \\
  &= \langle g \cdot b, b \rangle_{\mathfrak{g}^\wedge}.
\end{align*}

As explained in \S\ref{sec:local-convergence-lemmas}, all of these integrals converge absolutely.  For instance, we have $\Phi(g) \ll \|\Ad(g)\|^{-\eta}$ for some fixed $\eta > 0$ by standard bounds for matrix coefficients (\S\ref{sec-2-1-4}), while $\langle g \cdot b, b \rangle_{\mathfrak{g}^\wedge} \ll \|\Ad(g)\|^{-1}$ by direct calculation.  Since $\int_{g \in G} \|\Ad(g)\|^{-1-\eta} \,d g  < \infty$, the required convergence follows.  The same argument gives also that for any $c_1, c_2 \in S^{-\infty}_{\delta}$,
\begin{equation}\label{eqn:a-priori-g-c1-c2}
  \int_{g \in G} \Phi(g)
  \langle g \cdot c_1, c_2 \rangle_{\mathfrak{g}^\wedge}  \,d g 
  \ll 1.
\end{equation}

Using \eqref{eqn:a-priori-g-c1-c2} as an \emph{a priori} estimate and inserting the asymptotic expansion $b \sim \sum_{j \geq 0} b_j$ noted previously, we deduce that for each fixed $N \geq 0$, there is a fixed $J \geq 0$ so that
\begin{equation*}
  \mathcal{I}(f) = \sum_{0 \leq j_1,j_2 < J} I_{j_1,j_2} + \O(\h^N), \quad I_{j_1,j_2} := \int_{g \in G} \Phi(g) \langle g \cdot b_{j_1}, b_{j_2} \rangle_{\mathfrak{g}^\wedge} \,d g .
\end{equation*}
The above lemma gives, with $\Phi'(\xi,\eta) := \int_{G_{\xi \leftarrow \eta}} \Phi$, that
\begin{equation}\label{eqn:see-coadjoint-orbits-appear-yay}
  I_{j_1,j_2}
  =
  \int_{t > 0}
  \int_{\xi, \eta \in \mathcal{O}(-t^2)}
  \Phi'(\xi,\eta)
  \overline{b_{j_2}}(\xi) 
  b_{j_1}(\eta)
  \, d \omega (\xi ) \, d \omega (\eta )
  \, \frac{d t}{2 \pi}.
\end{equation}
Using the orbit map at $\xi(t) \in \mathcal{O}(-t^2)$, we can view $\Phi'$ as a function on $(G/H)^2$.  Using the absolute convergence of the integral defining $\Phi'$ and the smoothness of the vectors $\Psi_1, \Psi_2$, we see that the function $\Phi'$ is smooth (see \cite[\S18]{nelson-venkatesh-1} for related arguments).  In particular, $\Phi'$ is Lipschitz near the origin, where it takes the value $\Phi'(\xi(t),\xi(t)) = \int_H \Phi$.  On the other hand, the factor $\overline{b_{j_2}}(\xi) b_{j_1}(\eta)$ vanishes unless $\xi,\eta = \xi(t) + \O(\h^\delta)$; in that case, it is bounded in magnitude by $\O(\h^{-4\delta + (1 - 2 \delta)(j_1 + j_2)})$, and we have $\Phi'(\xi,\eta) = \int_H \Phi + \O(\h^\delta)$.  The volume of the set of such pairs $(\xi,\eta)$ is $\O(\h^{4 \delta})$.  It follows that
\[
  I_{j_1,j_2} \ll \h^{(1 - 2 \delta) (j_1 + j_2)}.
\]
In particular, since $\delta$ is sufficiently small, we have $I_{j_1,j_2} \ll \h^\delta$ if $(j_1,j_2) \neq (0,0)$.  On the other hand, since $b_0 = a^2$ and $\int_{\mathcal{O}(-t^2)} a^2 = k(-t^2)$, we have
\begin{equation*}
  I_{0,0} = \left(\int_H \Phi\right) \int_{t > 0} k(-t^2)^2 \, \frac{d t}{2 \pi} + \O(\h^{\delta}).
\end{equation*}
This completes the proof of the required estimate for $\mathcal{I}(f)$.

\begin{remark*}
  An alternative proof may be obtained by first decomposing $\langle \Ad(g) f, f \rangle$ over the spectrum of $L^2(G)$ as the integral of the Hilbert--Schmidt inner products $\langle \pi(g) \pi(f) \pi(g)^{-1}, \pi(f) \rangle$.  This decomposition is reflected above, in \eqref{eqn:see-coadjoint-orbits-appear-yay}, at the level of coadjoint orbits.
\end{remark*}

\subsection{Error estimates}
\label{sec:org543f416}
We now prepare for the verification of \eqref{eqn:req-E}, which requires some Lie-theoretic preliminaries\label{sec:error-lie-prelim}.  We begin with some trivial remarks concerning the complex plane which we hope convey a useful reference picture for the estimates to follow.  There are two common choices of coordinates: the rectangular coordinates, described by real and imaginary part, and the polar coordinates, described by radius and angle.  The former is adapted to addition, the latter to multiplication.  If we restrict the radius to a fixed compact subset of the positive reals and the angle to a sufficiently small neighborhood of the origin, then the real part is likewise restricted to a fixed compact subset of the positive reals; moreover, the imaginary part and the angle are bounded from above and below by constant multiples of one another.  These restrictions define a region in the complex plane.  Given a scalar-valued function $\phi$ on that region and a small scaling parameter $\h > 0$, there are two natural ways to rescale $\phi$ so that its support concentrates along the positive reals: in rectangular coordinates (by scaling the imaginary part) or in polar coordinates (by scaling the angle).  The two classes of rescaled functions obtained in this way resemble one another.  We aim now to record some analogues and elaborations of these observations with the complex numbers replaced by the $2 \times 2$ matrix algebra $M$.  Such considerations are natural because we are ultimately studying a problem in multiplicative harmonic analysis (the variance sums $\mathcal{V}(f)$ attached to test functions $f$ on the group $G = \PGL_2(\mathbb{R})$) using additive harmonic analysis (via theta functions attached to Schwartz functions $\heartsuit^{\tau} f$ on the matrix algebra $M = M_2(\mathbb{R})$).

We denote by $\mathfrak{g} = \slLie_2(\mathbb{R})$ the Lie algebra of $G$.  We may identify $\mathfrak{g}$ with the subspace of traceless elements in $M$; then $M = \mathbb{R} \oplus \mathfrak{g}$, where $\mathbb{R}$ is the subspace of scalar matrices.

Let $\mathcal{R} \subseteq \mathbb{R}^\times_+$ and $\mathcal{G} \subseteq \mathfrak{g}$ be precompact open subsets.  We assume that $0 \in \mathcal{G}$ and that $\mathcal{G}$ is star-shaped: $\h \mathcal{G} \subseteq \mathcal{G}$ for $\h \in [0,1]$.  We assume that $\mathcal{G}$ is taken small enough in terms of $\mathcal{R}$; in particular, we assume that the map
\[
  \mathcal{R} \times \mathcal{G} \ni (r,x) \mapsto r^{1/2} \exp(x) \in \GL_2^+(\mathbb{R})
\]
is an analytic isomorphism onto its image, which we denote by $\mathcal{M}$.  We write
\[
  (\rho,\theta) : \mathcal{M} \rightarrow \mathcal{R} \times \mathcal{G}
\]
for the inverse isomorphism.  Thus for $v \in \mathcal{M}$, we have $\rho(v) = \sqrt{\det(v)}$, while $\theta(v)$ is the logarithm of the image of $v$ in $G$.  We informally regard $\rho(v)$ and $\theta(v)$ as the respective radial and angular parts of $v$.

Every element of $M_2(\mathbb{R})$ may be written uniquely in the form $t + u$, where $t \in \mathbb{R}$ and $u \in \mathfrak{g}$.  Write $\gamma(t,u) = (\rho(t,u), \theta(t,u)) =: (r,x)$, say.  Then
\[
  r = \sqrt{t^2 - u^2}, \quad x = \frac{1}{2} \log \frac{t + u }{t - u}.
\]
We informally regard $(t,u)$ and $(r,x)$ as the respective rectangular and polar coordinates on $\mathcal{M}$.

Since $\mathcal{G}$ is small, we have $t \neq 0$.  Thus $r,t$ are both constrained to lie in compact subsets of $\mathbb{R}^\times$, and so
\begin{equation}\label{eq:r-t-same-size}
  |r| \asymp 1 \asymp |t|.
\end{equation}

Moreover,
\begin{equation}\label{eq:x-u-same-size}
  |x| \asymp |u|.
\end{equation}
Indeed, we may expand the analytic map
\begin{equation}\label{eq:gamma-2-of-t-u}
  \mathbb{R} \oplus \mathfrak{g}
  \ni (t,u) \mapsto \theta(t+u) \in \mathcal{G}
\end{equation}
as a Taylor series $\sum_{n \geq 1} c_n (u/t)^n$, and similarly for the inverse map.  By \eqref{eq:r-t-same-size}, it follows that $x$ and $u$ tend to zero simultaneously and at the same rate.  Since the magnitudes of both are bounded from above, we deduce \eqref{eq:x-u-same-size}.

We now fix a cutoff $q \in C_c^\infty(\mathcal{M})$ and define a family of maps of Schwartz spaces
\begin{equation*}
  \mathcal{S}(\mathfrak{g}) \rightarrow \mathcal{S}(M),
  \quad 
  \phi \mapsto \Phi_{\h},
\end{equation*}
indexed by $\h \in (0,1)$, by the formula
\begin{equation*}
  \Phi_{\h}(t+ u) := q(t + \h u) \phi(\h^{-1} \theta(t + \h u)).
\end{equation*}
Informally, $\Phi_{\h}$ is obtained from $\Phi_{1}$ by rescaling in two stages: first shrinking the angular support in polar coordinates, then stretching the imaginary support in rectangular coordinates.  The reference picture discussed above hopefully renders the following observation unsurprising:
\begin{lemma*}
  The family of maps $\phi \mapsto \Phi_{\h}$ is equicontinuous for the Schwartz topology.
\end{lemma*}
\begin{proof}
  The derivatives of the map $(t,u) \mapsto q(t+\h u)$ are uniformly bounded.

  Moreover, the derivatives of the map $(t,u) \mapsto \h^{-1} \theta(t + \h u)$ are bounded, uniformly for $t+\h u \in \supp(q)$.  Indeed, we may write each such derivative as a convergent Taylor series; applying the triangle inequality gives a finite bound for the magnitude of this series, and the observation following \eqref{eq:gamma-2-of-t-u} implies that this bound improves as $\h$ decreases.

  By the chain rule, we deduce that the derivatives of $\Phi_{\h}$ at $t + u$ are dominated by derivatives of $\phi$ at those elements $\h^{-1} \theta(t +\h u)$ for which $t + \h u \in \supp(q)$.  For such elements, the estimates \eqref{eq:r-t-same-size} and \eqref{eq:x-u-same-size} give $|t| \asymp 1$ and $|\h^{-1} \theta(t + \h u)| \asymp |u|$.  Thus the rapid decay of the derivatives of $\Phi_{\h}$ follows from that of $\phi$.
\end{proof}

We now apply these considerations to establish \eqref{eqn:req-E}, which we copy here for convenience:
\begin{equation}
  \mathcal{E}_{\tau_1,\tau_2}(
  \heartsuit^{\tau_1} f,
  \heartsuit^{\tau_2} f)
  \ll \h^{1-\delta'}.
  \tag{\ref{eqn:req-E}}
\end{equation}
Let $\tau \in \{\tau_1, \tau_2\}$.  We assume $\mathcal{R}$ taken large enough to contain the support of $r \mapsto W(\tau r)$, and take $\mathcal{G}$ small enough.  We may assume that $\Opp$ was defined with respect to a cutoff supported in $\mathcal{G}$.  Set
\[
  q^\tau(v) := \frac{W(\tau \rho(v)^2) }{ |\tau \rho(v)^2| } \jac^{-1/2} \chi'(\theta(v))
\]
and $\phi := b^\vee$, with $b$ as in \S\ref{sec:defn-f}.  We see then by unwinding the definitions and applying \eqref{eq:cnsez7nean} that for $v \in \mathcal{M}$,
\[
  \heartsuit^{\tau} f(v) = q^\tau(v) \h^{-3/2} \phi(\h^{-1} \theta(v)).
\]
Let $\Phi_{\h}^\tau$ be defined as in \S\ref{sec:error-lie-prelim} using $q^\tau$ and $\phi$.  Then $\heartsuit^{\tau}f(t + \h u) = \h^{-3/2} \Phi_{\h}^\tau(t + u)$.  We may express this identity in terms of the normalized dilation operators $D_y$ (see \S\ref{sec:orgb2d0a08}) as
\[
  \heartsuit^{\tau} f = D_{1/\h} \Phi_{\h}^\tau,
\]
so by Theorem \ref{thm:var-2},
\[
  \mathcal{E}_{\tau_1,\tau_2}(\heartsuit^{\tau_1} f, \heartsuit^{\tau_2} f) = \mathcal{E}_{\tau_1,\tau_2} (D_{1/\h} \Phi_{\h}^{\tau_1}, D_{1/\h} \Phi_{\h}^{\tau_2}) \ll \h^{1} \log (\h^{-1}) \mathcal{C}(\Phi_{\h}^{\tau_1}) \mathcal{C}(\Phi_{\h}^{\tau_2}),
\]
for some fixed continuous seminorm $\mathcal{C}$.  We appeal now to the ($\h$-uniform) continuity of $\phi \mapsto \Phi_{\h}^{\tau}$ noted in \S\ref{sec:error-lie-prelim}, together with the continuity of the map $a \mapsto b$ composed with the Fourier transform $b \mapsto b^\vee = \phi$, to write $\mathcal{C}(\Phi_{\h}^{\tau_1}) \mathcal{C}(\Phi_{\h}^{\tau_2}) = \O(\mathcal{C}(a)^2)$ for some fixed continuous seminorm $\mathcal{C}$ on $\mathcal{S}(\mathfrak{g}^\wedge)$.  The definition of $S^{-\infty}_{\delta}$ implies that $\mathcal{C}(a) \ll \h^{- N \delta}$ for some fixed $N \in \mathbb{Z}_{\geq 0}$, so we may conclude by taking $\delta' := (2 N + 1) \delta$.

\begin{remark*}
  The proof of \eqref{eqn:req-E} recorded above is a bit different from that of the corresponding estimate in \cite{nelson-variance-II}, whose analogue here would be to exploit the smoothness of $\Psi$ and the diagonal $G$-invariance of $(f,\Psi) \mapsto \mu(\pi(f),\Psi)$ (i.e., the ``$\SO(B_S^0)$-invariance'' noted in Theorem \ref{thm:main-estimate-general-variance}) to ``fatten up'' the symbol $a$ under the adjoint action.  The argument given here produces weaker estimates, but is a bit shorter and simpler.
\end{remark*}

The proof of Theorem \ref{thm:var-3} is now complete.

\subsection{The holomorphic analogue}\label{sec:holomorphic-analogue-1}
Central to our analysis of microlocal lifts of Maass forms was the definition \eqref{eqn:defn-xi-of-t} of $\xi(t) \in \mathcal{O}(-t^2)$.  We may reformulate our definition of $H$ and choice of $w$ as follows:
\begin{itemize}
\item $H$ is the centralizer of $\xi(t)$ (for $t \neq 0$).
\item $w$ is an element of $N(H) - H$.  Then $\xi(t) = - w \cdot \xi(t)$.
\end{itemize}

For the proof of the holomorphic analogue (Theorem \ref{thm:var-quat-annals-submission:limit-begin-lim_h}), we take instead
\begin{equation*}
  \xi'(t) = i \begin{pmatrix}
    0 & -t \\
    t & 0
  \end{pmatrix} \in \mathcal{O}(t^2).
\end{equation*}
Thus $X(\xi'(t)) = Y(\xi'(t)) = 0$, while $W(\xi'(t)) = t$.  We choose $H$ and $w$ analogously:
\begin{equation*}
  H = K^1,
  \quad
  w = \begin{pmatrix}
    0 & 1 \\
    1 & 0
  \end{pmatrix}
  \in N(H) - H = K - K^1.
\end{equation*}
The relevant orbits are now the $\mathcal{O}(t^2)$ rather than the $\mathcal{O}(-t^2)$.  The relevant representations are now the $\pi \in A_0$ with $\lambda_\pi  = (k-1/2)^2 > 0$ for some natural number $k$.  Recall from \S\ref{sec:holomorphic-analogue} that for such $\pi$, we define $\mu_\pi$ to be the $L^2$-mass of a normalized lowest weight vector $\varphi_\pi$.  One can motivate the definition of $\xi '(t)$, as in \S\ref{sec:states-approx-microlocal}, in terms of the approximate symmetries of $\varphi_\pi$.  As partial motivation, we note that $W \varphi_\pi = W(\xi '(k)) \varphi_\pi$.

The analogue of Lemma \ref{lem:WW} of \S\ref{sec:branch-coeff} remains valid, with the same proof.

The analogue of Proposition \ref{thm:characterize-mu-pi}  of \S\ref{sec:states-approx-microlocal}, obtained by
\begin{itemize}
\item replacing the assumption $\lambda_\pi < 0$ with $\lambda_\pi > 0$, and
\item replacing $\tau = \xi(\sqrt{- \Omega_\pi })$
  with
  $\tau = \xi'(\sqrt{\Omega_\pi })$,
\end{itemize}
remains valid, by similar arguments.

In \S\ref{sec:org8bcf1bd}, we now define $\tilde{K}$ by requiring that $a$ be supported in a $\O(\h^\delta)$ neighborhood of some fixed compact subset of $\{\xi '(t) : t > 0\}$ and satisfy $a(-(w \cdot \xi)) = a(\xi)$, with the ``new'' value of $w$.  We define $\mathcal{K} \subseteq C_c^\infty(\mathbb{R}_{>0})$ analogously.  The analogue of the lemma of \S\ref{sec:org8bcf1bd} then holds, with the same proof.

The discussion of \S\ref{sec:defn-f} applies upon replacing each occurrence of $-\lambda_\pi$ with $\lambda_\pi$.

Once we have chosen the measure on $H$ carefully, the lemma of \S\ref{sec:orge0a88b2} will hold upon replacing $\Lambda(\xi) < 0$ with $\Lambda(\xi) > 0$ in the support condition and $\mathcal{O}(-t^2)$ with $\mathcal{O}(t^2)$ in the integral, with $G_{\xi \leftarrow \eta}$ defined analogously using the ``new'' definition of $H$.  The proof boiled down to a calculation with differential forms.  This can be carried out after complexifying and conjugating $\xi '(t)$ to a diagonal matrix, at which point the same calculation applies.  For the normalizations to be consistent, we need the Haar measure on $H$ to correspond to the $1$-form on $\mathfrak{h}$ that, after complexifying and conjugating $H$ to the diagonal subgroup, is given in the coordinates of \S\ref{sec:lie-algebra} by $d x_1$.  We verify readily that the appropriate Haar on $H$ is thus $\int _{H} f := \int _{\theta \in \mathbb{R} / 2 \pi \mathbb{Z} } f(e^{i \theta W}) \, d \theta$.  In particular, $\vol(H) = 2 \pi$.

The error estimates (\S\ref{sec:org543f416}) go through without modification.

We obtain in this way a result exactly analogous to Theorem \ref{thm:var-3}, but with the condition $0 < -\h^2 \lambda_\pi < 1$ replaced by $0 < \h^2 \lambda_\pi < 1$.  We simplify by evaluating $\int_{s \in H} \langle s \Psi_1^w, \Psi_2^w \rangle \,d s = 2 \pi \langle \Psi_1, \Psi_2 \rangle$, using here that each $\Psi_i$ is invariant under $K$ (and thus also under $H$ and $w$).  The proof of Theorem \ref{thm:var-quat-annals-submission:limit-begin-lim_h} is then complete.

\section{Concluding remarks}\label{sec:35ac3e57a7}
\subsection{Removing the arithmetic weights}\label{sec:remov-arithm-weights-1}
\label{sec:remov-arithm-weights}
Here we fulfill the promise made in Remark \ref{rmk:removing-weights} of \S\ref{sec:main-result} by explaining (a bit informally) how the modification factor \eqref{eq:c-sigma-defn} obtained by Sarnak--Zhao arises from the perspective of our method.  Recall that $\iota_{\pi} = L^{(S)}(\ad \pi,1)$.  The idea is to write the desired unweighted variance sums
\[
  \lim_{\h \rightarrow 0} \h \sum_{0 < - \h^2 \lambda_\pi < 1} \mu_\pi(\Psi_1) \overline{\mu_\pi(\Psi_2)}
\]
as the double limit
\begin{equation}\label{eq:double-limit}
  \lim_{\h \rightarrow 0}
  \lim_{S'}
  V(S'),
  \quad
  V(S')
  :=
  \h \sum_{0 < - \h^2 \lambda_\pi < 1}
  L^{(S')}(\ad \pi,1)
  \mu_\pi(\Psi_1)
  \overline{\mu_\pi(\Psi_2)},
\end{equation}
where $\lim_{S'}$ denotes the limit taken over increasing finite subsets $S' \supseteq S$ of the set of finite primes of $F$, ordered by inclusion.  We then try to swap the limits.  The subtlety in making this precise is that the Euler product of $L(\ad \pi,1)$ fails to converge absolutely, but because $1$ is at the edge of the critical strip, the failure is mild, so we at least expect the naive swapping of limits to produce the correct answer.  Theorem \ref{thm:main-estimate-general-variance} applies to any $S' \supseteq S$; combining it with the estimates of Part \ref{part:appl-micr-lifts} shows that as far as main terms are concerned,
\[
  V(S') \approx V(S) \prod_{\p \in S' - S} c_\sigma(p),
\]
where
\begin{equation*}
  c_\sigma(p) :=
  \frac{1}{\zeta_p(2) L_p(\sigma,\tfrac{1}{2} )}
  \int_{g \in \PGL_2(F_p)}
  \langle \Ad(g) f, f \rangle_{\PGL_2(F_p)} \Phi(g),
\end{equation*}
where
\begin{itemize}
\item $\zeta_p(s) = (1 - |p|^{-s})^{-1}$ denotes the local zeta function for $F_p$,
\item $L_p(\sigma,s) = (1 - \lambda_\sigma(p) p^{-s} + p^{-2s})^{-1}$ denotes the local factor for $L(\sigma,s)$ at $p$,
\item we fix an arbitrary Haar measure on $\PGL_2(F_p)$ (the quantity $c_\sigma(p)$ will not ultimately depend upon this choice),
\item $f$ is the normalized characteristic function $\vol(J_p)^{-1} 1_{J_p}$ of a maximal compact subgroup $J_p$ of $\PGL_2(F_p)$,
\item $\Ad(g) f(x) := f(g^{-1} x g)$ as usual, and
\item $\Phi$ is the normalized bi-$J_p$-invariant matrix coefficient of the unramified representation of $\PGL_2(F_p)$ corresponding to the action of $T_p$ on $\sigma$, so that for instance $\Phi(1) = 1$ and $\Phi(\diag(\varpi,1)) = \frac{\lambda_\sigma(\p)}{|p|^{1/2} + |p|^{-1/2}}$, with $\varpi \in F_p$ a uniformizer.
\end{itemize}
The modification factor \eqref{eq:c-sigma-defn} is then explained by the following local calculation:
\begin{lemma*}
  $c_\sigma(p) = \frac{1}{\zeta_p(2)} (1 - \frac{\lambda_\sigma(p)}{|p|^{3/2} + |p|^{1/2}})$.
\end{lemma*}
\begin{proof}
  This follows by direct calculation with the Macdonald formula \cite[Thm 4.6.6]{MR1431508} and the Cartan decomposition; we leave it to the interested reader.
\end{proof}

\subsection{Heuristics}\label{sec:heuristics}
In this section, we record a heuristic derivation of the limiting variance \eqref{eqn:limiting-variance-formula} obtained in our main result (or more precisely, its unweighted variant discussed in \S\ref{sec:remov-arithm-weights-1}).  This serves both to check of the correctness
of our results and to offer some perspective on the deviation in behavior of variance sums between arithmetic and non-arithmetic settings.

\subsubsection{Overview}\label{sec:35ac3e57aa}
We revoke our general assumptions by taking for $\Gamma$ any discrete cocompact subgroup of $G$ (possibly non-arithmetic).  The definitions of \S\ref{sec:intro-setup} adapt fairly painlessly to this setting, possibly after making some choices in the event of eigenvalue multiplicities.  We fix $\Psi \in C^\infty(\Gamma \backslash G)$ of mean zero (not necessarily an eigenfunction) and suppose given some unit vectors $v_\pi \in \pi$ for each $\pi$ in some varying family $\mathcal{F} \subset A_0$ (e.g., we might take for $v_\pi$ the vectors defined at the beginning of \S\ref{eq:mu-T-Psi-vs-mu-pi-Psi-simple}, so that the microlocal lifts are given asymptotically by $\Psi \mapsto \langle v_\pi \Psi, v_\pi \rangle$).  Our aim is to understand the asymptotics of the variance sums $\sum_{\pi \in \mathcal{F}} |\langle v_\pi \Psi, v_\pi \rangle|^2$.

Translated into representation-theoretic language, the basic idea underlying the semiclassical predictions (see \cite{MR848319}, \cite[\S15.6]{2009arXiv0911.4312Z}, \cite[\S4.1.3]{MR3204186}) in the generic non-arithmetic setting is to postulate that
\begin{equation}\label{eq:key-QV-heur}
  |\langle v_\pi \Psi, v_\pi \rangle|^2 \approx
  |\langle v_\pi \Psi, v_{\pi'} \rangle|^2
\end{equation}
whenever $\pi'$ and $\pi$ are ``close'' (denoted $\pi ' \approx \pi$) as quantified by their isomorphism classes under the group $G$.  Then
\begin{equation}\label{eq:key-QV-heur-2}
  \sum_{\pi \in \mathcal{F}} |\langle v_\pi \Psi, v_\pi
  \rangle|^2
  \approx 
  \sum_{\pi \in \mathcal{F}} \mathbb{E}_{\pi ' \approx \pi } |\langle v_\pi \Psi,
  v_{\pi'} \rangle|^2.
\end{equation}
The right hand side of \eqref{eq:key-QV-heur-2} may often be studied rigorously via semiclassical analysis, leading to predictions concerning the left hand side.

This heuristic often requires some modification.  One way that \eqref{eq:key-QV-heur} can fail is when the representations $\pi \in \mathcal{F}$ are self-dual, i.e., equal to their complex conjugates $\overline{\pi}$ (the representation-theoretic incarnation of ``time-reversal symmetry''); in that case,
\[\langle v_\pi \Psi, v_\pi \rangle = \langle \overline{v_\pi }
  \Psi, \overline{v_\pi } \rangle \text{ with } \overline{v_\pi } \in \pi.
\]
Suppose for concreteness that $\overline{v_\pi} = w v_\pi$ for some involutory element $w \in G$.  It follows then that the distributions $\Psi \mapsto \langle v_\pi \Psi, v_\pi \rangle$ are $w$-invariant.  On the other hand, there is no obvious reason to suspect that the more general distributions $\Psi \mapsto \langle v_\pi \Psi, v_{\pi'} \rangle$ are $w$-invariant when $\pi ' \neq \pi$, so \eqref{eq:key-QV-heur} can fail, most obviously when $w \Psi = - \Psi$.  The simplest way to repair this failure is to restrict from the outset to observables $\Psi$ for which $w \Psi = \Psi$.

Another way that \eqref{eq:key-QV-heur} can fail is if the space $\Gamma \backslash G$ admits a nontrivial correspondence $T$.  We may assume then that $\pi$ and $\pi '$ are $T$-eigenspaces with eigenvalues $\lambda$ and $\lambda '$.  These eigenvalues may bias the asymptotics of $\langle v_\pi \Psi, v_{\pi '} \rangle$.  The bias is most striking when $T$ is an involution and $T \Psi = \Psi$, in which case parity considerations imply that $\langle v_\pi \Psi, v_{\pi '} \rangle = 0$ unless $\lambda = \lambda '$.  Thus \eqref{eq:key-QV-heur} fails.  We can repair it by strengthening the closeness condition $\pi ' \approx \pi$ to require also that $\lambda ' \approx \lambda$.  The right hand side \eqref{eq:key-QV-heur-2} can now be estimated using semiclassical analysis on ``$G \times T$,'' leading to a modification of the expected variance asymptotics.  For instance, in the case of an involution, the modification is given by doubling; the factor $2^{\# S}$ in Theorem \ref{thm:var-3} may be explained in this way in terms of the involutory Hecke operators $T_p$ ($p \in S$).

Such modified heuristics extend easily to finite commuting families of correspondence, but their further extension to arithmetic settings as in Theorem \ref{thm:var-3}, with infinitely many commuting correspondences $T_p$, requires some care.  A naive approach is to run the heuristics first taking into account only those $T_p$ for $p$ belonging to some large finite set $P$, and then to take the limit as $P$ increases.  We implement this naive approach in detail below.  We will encounter main terms involving finite Euler products $\prod_{p \in P} L_p(\sigma,\tfrac{1}{2} )$.  Modulo the subtle business of identifying these with their formal limit, we will see that the resulting predictions are consistent with our rigorous results and also with the triple product formula and $L$-function analysis.    Peter Sarnak pointed out to us that this consistency is analogous to the heuristic derivation of the Birch--Swinnerton-Dyer conjecture via the Hardy--Littlewood method.

\subsubsection{General predictions}\label{sec:general-heur}
Turning to details, choose a Haar measure on $G$ and denote by $G^\wedge$ the tempered dual, equipped with Plancherel measure.  Equip $\Gamma \backslash G$ with the quotient Haar.  Suppose given a nice subset $\tilde{\mathcal{F}}$ of $G^\wedge$ and a nice function $f : G \rightarrow \mathbb{C}$ such that
\begin{itemize}
\item for $\pi \in \tilde{\mathcal{F}}$, the operator $\pi(f)$ is the orthogonal projection onto the line $\mathbb{C} v_\pi$ spanned by some unit vector $v_\pi \in \pi$, and
\item for $\pi \notin \tilde{\mathcal{F}}$, we have $\pi(f) = 0$.
\end{itemize}
(In practice, such assumptions are satisfied exactly only for $p$-adic groups $G$; for real groups, one should instead smoothly weight the family $\tilde{\mathcal{F}}$ and work with families of vectors in each $\pi \in \tilde{\mathcal{F}}$, as illustrated in the bulk of this paper.  We omit such technicalities from this heuristic discussion to keep the exposition clean.)  We then have the spectral decomposition $f(g) = \int_{\pi \in \tilde{\mathcal{F}}} \langle v_\pi, g v_\pi \rangle$ and the formula $\int_{\pi \in \tilde{\mathcal{F}}} |\langle g v_\pi, v_\pi \rangle|^2 = \langle \Ad(g) f, f \rangle$, with $\Ad(g) f(x) = f(g^{-1} x g)$ as before and the latter inner product taken in $L^2(G)$.  We take for $\mathcal{F} \subset A_0$ the set of all $\pi$ whose isomorphism class belongs to $\tilde{\mathcal{F}}$.  The pretrace formula reads $\sum_{\pi \in \mathcal{F}} v_\pi(x) \overline{v_\pi(y)} = \sum_{\gamma \in \Gamma } f(x^{-1} \gamma y) = \sum_{\gamma \in \Gamma} \int_{\pi \in \tilde{\mathcal{F}}} \langle x v_\pi, \gamma y v_\pi \rangle$.  Dividing this by the Weyl law $\# \mathcal{F} \approx \vol(\Gamma \backslash G) \vol(\tilde{\mathcal{F}})$ gives
\begin{equation}\label{eq:average-v-pi-x-y-heur-1}
  \mathbb{E}_{\pi \in \mathcal{F}}
  v_\pi(x) \overline{v_\pi(y)}
  \approx 
  \frac{1}{\vol(\Gamma \backslash G)}
  \sum_{\gamma \in \Gamma}
  \mathbb{E}_{\pi \in \tilde{\mathcal{F}}}
  \langle x v_\pi, \gamma y v_\pi \rangle,
\end{equation}
where $\mathbb{E}$ denotes the average (taken with respect to the counting measure on $\mathcal{F}$ and the Plancherel measure on $\tilde{\mathcal{F}}$).

Suppose temporarily that $\tilde{\mathcal{F}}$ is sufficiently concentrated near some given $\pi \in A_0$ that $\langle g v_{\pi '}, v_{\pi '} \rangle \approx \langle g v_{\pi}, v_\pi \rangle$ for all $\pi ' \in \tilde{\mathcal{F}}$.  Then \eqref{eq:average-v-pi-x-y-heur-1} simplifies to
\begin{equation}\label{eq:average-v-pi-x-y-heur-2}
  \mathbb{E}_{\pi' \in \mathcal{F}}
  v_{\pi'}(x) \overline{v_{\pi'}(y)}
  \approx 
  \frac{1}{\vol(\Gamma \backslash G)}
  \sum_{\gamma \in \Gamma}
  \langle x v_\pi, \gamma y v_\pi \rangle.
\end{equation}
Assume that quantum ergodicity holds in the strong form
\begin{equation}\label{eq:QE-hyp}
  \langle g (v_\pi \Psi), v_\pi \Psi  \rangle
  \approx
  \frac{1}{\vol(\Gamma \backslash G)}
  \langle g v_\pi, v_\pi  \rangle
  \langle g \Psi, \Psi  \rangle,
\end{equation}
at least on average over $\pi$.  From \eqref{eq:average-v-pi-x-y-heur-2}, \eqref{eq:QE-hyp} and ``unfolding,'' we obtain
\begin{equation}\label{eq:average-v-pi-psi-v-pi-prime}
  \mathbb{E}_{\pi' \in \mathcal{F}}
  \left\lvert \langle v_\pi \Psi, v_{\pi '} \rangle
  \right\rvert^2
  \approx
  \frac{1}{\vol(\Gamma \backslash G)^2}
  \int_{g \in G}
  |\langle g v_\pi, v_\pi  \rangle|^2
  \langle g \Psi, \Psi  \rangle.
\end{equation}

We now relax our assumption that $\tilde{\mathcal{F}}$ be concentrated and consider fairly general families.  By the Weyl law, we expect
\begin{equation}\label{eqn:weyl-law-ip-cons}
  \sum_{\pi \in \mathcal{F}}
  |\langle g v_\pi, v_\pi  \rangle|^2
  \approx
  \vol(\Gamma \backslash G)
  \int_{\pi \in \tilde{\mathcal{F}}}
  |\langle g v_\pi, v_\pi  \rangle|^2
  =
  \vol(\Gamma \backslash G)
  \langle \Ad(g) f, f \rangle.
\end{equation}
Suppose that the heuristic \eqref{eq:key-QV-heur} holds.  We may then apply \eqref{eq:average-v-pi-psi-v-pi-prime} to the family $\{\pi ' : \pi ' \approx \pi \}$, substitute the result into \eqref{eq:key-QV-heur-2}, and appeal to \eqref{eqn:weyl-law-ip-cons}, giving the prediction
\begin{equation}\label{eqn:general-QV-prediction}
  \sum_{\pi \in \mathcal{F}}
  \left\lvert \langle v_\pi \Psi, v_{\pi} \rangle
  \right\rvert^2
  \approx
  \frac{1}{\vol(\Gamma \backslash G)}
  \int_{g \in G}
  \langle \Ad(g) f, f \rangle
  \langle g \Psi, \Psi  \rangle
\end{equation}
subject to the modifications indicated above in the case of the ``time-reversal symmetry'' $\pi = \overline{\pi}$ or the presence of nontrivial correspondences on $\Gamma \backslash G$.

We note that this argument applies to fairly general quotients $\Gamma \backslash G$.  This generality will be exploited below.

In the following, we specialize \eqref{eqn:general-QV-prediction} in three ways.

\subsubsection{Generic non-arithmetic lattices}\label{sec:35ac3e57ae}
First, we take $G := \PSL_2(\mathbb{R})$, $\Gamma \leq G$ a ``generic'' (i.e., trivial commensurator) non-arithmetic cocompact lattice, $v_\pi$ as in \S\ref{eq:mu-T-Psi-vs-mu-pi-Psi-simple}, and $\tilde{\mathcal{F}} = \{\pi : 0 < - \h^2 \lambda_\pi < 1\}$ (the relevant definitions apply equally well to $\PSL_2(\mathbb{R})$ as to $\PGL_2(\mathbb{R})$).  We may take $f$ essentially (i.e., up to the constant factor $\h^{1/2}$) as in \S\ref{sec:defn-f}, with $k$ approximating the characteristic function of the interval $(-1,0)$.  We assume that $\pi = \overline{\pi}$ and that $\Psi$ is invariant by $w = \begin{pmatrix} 0 & 1 \\ -1 & 0 \end{pmatrix} \in G$.  By the analogue of \eqref{eqn:req-I} for $\PSL_2(\mathbb{R})$, the prediction \eqref{eqn:general-QV-prediction} simplifies to
\begin{equation*}
  \h
  \sum_{
    \substack{
      0 < -\h^2 \lambda_\pi < 1
    }
  }
  | \mu_\pi(\Psi) |^2
  \approx 
  \frac{1}{2 \pi \vol(\Gamma \backslash G)}
  \int_{u \in \mathbb{R}}
  \langle \begin{pmatrix}
    e^{u/2} &  \\
    & e^{-u/2}
  \end{pmatrix} \Psi, \Psi \rangle \, d u,
\end{equation*}
which may be seen to agree with the prediction of \cite{MR848319}.

\subsubsection{Arithmetic lattices}\label{sec:35ac3e57af}
Second, we address the setting of Theorem \ref{thm:var-3}.  We focus for simplicity on the diagonal case $\Psi \in \sigma \in A_0$, and assume that $\Psi = \Psi^{\sym}$ as in \eqref{eq:weyl-inv}.  Our task is accomplished most directly by applying \eqref{sec:general-heur} to the adelic quotient $\mathbf{G}(F) \backslash \mathbf{G}(\mathbb{A})$ (notation as in \S\ref{sec:intro-setup}), which has the effect of incorporating the nontrivial correspondences on the real quotient.  We take for $f$ a tensor product $\otimes_p f_p$ over the places $p$ of $F$, where the local factor $f_\mathfrak{q}$ at the distinguished real place $\mathfrak{q}$ is as in the previous paragraph and the remaining $f_p$ are the normalized characteristic functions of compact open subgroups $J_p$ as in \S\ref{sec:intro-setup}.  The (absolutely divergent) integral on the right hand side of \eqref{eqn:general-QV-prediction} factors (formally) as a product $\prod_p I_p$ of local integrals over all places $p$ of $F$; at places other than $\mathfrak{q}$, the component of $\langle g \Psi, \Psi \rangle$ is the normalized bi-$J_p$-invariant matrix coefficient of $\sigma$ at $p$, as in \S\ref{sec:remov-arithm-weights-1}, while at $\mathfrak{q}$ we take the usual matrix coefficient.  The local integrals $I_p$ have been computed.  The local integral $I_\mathfrak{q}$ is given by \eqref{eq:I-f-after-simpl}.  For archimedean places $p$ other than the distinguished real place $\mathfrak{q}$, we have $I_p = 1$ by \eqref{eqn:local-rallis-ipf-integral-nonsplit-real}.  For finite primes $p \in S$, we have $I_p = 2$ by \eqref{eqn:local-rallis-ipf-integral-nonsplit-finite}.  For finite primes $p \notin S$, we have $I_p = L_p(\sigma,\tfrac{1}{2} ) (1 - \frac{\lambda_\sigma(p)}{|p|^{3/2} + |p|^{1/2}})$ by \S\ref{sec:remov-arithm-weights-1}.  Modulo identifying $\prod_{p \notin S} L_p(\sigma,\tfrac{1}{2} )$ with $L^{(S)}(\sigma,\tfrac{1}{2} )$, we derive from \eqref{eqn:general-QV-prediction} the prediction
\begin{equation*}
  \h
  \sum_{
    \substack{
      0 < -\h^2 \lambda_\pi < 1
    }
  }
  | \mu_\pi(\Psi) |^2
  \approx
  \frac{  c_\sigma ' }{2 \pi \vol(\Gamma \backslash G)}
  \int_{u \in \mathbb{R}}
  \langle \begin{pmatrix}
    e^{u/2} &  \\
    & e^{-u/2}
  \end{pmatrix} \Psi, \Psi \rangle \, d u,
\end{equation*}
\begin{equation*}
  c_\sigma '
  :=
  2^{\# S + 1}
  L^{(S)}(\sigma,\tfrac{1}{2} )
  \prod_{p \notin S}
  (1 - \frac{\lambda_\sigma(p)}{|p|^{3/2} +
    |p|^{1/2}}).
\end{equation*}
We verify readily that this prediction agrees with the unweighted variant of Theorem \ref{thm:var-3} discussed in \S\ref{sec:remov-arithm-weights-1}.

\subsubsection{Comparison with $L$-function heuristics}\label{sec:35ac3e57b1}

Thirdly, we verify that the predictions of \S\ref{sec:general-heur}, and hence likewise our main results, are (unsurprisingly) consistent with the triple product formula and standard heuristics for averages of families of $L$-functions.  We include this discussion not only as a further check of our calculations, but also because we feel that it offers an interesting semiclassical perspective on the triple product formula itself.

We continue to take $\Gamma \backslash G = \mathbf{G}(F) \backslash \mathbf{G}(\mathbb{A})$.  Equip $\mathbf{G}(\mathbb{A})$ with Tamagawa measure, so that $\vol(\Gamma \backslash G) = 2$, and factor the measure on $\mathbf{G}(\mathbb{A})$ over the places $v$ of $F$ in such a way that for $p \notin S$, the local measure at $G_p = \mathbf{G}(F_p)$ assigns volume one to a maximal compact subgroup.  The main result of \cite{MR2449948} then says that for $v_\pi$ and $\Psi$ unramified outside $S$,
\[
  \lvert \langle v_\pi \Psi, v_\pi \rangle \rvert^2 = \frac{1}{8} \ell^{(S)}(\pi,\sigma) w(\pi),
\]
where
\[
  \ell^{(S)}(\pi,\sigma) := \zeta_F^{(S)}(2)^2 \frac{L^{(S)}(\pi \otimes \overline{\pi } \otimes \sigma, \tfrac{1}{2})}{ L^{(S)}(\ad \pi,1)^2 L^{(S)}(\ad \sigma,1)}
\]
and $w(\pi) := \prod_{v \in S} \int_{g \in \prod_{v \in S} G_v} \left\lvert \langle g v_\pi, v_\pi \rangle \right\rvert^2 \langle g \Psi, \Psi \rangle$.  We assume (for simplicity, and without loss of generality) that our family $\tilde{\mathcal{F}}$ has been chosen sufficiently concentrated that the weight $\pi \mapsto w(\pi)$ is essentially constant over its support.  By comparing with \eqref{eq:key-QV-heur} and \eqref{eq:average-v-pi-psi-v-pi-prime}, we see that our predictions translate to
\begin{equation*}
  \mathbb{E}_{\pi \in \mathcal{F}}
  \ell^{(S)}(\pi,\sigma)
  \approx
  \frac{8}{\vol(\Gamma \backslash G)^2}
  \prod_{p \notin S} I_p,
\end{equation*}
where $I_p$ is as above.  (As before, the product diverges and is to be understood formally; in particular, it hides the factor $L^{(S)}(\sigma,\tfrac{1}{2} )$.)  We may spectrally expand $I_p$ as the integral over unramified $\pi_p \in G_p^\wedge$ of the integral $\int_{g \in G_p} \Xi_{\pi_p}(g)^2 \Xi_{\sigma_p}(g)$ of normalized unramified matrix coefficients.  Ichino--Ikeda \cite[Theorem 1.2]{MR2585578} have shown that the latter integral evaluates to the local Euler factor $\ell_p(\pi_p,\sigma_p)$ for $\ell^{(S)}(\pi,\sigma)$, so that in fact
\begin{equation}\label{eq:I-p-rewritten-spectral-integral}
  I_p =  \int_{\text{unramified }\pi_p \in G_p^\wedge}
  \ell_p(\pi_p,\sigma_p).
\end{equation}
We may factor $L(\pi \otimes \overline{\pi } \otimes \sigma, \tfrac{1}{2}) = L(\ad \pi \otimes \sigma, \tfrac{1}{2}) L(\sigma, \tfrac{1}{2})$.  The family $\pi \mapsto L(\ad \pi \otimes \sigma,\tfrac{1}{2} )$ is self-dual with positive root numbers (assuming $\sigma$ even) and orthogonal symmetry type, so random matrix heuristics (see, e.g., \cite[\S1.2]{nelson-venkatesh-1}) predict that
\[
  \mathbb{E}_{\pi \in \mathcal{F}} \ell^{(S)}(\pi,\sigma) \approx 2 \prod_{p \notin S} I_p,
\]
with $I_p$ as given by \eqref{eq:I-p-rewritten-spectral-integral}.  Since $8/\vol(\Gamma \backslash G)^2 = 2$, those heuristics are consistent with ours.

\subsection*{Acknowledgements}
We would like to thank Peter Sarnak for encouragement and Raphael Steiner for many helpful comments and corrections on an earlier draft.  We gratefully acknowledge the support of NSF grant OISE-1064866 and SNF grant SNF-137488 during the work leading to this paper, and of a research grant (VIL54509) from VILLUM FONDEN during the paper's revision.

\newcommand{\etalchar}[1]{$^{#1}$}
\def\cprime{$'$} \def\cprime{$'$} \def\cprime{$'$} \def\cprime{$'$}


\end{document}